\documentclass[a4paper, 12pt, twoside]{article}

\usepackage[top=3cm, bottom=3cm, left=3cm, right=3cm]{geometry}
\usepackage[pdftex]{graphicx}
\usepackage{tikz}
\usetikzlibrary{shapes, arrows, arrows.meta}
\usepackage{array}
\usepackage{amsfonts,amstext,amscd,bezier,amssymb}
\usepackage[centertags]{amsmath}
\usepackage{amsthm}
\usepackage{mathtools}
\usepackage[integrals]{wasysym}
\usepackage{stmaryrd}
\usepackage{bbm}
\usepackage[mathcal]{euscript}
\usepackage[utf8]{inputenc}
\usepackage[T1]{fontenc}
\usepackage[english]{babel}
\usepackage{cite}
\usepackage{mathptmx}
\usepackage[inline]{enumitem}
\usepackage{calc}
\usepackage{anyfontsize}

\usepackage[hidelinks, bookmarks, bookmarksnumbered, pdfstartview={XYZ null null 1.00}]{hyperref}

\theoremstyle{plain}
\newtheorem{theo}{Theorem}[section]
\newtheorem{lemm}[theo]{Lemma}
\newtheorem{coro}[theo]{Corollary}
\newtheorem{prop}[theo]{Proposition}
\theoremstyle{definition}
\newtheorem{defi}[theo]{Definition}
\newtheorem{remk}[theo]{Remark}
\newtheorem{expl}[theo]{Example}

\DeclareMathOperator{\spr}{\rho} 
\DeclareMathOperator{\co}{co} 

\newcommand{\id}{\mathrm{Id}} 
\newcommand*\dd{\mathop{}\!\mathrm{d}}
\newcommand{\na}{\nabla}

\DeclarePairedDelimiter\abs{\lvert}{\rvert}
\DeclarePairedDelimiter\norm{\lVert}{\rVert}

\newcommand{\po}{\left(}
\newcommand{\pf}{\right)}
\newcommand{\cco}{\llbracket}
\newcommand{\ccf}{\rrbracket}
\newcommand{\R}{\mathbb R}
\newcommand{\N}{\mathbb N}  

\newcommand{\1}{\mathbbm{1}}
\newcommand{\Ebz}{\overline{\mathcal E}_z^{h+1}}

\numberwithin{figure}{section}
\numberwithin{equation}{section}

\let\originalleft\left
\let\originalright\right
\renewcommand{\left}{\mathopen{}\mathclose\bgroup\originalleft}
\renewcommand{\right}{\aftergroup\egroup\originalright}

\begin{document}

\setlist[enumerate, 1]{label={\textnormal{(\alph*)}}, ref={(\alph*)}, leftmargin=0pt, itemindent=*}
\setlist[enumerate, 2]{label={\textnormal{(\roman*)}}, ref={(\roman*)}}

\title{On the gap between deterministic and probabilistic Lyapunov exponents for continuous-time linear systems}
\author{Yacine Chitour\footnote{Université Paris-Saclay, CNRS, CentraleSupélec, Laboratoire des signaux et systèmes, 91190, Gif-sur-Yvette, France.} \and Guilherme Mazanti\footnote{Université Paris-Saclay, CNRS, CentraleSupélec, Inria, Laboratoire des signaux et systèmes, 91190, Gif-sur-Yvette, France.} \and Pierre Monmarch\'e\footnote{Sorbonne Université, CNRS, Laboratoire Jacques-Louis Lions (LJLL), Laboratoire de Chi\-mie Thé\-o\-rique (LCT), F-75005 Paris, France} \and Mario Sigalotti\footnote{Sorbonne Université, Inria, CNRS, Laboratoire Jacques-Louis Lions (LJLL), F-75005 Paris, France}}

\maketitle

\begin{abstract}
Consider a non-autonomous continuous-time linear system in which the time-dependent matrix determining the dynamics is piecewise constant and takes finitely many values $A_1, \dotsc, A_N$. This paper studies the equality cases between the maximal Lyapunov exponent associated with the set of matrices $\{A_1, \dotsc, A_N\}$, on the one hand, and the corresponding ones for piecewise deterministic Markov processes with modes $A_1, \dotsc, A_N$, on the other hand. A fundamental step in this study consists in establishing a result of independent interest, namely, that any sequence of Markov processes associated with the matrices $A_1,\dotsc, A_N$ converges, up to extracting a subsequence, to a Markov process associated with a suitable convex combination of those matrices.
\end{abstract}



\tableofcontents

\section{Introduction}

In this paper, we consider the family of non-autonomous continuous-time linear systems
\begin{equation}
\label{eq:switch}
\dot x(t) = A_{\sigma(t)} x(t),
\end{equation}
where $x(\cdot)$ takes values in $\mathbb R^d$, $\sigma(\cdot)$ is piecewise constant and takes values in the finite set of indices $\{1, \dotsc, N\}$, and we set $A$ to be the $N$-tuple made of 
$d \times d$ matrices with real coefficients 
 $A_1, \dotsc, A_N$ (also called the  \emph{modes} of \eqref{eq:switch}). Each \emph{signal} $\sigma$ corresponds to a possible evolution in time of a discrete parameter affecting the dynamics. This class of systems can be used to describe phenomena where different dynamical modes operate and the order in which they are active is not precisely known. In the engineering literature, such systems and their discrete-time counterparts bear the name of \emph{switched systems} \cite{Liberzon2003Switching} and they have been widely studied in the mathematical community since \cite{Daubechies1992Sets}.

One of the major issues regarding these systems concerns their asymptotic stability, uniformly with respect to 
the signal $\sigma$. Indeed, the fact that each individual mode is asymptotically stable does not imply that the trajectories of the corresponding switched system converge to $0$: it is easy to find two positive times $t_1, t_2$ and two matrices $A_1, A_2$ whose eigenvalues have negative real part such that the spectral radius of $e^{A_1 t_1} e^{A_2 t_2}$ is larger than $1$, as illustrated, for instance, in \cite{Liberzon2003Switching}. The measure of stability of a switching system with respect to \emph{all} possible signals $\sigma$ is characterized by its \emph{deterministic maximal Lyapunov exponent} $\lambda_{\mathrm d}(A)$, measuring the maximal asymptotic exponential rate of \eqref{eq:switch} (see \eqref{eq:defLambdaD} below). 

A difficulty in dealing with $\lambda_{\mathrm d}(A)$ is that the characterization of the asymptotic behavior of \eqref{eq:switch} based on its value is in general conservative from a practical viewpoint, since it corresponds to a maximization with respect to all possible signals. In the (few) cases where maximizing signals (or maximizing sequences of signals) are known, they happen to have a very specific structure, for instance switching between modes at precise times or at a fast rate \cite{Balde2009Note}. In many situations, one disposes of additional information on the signal $\sigma$ implying that such specific structures occur rarely, at least in a probabilistic sense. This motivates addressing the measure of stability of switched systems within a probabilistic framework, a question which has been considered in the literature, for instance in \cite{Benaim2014Stability} for systems in dimension $d = 2$. One may naturally expect that, except for some very particular situations, such a probabilistic framework  gives rise to less conservative measures of stability of \eqref{eq:switch}, as it is indeed observed for the two-dimensional systems considered in \cite{Benaim2014Stability}.

An important class of switched systems with random switching is that of piecewise deterministic Markov processes (PDMPs) introduced in \cite{Davis1984Piecewise}, which provide a natural modeling framework for phenomena of random switching without memory and corresponds to considering $\sigma$ as a continuous-time Markov process. In that case, the asymptotic behavior of \eqref{eq:switch} can be studied through a \emph{probabilistic Lyapunov exponent} $\lambda_{\mathrm p}(\nu, \mu, P, A)$, which consists of the expected value with respect to $\sigma$ of the asymptotic exponential rate of \eqref{eq:switch}, where $(\nu, \mu, P)$ are the parameters of the Markov process as described in Section~\ref{SecLambdaP} below. Recall that, due to the classical result in \cite{Furstenberg1960Products}, under generic assumptions on the Markov process, the asymptotic exponential rate for a given $\sigma$ is equal to $\lambda_{\mathrm p}(\nu, \mu, P, A)$ almost surely.

In applications, the quantity $\lambda_{\mathrm p}(\nu, \mu, P, A)$ is a suitable measure of asymptotic behavior if the parameters $(\nu, \mu, P)$ of the Markov process are fixed. However, it is also important to consider situations in which such parameters are not known exactly, and, for that purpose, we introduce the quantity $\lambda_{\mathrm p}^{\sup}(A)$ as the supremum of $\lambda_{\mathrm p}(\nu, \mu, P, A)$ with respect to all Markov processes $(\nu, \mu, P)$, which corresponds to a worst-case scenario.

Clearly, for every Markov process $(\nu, \mu, P)$ and every $N$-tuple of matrices $A$, one has $\lambda_{\mathrm p}(\nu, \mu, P, A) \leqslant \lambda_{\mathrm d}(A)$, and hence one also has $\lambda_{\mathrm p}^{\sup}(A) \leqslant \lambda_{\mathrm d}(A)$. The goal of this paper is to investigate under which conditions on $A$ the probabilistic point of view is strictly less restrictive than the deterministic one, i.e., to characterize in terms of $A$ the strict inequality (or, equivalently, the equality) between $\lambda_{\mathrm p}(\nu, \mu, P, A)$ and $\lambda_{\mathrm d}(A)$ for a fixed $(\nu, \mu, P)$ and also between $\lambda_{\mathrm p}^{\sup}(A)$ and $\lambda_{\mathrm d}(A)$. Hence, this paper is the continuous-time counterpart of \cite{Chitour2021Gap}, in which similar issues have been addressed for discrete-time systems.

Our main results in that sense, Theorems~\ref{TheoPNotSC-ContinuousTime} and \ref{thm:lambda-p-conv}, show that probabilistic measures of stability of \eqref{eq:switch} are indeed less restrictive than the deterministic ones, except for some particular situations that we characterize. For instance, under additional irreducibility and strong connectedness assumptions, Theorem~\ref{TheoPNotSC-ContinuousTime} reduces to Proposition~\ref{PropFixedP-ContinuousTime}, which states that $\lambda_{\mathrm p}(\nu, \mu, P, A)$ is strictly smaller than $\lambda_{\mathrm d}(A)$ except for the particular situation in which the matrices $A_1, \dotsc, A_N$ are skew-symmetric up to a common translation by a multiple of the identity matrix and a common change of basis. Up to a technical assumption which is always satisfied in dimension $d \leqslant 3$, Theorem~\ref{thm:lambda-p-conv} states that the probabilistic measure of stability $\lambda_{\mathrm p}^{\sup}(A)$ is strictly smaller than $\lambda_{\mathrm d}(A)$ except for the particular situation where the worst possible deterministic behavior is attained by a matrix in the convex hull of $A_1, \dotsc, A_N$. In particular, these results confirm that probabilistic measures of stability are most often less conservative than their deterministic counterparts, as previously observed in \cite{Benaim2014Stability} for two-dimensional systems.

The characterization of equality between $\lambda_{\mathrm p}(\nu, \mu, P, A)$ and $\lambda_{\mathrm d}(A)$ follows essentially the same lines as the corresponding problem in discrete time addressed in \cite{Chitour2021Gap}. We first consider the case where the $N$-tuple $A$ is irreducible and the matrix $P$ is strongly connected (as defined in Section~\ref{SecDefinitions}): it is shown that equality occurs if and only if the matrices $A_i - \lambda_{\mathrm d}(A) \id$ are skew-symmetric (up to a common change of basis), a result that relies on the use of an extremal norm for $A$ (Definition~\ref{DefiExtremal}) and the characterization of semigroups with constant spectral radius from \cite{Protasov2017Matrix}. We then treat the general case by decomposing $A$ into irreducible blocks and $P$ into strongly connected blocks.

As regards the question of equality between $\lambda_{\mathrm p}^{\sup}(A)$ and $\lambda_{\mathrm d}(A)$, we show that it implies that $\lambda_{\mathrm d}(A)$ is equal to the maximum of the real part of the eigenvalues of some matrix $M$ belonging to the convex hull of $A$, and that the converse is true under a technical assumption on $A = (A_1, \dotsc, A_N)$ and $M$, cf.\ Definition~\ref{defi:cond-C}, which is always satisfied in dimension $d \leqslant 3$. We conjecture that this technical assumption is not necessary to get the converse implication. The difficulty in removing the technical assumption consists in proving the convergence of the probabilistic Lyapunov exponent when the jump rate goes to $+\infty$.

Indeed, our analysis relies on the investigation of the behavior of maximizing sequences of Markov processes for $\lambda_{\mathrm p}^{\sup}(A)$. In the discrete-time setting considered in \cite{Chitour2021Gap}, the issue is easily handled thanks to the compactness of the space of discrete-time Markov processes. This is not anymore the case in the continuous-time setting, where the situation is much more delicate since switching between modes can occur arbitrarily fast. When the transition matrix $P$ is fixed and strongly connected and the jump rate $\mu$ goes to infinity, it is well-known that high-frequency jumps lead to deterministic averaging \cite{benaim2019, Benaim2014Stability}, namely the Markov process converges to a deterministic motion $\dot x = M x$, where $M$ belongs to the convex hull of $A$. More generally (if $P$ is not fixed as $\mu\rightarrow +\infty$), in case of a combination of fast and slow jumps, one can expect the convergence towards a Markov process on convex combinations of matrices of $A$.

We rigorously handle such a decomposition of Markov processes in different ti\-me\-sca\-les by relying on results from various works by Landim and collaborators, in particular \cite{Landim2016}, although these works are primarily interested in metastability phenomena, i.e., Markov chains for which the different timescales are all slow, instead of fast as in our case. As a consequence, we prove that we can extract from any sequence of Markov processes with modes in $A$ a subsequence that converges in law to a Markov process associated with a suitable convex combinations of the original matrices in $A$. This compactification result, Theorem~\ref{Thm:CV-mu-pas-borne}, is one of our main results.

Since the convergence in law obtained in Theorem~\ref{Thm:CV-mu-pas-borne} is not uniform in time, it is not sufficient to deduce convergence of the Lyapunov exponents of the sequence of Markov processes to the Lyapunov exponent of the limit process. Such a convergence property is interesting in itself and has already been addressed in particular cases, e.g.\ \cite[Section~2.5]{Benaim2014Stability} (where $N = d = 2$ and the matrices are Hurwitz) or \cite[Corollary~2.15]{benaim2019} (where the matrix $M$ appearing in the limit is  Metzler and strongly connected). We prove in Proposition~\ref{prop:converge} a result in that direction under the already mentioned additional technical assumption on $A$ and $M$. Establishing such a result unconditionally would provide a complete characterization of the equality between $\lambda_{\mathrm d}(A)$ and $\lambda_{\mathrm p}^{\sup}(A)$.  

As a conclusion, our results show that the equality between the deterministic maximal Lyapunov exponent and the probabilistic ones (either for a fixed Markov chain or for the worst probabilistic case) only occurs in very specific cases. This shows that, in most cases, working within a probabilistic framework yields a less conservative estimate, closer to the stability properties most commonly observed in practice.

The paper is organized as follows. Section~\ref{SecDefinitions} collects definitions, notations, and basic facts relative to deterministic and probabilistic Lyapunov exponents as well as continuous-time Markov processes. The statements of the main results proved in the paper are presented in Section~\ref{sec:main-results}. Section~\ref{SecCharactFixedMarkov} addresses the characterization of equality between $\lambda_{\mathrm p}(\nu, \mu, P, A)$ and $\lambda_{\mathrm d}(A)$. We describe in Section~\ref{sec:compactification} the compactification of the space of Markov processes, which is used in Section~\ref{SecCharactSupMarkov} to study the case of equality between $\lambda_{\mathrm p}^{\sup}(A)$ and $\lambda_{\mathrm d}(A)$. The paper is completed by two appendices. In Appendix~\ref{Sec:MarkovDecomposition}, we prove a general result of decomposition into different timescales for a sequence of Markov chains on a finite state space, adapted from \cite{Landim2016}, which is a central tool in the proofs of Section~\ref{sec:compactification}. Appendix~\ref{app:B} provides the proofs of some linear-algebraic technical results used in Section~\ref{SecCharactSupMarkov}.

\section{Definitions, notations, and basic facts}
\label{SecDefinitions}

Throughout the paper, $d$ and $N$ belong to $\mathbb N$, which is used to denote the set of positive integers. If $a$ and $b$ are integers, $\llbracket a, b \rrbracket$ denotes the set of integers $j$ such that $a \leqslant j \leqslant b$. We use $\abs{\cdot}$ to denote a norm in $\mathbb R^d$ and $\norm{\cdot}$ to denote the corresponding induced norm on the space $\mathcal M_d(\mathbb R)$ of $d \times d$ matrices with real coefficients. The identity matrix in $\mathcal M_d(\mathbb R)$ is denoted by $\id$. The spectral radius of a square matrix $M$ is denoted by $\spr(M)$, and its spectral abscissa, defined as the maximum of the real parts of its eigenvalues, is denoted by $\lambda(M)$. An $N$-tuple $A = (A_1, \dotsc, A_N) \in \mathcal M_d(\mathbb R)^N$ is said to be \emph{irreducible} if the only invariant subspaces by all $A_i$ are $\{0\}$ and $\mathbb R^d$. If $\mathcal A \subset \mathcal M_d(\mathbb R)$, we use $\co(\mathcal A)$ to denote the convex hull of $\mathcal A$.

Denote by $\Sigma$ the set of all piecewise constant right-continuous functions defined on $[0,\infty)$ and taking values in $\llbracket 1, N\rrbracket$. Given $\sigma\in\Sigma$, we use $t \mapsto \Phi_\sigma(t)$ to denote the flow of
\[\dot x = A_{\sigma(t)} x\]
with $\Phi_\sigma(0) = \id$. In particular,
\[\Phi_\sigma(t_n) = e^{A_{\sigma(t_{n-1})} (t_n-t_{n-1})} \dotsm e^{A_{\sigma(0)} t_1},\]
where $(t_i)_{i \in \mathbb N}$ is an increasing sequence containing all discontinuity times of $\sigma$.

\subsection{Deterministic Lyapunov exponent}

Let $A = (A_1, \dotsc, A_N) \in \mathcal M_d(\mathbb R)^N$. The \emph{deterministic Lyapunov exponent} $\lambda_{\mathrm d}(A)$ associated with $A$ is defined as
\begin{equation}
\label{eq:defLambdaD}
\lambda_{\mathrm d}(A) =\limsup_{t \to \infty} \frac{1}{t}  \sup_{\sigma\in \Sigma} \log\norm*{\Phi_\sigma(t)}.
\end{equation}
Since all norms in $\mathbb R^d$ are equivalent, it immediately follows that $\lambda_{\mathrm d}(A)$ does not depend on the specific choice of $\norm{\cdot}$. It turns out (see, e.g., \cite[Lemma~1.2]{Jungers2009Joint}) that, since $\norm{\cdot}$ is submultiplicative on $\mathcal M_d(\mathbb R)$, one has
\begin{equation}
\label{eq:LyapDeterm}
\lambda_{\mathrm d}(A) = \lim_{t \to \infty} \frac{1}{t}\sup_{\sigma\in \Sigma} \log\norm*{\Phi_\sigma(t)} = \inf_{t > 0} \frac{1}{t}\sup_{\sigma\in \Sigma} \log\norm*{\Phi_\sigma(t)}.
\end{equation}
Moreover, for every $\sigma \in \Sigma$ and $t > 0$, one has
\begin{equation}
\label{LambdaDGeqLambdaWord}
\frac{1}{t}\log\spr(\Phi_\sigma(t)) \leqslant \lambda_{\mathrm d}(A).
\end{equation}
Indeed, let $\hat \sigma \in \Sigma$ be the $t$-periodic signal coinciding with $\sigma$ on the interval $[0, t)$. Then, for every $k \in \mathbb N$,
\begin{align*}
\frac{1}{t}\log\spr(\Phi_\sigma(t)) & = \frac{1}{k t}\log\spr(\Phi_{\hat\sigma}(k t))  \leqslant \frac{1}{k t}\log\norm*{\Phi_{\hat\sigma}(k t)} \displaybreak[0] \\ 
& \leqslant \frac{1}{k t} \sup_{\varsigma \in \Sigma} \log\norm*{\Phi_{\varsigma}(k t)},
\end{align*}
and we conclude from \eqref{eq:defLambdaD} by taking the $\limsup$ as $k \to +\infty$.

\begin{defi}[Extremal norm]
\label{DefiExtremal}
Let $A = (A_1, \dotsc, A_N) \in \mathcal M_d(\mathbb R)^N$. A norm $\norm{\cdot}_{\mathrm e}$ in $\mathcal M_d(\mathbb R)$ is said to be 
\emph{extremal for $A$} if, for every $\sigma \in \Sigma$ and $t\geqslant 0$, it holds $\norm{\Phi_\sigma(t)}_{\mathrm e} \leqslant e^{\lambda_{\mathrm d}(A) t}$.
\end{defi}

\begin{remk}
\label{RemkNondefective}
A necessary and sufficient condition for the existence of an extremal norm for a given $A = (A_1, \dotsc, A_N)$ is the \emph{nondefectiveness} of $A$, i.e., the existence of $C > 0$ such that $\norm{\Phi_\sigma(t)} \leqslant C e^{\lambda_{\mathrm d}(A) t}$ for every $t \geqslant 0$ and $\sigma \in \Sigma$ (see, e.g., \cite[Theorem~2.2]{Jungers2009Joint} for the discrete-time case, which extends readily to the continuous-time setting).

Note that, since the computation of $\lambda_{\mathrm d}(A)$ is intractable in general (cf.\ \cite{Jungers2009Joint}), nondefectiveness turns out to be also difficult to check. This motivates the search for simpler conditions implying the nondefectiveness of a family of matrices $A$. One such condition is the irreducibility of $A$ (see, e.g., \cite{Wirth2002Generalized}, where it shown that irreducibility actually implies the existence of a so-called Barabanov norm, which is
an extremal norm satisfying some additional properties).
\end{remk}

\subsection{Continuous-time Markov processes}

In this paper, we consider continuous-time Markov processes in $\llbracket 1, N\rrbracket$ defined by triples $(\nu, \mu, P)$, where $P = (p_{ij})_{i, j = 1}^N \in \mathcal M_N(\mathbb R)$ is a stochastic matrix, $\mu > 0$, and $\nu \in \mathbb R^N$ is a probability vector, seen as a row vector, i.e., as a $1 \times N$ matrix. The Markov process corresponding to $(\nu, \mu, P)$, denoted by $\sigma$, is the continuous-time Markov chain on $\llbracket 1, N\rrbracket$ with initial law $\nu$, transition matrix $P$, and jump rate $\mu$. Hence, if $\sigma(t) = i \in \llbracket 1, N\rrbracket$ and $t^\prime$ is the next jump time, then $t^\prime - t$ follows an exponential law of parameter $\mu$ and $\sigma(t^\prime) = j \in \llbracket 1, N\rrbracket$ with probability $p_{ij}$. Note that trivial jumps (i.e., from a state to itself) are allowed and that, for every $\alpha \in(0,1]$, both triples of parameters $(\nu, \mu/\alpha, \id + \alpha(P - \id))$ and $(\nu, \mu, P)$ determine Markov processes with the same law.

When $\sigma(\cdot)$ and $\nu$ are, respectively, a Markov chain and a probability vector on $\cco 1,N\ccf$, we occasionally denote by $\mathbb P_{\nu}$ and $\mathbb E_{\nu}$ probabilities and expectations to indicate that the law of the initial condition $\sigma(0)$ is $\nu$. If $\nu=\delta_i$ for some $i\in\cco1,N\ccf$, we simply write $\mathbb P_i$ and $\mathbb E_i$.

Given $A = (A_1, \dotsc, A_N) \in \mathcal M_d(\mathbb R)^N$ and $x_0 \in \mathbb R^d$, the above Markov process $\sigma$ in $\llbracket 1, N\rrbracket$ induces the stochastic processes $A_{\sigma(\cdot)}$ in $\mathcal M_d(\mathbb R)$ and  $\Phi_{\sigma}(\cdot) x_0$ in $\mathbb R^d$. When clear from the context, we still identify such processes with the triple $(\nu, \mu, P)$. The matrices $A_1, \dotsc, A_N$ are called the \emph{modes} of the Markov process and, for distinct $i, j \in \llbracket 1, N\rrbracket$, $\lambda(i, j) = \mu p_{ij}$ is the \emph{jump rate from $i$ to $j$}. Notice that, although $x(\cdot)=\Phi_{\sigma}(\cdot) x_0$ is not a Markov process by itself, this is the case for $(x(\cdot),\sigma(\cdot))$, which is a PDMP. However, with a slight abuse of language, we will sometimes refer to $x(\cdot)$ as a Markov process for $A$.

We say that a stochastic matrix $P$ is \emph{strongly connected} if it is not similar via a permutation to a nontrivial block upper triangular matrix, i.e., if its associated directed graph is strongly connected. (Such a matrix is usually called \emph{irreducible}, but we already use the latter term in its linear algebraic meaning.) More generally, every stochastic matrix $P \in \mathcal M_N(\mathbb R)$ admits, up to a permutation in the set of indices $\llbracket 1, N\rrbracket$, the decomposition into strongly connected blocks (see, e.g., \cite{Seneta2006Nonnegative}) given by
\begin{equation}
\label{DecomposeP}
P = \begin{pmatrix}
P_{1} & 0 & \cdots 
& \cdots & 0 \\
0& \ddots & \ddots 
& &\vdots \\
\vdots & \ddots & \ddots&\ddots&\vdots\\
0& \cdots & 0& P_R 
& 0 \\
* & \cdots &\cdots & * & Q\\
\end{pmatrix},
\end{equation}
where $\spr(Q) < 1$ and, for $i \in \llbracket 1, R\rrbracket$, $P_i \in \mathcal M_{n_i}(\mathbb R)$ is a stochastic and strongly connected matrix for some positive integers $R, n_1, \dotsc, n_R$. For every $i \in \llbracket 1, R\rrbracket$, we define the recurrence class $\mathcal I(i)$ for $P$ by
\begin{equation}
\label{eq:defi-Ii}
\mathcal I(i) = \llbracket n_1 + \dotsb + n_{i-1} + 1, n_1 + \dotsb + n_i\rrbracket,
\end{equation}
and the set of transient states by $\mathcal T=\cco n_1+\dots+n_R+1,N\ccf$ (possibly empty). Given a probability vector $\nu \in \mathbb R^N$, a recurrence class $\mathcal I(i)$ is said to be \emph{accessible from $\nu$} if $\mathbb P_{\nu}(\sigma(\tau) \in \mathcal I(i))>0$, where $\tau = \inf\{t\geqslant 0 \mid \sigma(t)\notin \mathcal T\}$.

A probability vector $\nu\in\R^N$ is said to be \emph{invariant for $P$} if $\nu P = \nu$. When $P$ is strongly connected, such an invariant law exists and is unique. More generally, considering a decomposition in strongly connected blocks \eqref{DecomposeP}, we  let $\nu^{[i]}$ be the unique invariant probability vector for $P_i$, which is canonically extended (adding zero components) to a vector in $\mathbb R^N$ still denoted by $\nu^{[i]}$. Then every invariant probability vector $\nu \in \mathbb R^N$ of $P$ can be uniquely decomposed as
\begin{equation}
\label{DecomposeNu}
\nu = \sum_{i=1}^R \alpha_i \nu^{[i]},
\end{equation}
where $\alpha_1, \dotsc, \alpha_R \in [0, 1]$ and $\sum_{i=1}^R \alpha_i = 1$. This means that the corresponding Markov process $\sigma$ starts in the recurrence class $\mathcal I(i)$ with probability $\alpha_i$, for $i \in \llbracket 1, R\rrbracket$, and remains there for all positive times.

\subsection{Probabilistic Lyapunov exponent}
\label{SecLambdaP}

Let $(\nu, \mu, P)$ define a Markov process in $\Sigma$ and $A \in \mathcal M_d(\mathbb R)^N$. The \emph{probabilistic Lyapunov exponent} for $(\nu, \mu, P, A)$ is the quantity defined by
\begin{equation}
\label{eq:defLambdaP}
\lambda_{\mathrm p}(\nu, \mu, P, A) = \limsup_{t \to +\infty} \frac{1}{t} \mathbb E[\log\norm{\Phi_\sigma(t)}],
\end{equation}
where $\mathbb E$ denotes the expectation with respect to $\sigma$ distributed according to $(\nu, \mu, P)$. As for $\lambda_{\mathrm d}(A)$, the above expression is independent on the choice of the specific induced norm $\norm{\cdot}$ in $\mathcal M_d(\mathbb R)$. 

Given $A \in \mathcal M_d(\mathbb R)^N$, we denote by $\lambda_{\mathrm p}^{\sup}(A)$ the supremum of $\lambda_{\mathrm p}(\nu, \mu, P, A)$ over all parameters $(\nu, \mu, P)$. 
For every Markov process $(\nu, \mu, P)$, we have
\begin{equation}
\label{LambdaPLeqLambdaD-Continuous}
\lambda_{\mathrm p}(\nu, \mu, P, A) \leqslant \lambda_{\mathrm p}^{\sup}(A) \leqslant \lambda_{\mathrm d}(A),
\end{equation}
since, for fixed $t > 0$,
\[\mathbb E[\log\norm{\Phi_\sigma(t)}] \leqslant \sup_{\sigma' \in \Sigma} \log\norm{\Phi_{\sigma'}(t)}.\]

Provided that $\nu$ is invariant for $P$,  by classical subadditivity arguments on the function $t \mapsto \mathbb E[\log\norm{\Phi_\sigma(t)}]$, one has
\begin{equation}
\label{eq:LyapProba}
\lambda_{\mathrm p}(\nu, \mu, P, A) = \lim_{t \to +\infty} \frac{1}{t} \mathbb E[\log\norm{\Phi_\sigma(t)}] = \inf_{t > 0} \frac{1}{t} \mathbb E[\log\norm{\Phi_\sigma(t)}].
\end{equation}
In fact, the initial condition $\nu$ is not very important, due to the Markov property and the ergodic behaviour within each recurrent class. More precisely, we can state the following.

\begin{prop}
\label{prop:condition-initiale}
Consider the decomposition \eqref{DecomposeP} and let $\tau = \inf\{t\geqslant 0 \mid \sigma(t)\notin \mathcal T\}$ be the first time at which $\sigma$ reaches a recurrent class. Then, for each probability vector $\nu\in\R^N$,
\[\lambda_{\mathrm p}(\nu,\mu,P,A) \ = \ \sum_{i=1}^R \mathbb P_{\nu} \po \sigma(\tau) \in \mathcal I(i)\pf \lambda_{\mathrm p}(\nu^{[i]},\mu,P,A)\,.\]
\end{prop}

\begin{proof}
Let $K=\max_{i\in\cco 1,N \ccf}\|A_i\|$. Then, for all $\sigma\in\Sigma$ and all $t,s\geqslant 0$,
\begin{equation}\label{eq:phi}
   e^{-Kt}\|\Phi_{\sigma}(s)\|  \ \leqslant \ \|\Phi_{\sigma}(t+s)\| \ \leqslant \ e^{Kt}\|\Phi_{\sigma}(s)\| \,. 
\end{equation}
Fix a probability vector $\nu\in\R^N$ and let  $\alpha_i = \mathbb P_{\nu} \po \sigma(\tau) \in \mathcal I(i)\pf$ for $i\in\cco 1,R\ccf$ and $\tilde \nu = \sum_{i=1}^R \alpha_i \nu^{[i]}$. By standard arguments on Markov chains, denoting $\nu(t) = \nu e^{t \mu (P-I)}$ the law at time $t$ of a chain associated with $P$ with initial condition $\nu$, then $\nu(t)$ converges to $\tilde \nu$ as $t\rightarrow +\infty$. For an arbitrary $\varepsilon>0$, let $t_0$ be such that the total variation norm of $\nu(t_0)-\tilde \nu$ is less than  $\varepsilon$. It means that there exist random variables $\sigma_0,\sigma'_0$ on $\cco 1,N\ccf$ respectively distributed according to $\nu(t_0)$ and $\tilde \nu$ such that $\mathbb P(\sigma_0 \neq \sigma_0') \leqslant \varepsilon/2$. Considering two chains $\sigma$ and $\sigma'$ with respective initial conditions $\sigma_0$ and $\sigma_0'$ and such that, conditionally to $\{\sigma_0=\sigma_0'\}$, $\{\sigma(t)=\sigma'(t)\ \forall t\geqslant 0\}$, we get that, for all $t\geqslant t_0$, 
\[\left|  \mathbb E_{\nu(t_0)} \po \log\|\Phi_{\sigma}(t-t_0)\|\pf \ - \ \mathbb E_{\tilde \nu} \po \log\|\Phi_{\sigma}(t-t_0)\|\pf\right| \  \leqslant \ \varepsilon K(t-t_0)\,.\] 
From \eqref{eq:phi},
\[\left| \log \|\Phi_{\sigma}(t)\| - \log \|\Phi_{\sigma(t_0+\cdot)}(t-t_0)\| \right| \leqslant Kt_0\,.\]
Thus, using the Markov property, for all $t\geqslant t_0$,
\[\left|  \mathbb E_{\nu} \po \log\|\Phi_{\sigma}(t)\|\pf \ - \ \mathbb E_{\nu(t_0)} \po \log\|\Phi_{\sigma}(t-t_0)\|\pf\right| \  \leqslant \   Kt_0\,.\] 
Combining these two bounds, dividing by $t$, taking the $\limsup$ as $t\rightarrow +\infty$, and using that $\varepsilon$ is arbitrary, we get that 
\[\lambda_{\mathrm p}(\nu,\mu,P,A) \ = \ \lambda_{\mathrm p}(\tilde\nu,\mu,P,A)\,.\]
Besides, conditioning with respect to the recurrence class of the initial condition, we immediately get that 
\[\mathbb E_{\tilde \nu} \po \log\|\Phi_{\sigma}(t)\|\pf \ = \ \sum_{i=1}^R  \alpha_i \mathbb E_{\nu^{[i]}} \po \log\|\Phi_{\sigma}(t)\|\pf  \]
for all $t\geqslant 0$. Dividing by $t$ and letting $t\rightarrow +\infty$ yield the conclusion.
\end{proof}

{\sloppy

From Proposition~\ref{prop:condition-initiale}, we deduce that $\lambda_{\mathrm p}(\nu,\mu,P,A) \leqslant \max_{i\in\cco 1,R\ccf} \lambda_{\mathrm p}(\nu^{[i]},\mu,P,A)$. When $\sigma(0)$ is distributed according to $\nu^{[i]}$ for some $i\in\cco 1,R\ccf$, $\sigma$ stays for all times in the class $\mathcal I(i)$, where $P_i$ is strongly connected. Since we are interested in maximal values of the Lyapunov exponent, we are going to use several times in what follows Proposition~\ref{prop:condition-initiale} to reduce to the strongly connected case.

}

\section{Statements of the main results}
\label{sec:main-results}

Our main result concerning equality between $\lambda_{\mathrm p}(\nu, \mu, P, A)$ and $\lambda_{\mathrm d}(A)$ for a given Markov process $(\nu, \mu, P)$ and a given $A \in \mathcal M_d(\mathbb R)^N$ is the following.

\begin{theo}
\label{TheoPNotSC-ContinuousTime}
Let $P \in \mathcal M_{N}(\mathbb R)$ be a stochastic matrix, $\nu$ be a probability vector of $\R^N$, $\mu>0$, and $A = (A_1, \dotsc,\allowbreak A_N) \allowbreak \in \mathcal M_d(\mathbb R)^N$. Then the following statements are equivalent:
\begin{enumerate}
\item\label{TheoPNotSC-ContinuousTime-Equality} $\lambda_{\mathrm d}(A) = \lambda_{\mathrm p}(\nu, \mu, P, A)$.
\item\label{TheoPNotSC-ContinuousTime-Spr} For every  recurrent class $\mathcal I$ of $P$ which  is accessible from $\nu$, every $k\in \N$, $i_1,\dotsc,i_k \in \mathcal I$,  and $t_1, \dotsc, t_k \geqslant 0$, it holds
\[\spr(e^{A_{i_k} t_k} \dotsm e^{A_{i_1} t_1}) = e^{\lambda_{\mathrm d}(A) (t_1 + \dotsb + t_k)}.\]
\end{enumerate}
\end{theo}

The proof of Theorem~\ref{TheoPNotSC-ContinuousTime} is the main goal of Section~\ref{SecCharactFixedMarkov}.

The sequel of the paper is motivated by the problem of characterizing equality between $\lambda_{\mathrm d}(A)$ and $\lambda_{\mathrm p}^{\sup}(A)$. A first step in that direction is to understand the behavior of sequences $(\nu_n, \mu_n, P_n)_{n \in \mathbb N}$ of Markov processes for a given $A = (A_1, \dotsc, A_N)$. For that purpose, we introduce the following definition.

\begin{defi}
A \emph{convexified  Markov process} $(x,\sigma)$ for $A$ is a continuous-time Markov process 
with modes $B_1,\dots,B_k$, where $k\in\llbracket 1,N\rrbracket$, $B_j\in {\rm co}\{A_\ell\mid \ell\in I_j\}$
for  $j=1,\dots,k$, and  
$I_1,\dots,I_k$ are 
pairwise disjoint nonempty subsets of $\llbracket 1, N\rrbracket$. 

We also define the quantity
\begin{equation}\label{eq:lambdapconv}
\lambda_{\mathrm p}^{\mathrm{conv}}(A)=\sup_{(\nu, \mu, P, B)} \lambda_{\mathrm p}(\nu, \mu, P, B),
\end{equation}
where the supremum is taken among all convexified Markov processes $(\nu, \mu, P, B)$ for $A$.
\end{defi}

{\sloppy

Note that the quantity $\lambda_{\mathrm p}^{\mathrm{conv}}(A)$ introduced above satisfies
\begin{equation}\label{eq:triplet}
\lambda_{\mathrm p}^{\sup}(A)\leqslant \lambda_{\mathrm p}^{\mathrm{conv}}(A)\leqslant \lambda_{\mathrm d}(A),
\end{equation}
where the last inequality follows from the fact that, for every convexified Markov processes $(\nu, \mu, P, B)$ for $A$, we have by \eqref{LambdaPLeqLambdaD-Continuous} that $\lambda_{\mathrm p}(\nu, \mu, P, B)\leqslant \lambda_{\mathrm d}(B)$ and, in addition, $\lambda_{\mathrm d}(B)\leqslant \lambda_{\mathrm d}(A)$, the latter inequality being a consequence of the fact that $\lambda_{\mathrm d}(\widehat A) = \lambda_{\mathrm d}(\co(\widehat A))$ for every $\widehat A \in \mathcal M_d(\mathbb R)^N$ (see, e.g., \cite{Shorten2007Stability}).

}

Our main result concerning convexified Markov processes is that they compactify the space of Markov processes. More precisely, we prove the following theorem.

\begin{theo}\label{Thm:CV-mu-pas-borne}
Consider a sequence $(x_n,\sigma_n)$ of Markov processes for $A$ with parameters $(\nu_n,\mu_n,P_n)_{n\in\N}$. Up to extracting a subsequence, there exists a convexified Markov process $(x,\sigma)$ for $A$ such that, for all $T,\delta>0$,
\begin{eqnarray*}
\mathbb P \po \sup_{t\in [0,T]} |x(t)-x_n(t)| > \delta \pf & \underset{n\rightarrow+\infty}\longrightarrow & 0\,.
\end{eqnarray*}
\end{theo}

The proof of Theorem~\ref{Thm:CV-mu-pas-borne} can be found in Section~\ref{sec:compactification}.

Thanks to the compactification result from Theorem~\ref{Thm:CV-mu-pas-borne}, we are able to provide the following result on the relations between $\lambda_{\mathrm p}^{\sup}(A)$, $\lambda_{\mathrm p}^{\mathrm{conv}}(A)$, and $\lambda_{\mathrm d}(A)$. Recall that, given a square matrix $M$, $\lambda(M)$ denotes its spectral abscissa.

\begin{theo}
\label{thm:lambda-p-conv}
Let $A = (A_1, \dotsc, A_N) \in \mathcal M_d(\mathbb R)^N$.
\begin{enumerate}
\item\label{thm:main-iff-conv} The equality $\lambda_{\mathrm d}(A) = \lambda_{\mathrm p}^{\mathrm{conv}}(A)$ holds true if and only if there exists $M \in \co(A)$ such that $\lambda(M) = {\lambda_{\mathrm d}(A)}$.

\item\label{thm:main-a} If $\lambda_{\mathrm d}(A) = \lambda_{\mathrm p}^{\sup}(A)$ then there exists $M \in \co(A)$ such that $\lambda(M) = {\lambda_{\mathrm d}(A)}$.

\item\label{thm:main-b} Assume that there exists $M \in \co(A)$ such that $\lambda(M) = {\lambda_{\mathrm d}(A)}$ and that, for every $\varepsilon>0$, there exist $M^\varepsilon\in\co(A)$ and a sequence $(\nu_n, \mu_n, P_n)_{n\in\N}$ of Markov processes for $A$  such that
$\norm{M-M^\varepsilon} < \varepsilon$ and 
\begin{equation}\label{eq:condC}
\lambda(M^\varepsilon)=\lim_{n\to\infty}\lambda_{\mathrm p}(\nu_n, \mu_n, P_n,A).
\end{equation}
Then $\lambda_{\mathrm d}(A) = \lambda_{\mathrm p}^{\sup}(A)$.
\end{enumerate}
\end{theo}

We conjecture that the converse of Theorem~\ref{thm:lambda-p-conv} \ref{thm:main-a} is true. However, we are only able to prove it in low dimension, as stated in the following result. 

\begin{prop}\label{prop:dim23}
Let $d\leqslant 3$ and $A  \in \mathcal M_d(\mathbb R)^N$. 
 Then $\lambda_{\mathrm d}(A) = \lambda_{\mathrm p}^{\sup}(A)$ if and only if there exists 
 $M \in \co(A)$ such that $\lambda(M) = {\lambda_{\mathrm d}(A)}$.
\end{prop}

The proofs of Theorem~\ref{thm:lambda-p-conv} and Proposition~\ref{prop:dim23} are provided in Section~\ref{SecCharactSupMarkov}.

\section[Characterization of equality between \texorpdfstring{$\lambda_{\mathrm d}(A)$}{lambda d(A)} and \texorpdfstring{$\lambda_{\mathrm p}(\nu, \mu, P, A)$}{lambda p(nu, mu, P, A)}]{Characterization of equality between $\lambda_{\mathrm d}(A)$ and \newline $\lambda_{\mathrm p}(\nu, \mu, P, A)$}
\label{SecCharactFixedMarkov}

The goal of this section is to prove Theorem~\ref{TheoPNotSC-ContinuousTime}. We start with the particular situation in which $P$ is strongly connected and $A$ is irreducible.

\begin{prop}
\label{PropFixedP-ContinuousTime}
Let $P \in \mathcal M_{N}(\mathbb R)$ be a stochastic strongly connected matrix, $\nu$ a probability vector of $\mathbb R^N$, and  $\mu>0$. Let $A = (A_1, \dotsc,\allowbreak A_N) \allowbreak \in \mathcal M_d(\mathbb R)^N$ be irreducible and $\norm{\cdot}_{\mathrm e}$ be an extremal norm for $A$. Then the following statements are equivalent:
\begin{enumerate}
\item\label{PropFixedP-ContinuousTime-Equality} $\lambda_{\mathrm d}(A) = \lambda_{\mathrm p}(\nu, \mu, P, A)$.
\item\label{PropFixedP-ContinuousTime-Barabanov} For every $k \in \mathbb N$, $i_1, \dotsc, i_k \in \llbracket 1, N\rrbracket$, and $t_1, \dotsc, t_k \geqslant 0$, one has
\[\norm*{e^{A_{i_k} t_k} \dotsm e^{A_{i_1} t_1}}_{\mathrm e} = e^{\lambda_{\mathrm d}(A) (t_1 + \dotsb + t_k)}.\]
\item\label{PropFixedP-ContinuousTime-Spr} For every $k \in \mathbb N$, $i_1, \dotsc, i_k \in \llbracket 1, N\rrbracket$, and $t_1, \dotsc, t_k \geqslant 0$, one has
\[\spr(e^{A_{i_k} t_k} \dotsm e^{A_{i_1} t_1}) = e^{\lambda_{\mathrm d}(A) (t_1 + \dotsb + t_k)}.\]
\item\label{PropFixedP-ContinuousTime-SkewSymmetric} Up to a common change of basis, the matrices $A_i - \lambda_{\mathrm d}(A) \id$, $i \in \llbracket 1, N\rrbracket$, are skew-symmetric.
\end{enumerate}
\end{prop}

\begin{proof}
Note that, in terms of the flow $\Phi_\sigma$, items \ref{PropFixedP-ContinuousTime-Barabanov} and \ref{PropFixedP-ContinuousTime-Spr} can be equivalently stated by saying that the quantities $\frac{1}{t} \log \norm*{\Phi_\sigma(t)}_{\mathrm e}$ and $\frac{1}{t} \log \spr(\Phi_\sigma(t))$, respectively, are independent of $t > 0$ and $\sigma \in \Sigma$ and are equal to $\lambda_{\mathrm d}(A)$.

We will first show that \ref{PropFixedP-ContinuousTime-Equality}, \ref{PropFixedP-ContinuousTime-Barabanov}, and \ref{PropFixedP-ContinuousTime-Spr} are equivalent. The fact that \ref{PropFixedP-ContinuousTime-Barabanov} implies \ref{PropFixedP-ContinuousTime-Spr} follows from Gelfand's formula for the spectral radius. To show that \ref{PropFixedP-ContinuousTime-Spr} implies \ref{PropFixedP-ContinuousTime-Equality}, notice that, for every induced norm $\norm{\cdot}$ in $\mathcal M_d(\mathbb R)$, it follows from \ref{PropFixedP-ContinuousTime-Spr} that
\[
\lambda_{\mathrm d}(A) = \frac{1}{t} \log \spr(\Phi_\sigma(t)) \leqslant \frac{1}{t} \log \norm{\Phi_\sigma(t)}
\]
for every $\sigma \in \Sigma$ and $t > 0$. Hence, by first taking the expectation with respect to $\sigma$ and then the $\limsup$ as $t$ tends to infinity, we obtain that $\lambda_{\mathrm d}(A) \leqslant \lambda_{\mathrm p}(\nu, \mu, P, A)$, yielding \ref{PropFixedP-ContinuousTime-Equality} thanks to \eqref{LambdaPLeqLambdaD-Continuous}.

Let us now prove that \ref{PropFixedP-ContinuousTime-Equality} implies \ref{PropFixedP-ContinuousTime-Barabanov}. 
Without loss of generality, by Proposition~\ref{prop:condition-initiale}, we assume that $\nu$ is the unique invariant probability vector for $P$.
By definition of extremal norm, for every $k \in \mathbb N$, $i_1, \dotsc, i_k \in \llbracket 1, N\rrbracket$, and $t_1, \dotsc, t_k \geqslant 0$, one has
\begin{equation*}
\norm*{e^{A_{i_k} t_k} \dotsm e^{A_{i_1} t_1}}_{\mathrm e} \leqslant e^{\lambda_{\mathrm d}(A) (t_1 + \dotsb + t_k)},
\end{equation*}
which can be equivalently rewritten as
\begin{equation}
\label{eq:logBarabanov}
\frac{1}{t}\log\norm{\Phi_\sigma(t)}_{\mathrm e} \leqslant \lambda_{\mathrm d}(A),
\end{equation}
for every $t > 0$ and $\sigma \in \Sigma$.

Arguing by contradiction, there exist $k \in \mathbb N$, $i_1, \dotsc, i_k \in \llbracket 1, N\rrbracket$, and $t_1, \dotsc, t_k \geqslant 0$ such that
\begin{equation}
\label{ContradictionInegStricte}
\norm*{e^{A_{i_k} t_k} \dotsm e^{A_{i_1} t_1}}_{\mathrm e} < e^{\lambda_{\mathrm d}(A) (t_1 + \dotsb + t_k)}.
\end{equation}
We claim that, with no loss of generality, $p_{i_1 i_2} \dotsm p_{i_{k-1} i_k} > 0$. Indeed, if it were not the case, then $p_{i_\ell i_{\ell + 1}} = 0$ for some $\ell \in \llbracket 1, k-1\rrbracket$. Since $P$ is strongly connected, there exist $r \in \mathbb N$ and $j_1, \dotsc, j_r \in \llbracket 1, N\rrbracket$ such that $j_1 = i_\ell$, $j_r = i_{\ell + 1}$, and $p_{j_1 j_2} \dotsm p_{j_{r-1} j_r} > 0$. Letting $s_1 = t_\ell$, $s_r = t_{\ell + 1}$, and $s_2 = \dotsb = s_{r-1} = 0$, we may then replace $e^{A_{i_{\ell + 1}} t_{\ell + 1}} e^{A_{i_{\ell}} t_{\ell}}$ by $e^{A_{j_r} s_r} \dotsm e^{A_{j_1} s_1}$ in \eqref{ContradictionInegStricte}. Repeating the previous construction for every $\ell$ such that $p_{i_\ell i_{\ell + 1}} = 0$, the claim is proved. Note also that, by continuity, \eqref{ContradictionInegStricte} holds for an open subset of times $t_1, \dotsc, t_k$ in $(0, +\infty)^k$. We have thus proved that there exists $t > 0$ and a set of positive probability of signals $\sigma \in \Sigma$ such that
\[
\frac{1}{t} \log\norm{\Phi_\sigma(t)}_{\mathrm e} < \lambda_{\mathrm d}(A).
\]
Combining with \eqref{eq:logBarabanov}, we deduce from \eqref{eq:LyapProba} that $\lambda_{\mathrm p}(\nu, \mu, P, A) < \lambda_{\mathrm d}(A)$.

To prove that \ref{PropFixedP-ContinuousTime-SkewSymmetric} implies \ref{PropFixedP-ContinuousTime-Spr}, notice that, if $A_i - \lambda_{\mathrm d}(A) \id$ is skew-symmetric for every $i \in \llbracket 1, N\rrbracket$, then, for every $k \in \mathbb N$, $i_1, \dotsc, i_k \in \llbracket 1, N\rrbracket$, and $t_1, \dotsc, t_k \geqslant 0$, the matrix
\[
e^{(A_{i_k} - \lambda_{\mathrm d}(A) \id) t_k} \dotsm e^{(A_{i_1} - \lambda_{\mathrm d}(A) \id) t_1} = e^{A_{i_k} t_k} \dotsm e^{A_{i_1} t_1} e^{-\lambda_{\mathrm d}(A) (t_1 + \dotsb + t_k)}
\]
is the product of orthogonal matrices, hence it is itself orthogonal and its spectral radius is equal to $1$. The conclusion follows.

Finally, to prove that \ref{PropFixedP-ContinuousTime-Spr} implies \ref{PropFixedP-ContinuousTime-SkewSymmetric}, notice that the semigroup
\[
\{e^{(A_{i_k} - \lambda_{\mathrm d}(A) \id) t_k} \dotsm e^{(A_{i_1} - \lambda_{\mathrm d}(A) \id) t_1} \mid k \in \mathbb N,\; i_1, \dotsc, i_k \in \llbracket 1, N\rrbracket,\; t_1, \dotsc, t_k \geqslant 0\}
\]
is irreducible since $A$ is irreducible. Moreover, \ref{PropFixedP-ContinuousTime-Spr} is equivalent to saying that the above semigroup has constant spectral radius. Then, using \cite[Theorem~2]{Protasov2017Matrix}, we deduce that, up to a common change of basis, $e^{(A_i - \lambda_{\mathrm d}(A) \id) t}$ is orthogonal for every $i \in \llbracket 1, N\rrbracket$ and $t \geqslant 0$, yielding the conclusion.
\end{proof}

\begin{remk}
\label{RemkProofPropFixedP}
The equivalences between \ref{PropFixedP-ContinuousTime-Equality}, \ref{PropFixedP-ContinuousTime-Barabanov}, and \ref{PropFixedP-ContinuousTime-Spr} only rely on the extremality of the norm $\norm{\cdot}_{\mathrm e}$ for $A$, and hence hold under the weaker assumption that $A$ is nondefective instead of irreducible (cf.\ Remark~\ref{RemkNondefective}). Notice also that the proof that \ref{PropFixedP-ContinuousTime-Spr} implies \ref{PropFixedP-ContinuousTime-Equality} requires neither the irreducibility of $A$ nor the strong connectedness of $P$.
\end{remk}

In the next result, we extend Proposition~\ref{PropFixedP-ContinuousTime} to the more general case where $A$ is not necessarily irreducible, but $P$ is still assumed to be strongly connected.

\begin{prop}
\label{CoroFixedP-ContinuousTime}
Let $P \in \mathcal M_{N}(\mathbb R)$ be a stochastic strongly connected matrix, $\nu$ be a probability vector of $\R^N$, $\mu>0$, and $A = (A_1, \dotsc,\allowbreak A_N) \allowbreak \in \mathcal M_d(\mathbb R)^N$. Then the following statements are equivalent:
\begin{enumerate}
\item\label{CoroFixedP-ContinuousTime-Equality} $\lambda_{\mathrm d}(A) = \lambda_{\mathrm p}(\nu, \mu, P, A)$.
\item\label{CoroFixedP-ContinuousTime-Spr} For every $k \in \mathbb N$, $i_1, \dotsc, i_k \in \llbracket 1, N\rrbracket$, and $t_1, \dotsc, t_k \geqslant 0$, it holds
\[\spr(e^{A_{i_k} t_k} \dotsm e^{A_{i_1} t_1}) = e^{\lambda_{\mathrm d}(A) (t_1 + \dotsb + t_k)}.\]
\end{enumerate}
\end{prop}

\begin{proof}
As in the proof of Proposition~\ref{PropFixedP-ContinuousTime}, thanks to  Proposition~\ref{prop:condition-initiale}, we can suppose that $\nu$ is the unique invariant measure of $P$. 
Due to Remark~\ref{RemkProofPropFixedP}, we are only left to show that \ref{CoroFixedP-ContinuousTime-Equality} implies \ref{CoroFixedP-ContinuousTime-Spr}. As in \cite[Lemma~3.5]{Chitour2021Gap}, a key ingredient of the argument is a simultaneous block decomposition of the matrices $A_1, \dotsc, A_N$. The idea is the following: if $A_1, \dotsc, A_N$ admit a common proper subspace $V$ of dimension $d'$, then, up to a linear change of coordinates corresponding to a basis of $\mathbb R^d$ consisting of a basis of $V$ in its first $d'$ elements, each matrix $A_j$ can be written as $\left(\begin{smallmatrix}B_j & C_j \\ 0 & D_j\end{smallmatrix}\right)$ for some matrices $B_j, C_j, D_j$, $j \in \llbracket 1, N\rrbracket$, with $B_j\in  \mathcal M_{d'}(\mathbb R)$. By an immediate inductive argument, up to a linear change of coordinates, $A_1, \dotsc, A_N$ can be presented in block-tri\-an\-gu\-lar form as
\begin{equation}
\label{DecomposeAj}
A_j = \begin{pmatrix}
A_{j}^{(1)} & \ast & \ast & \cdots & \ast \\
0 & A_{j}^{(2)} & \ast & \cdots & \ast \\
0 & 0 & A_{j}^{(3)} & \ddots & \ast \\
\vdots & \vdots & \ddots & \ddots & \vdots \\
0 & 0 & 0 & \cdots & A_{j}^{(S)}
\end{pmatrix}, \qquad j \in \llbracket 1, N\rrbracket,
\end{equation}
for some appropriate integer $S$, with $A^{(s)} = (A_1^{(s)}, \dotsc, A_N^{(s)})$ irreducible for every $s \in \llbracket 1, S\rrbracket$. Both deterministic and probabilistic Lyapunov exponents are obtained as maxima of the corresponding Lyapunov exponents over the diagonal blocks (see, e.g., \cite[Proposition~2]{Chitour2012Marginal} for the deterministic case and \cite{Gerencser2008Stability} for the probabilistic one). Notice also that, for every $k \in \mathbb N$, $i_1, \dotsc, i_k \in \llbracket 1, N\rrbracket$, and $t_1, \dotsc, t_k \geqslant 0$, it holds
\begin{equation}
\label{eq:lambdaDGeqBlocks}
\lambda_{\mathrm d}(A) \geqslant \frac{\log\spr\Bigl(e^{A_{i_k} t_k} \dotsm e^{A_{i_1} t_1}\Bigr)}{t_1 + \dotsb + t_k} = \max_{s \in \llbracket 1, S\rrbracket} \frac{\log\spr\Bigl(e^{A_{i_k}^{(s)} t_k} \dotsm e^{A_{i_1}^{(s)} t_1}\Bigr)}{t_1 + \dotsb + t_k},
\end{equation}
where the inequality comes from \eqref{LambdaDGeqLambdaWord} and the equality follows from the simple fact that the spectral radius of a block-triangular matrix is equal to the maximum of the spectral radii over the diagonal blocks.

Let $\overline s \in \llbracket 1, S\rrbracket$ be the index such that $\lambda_{\mathrm p}(\nu, \mu, P, A)\allowbreak = \lambda_{\mathrm p}(\nu, \mu, P, A^{(\overline s)})$ and notice that, thanks to \ref{CoroFixedP-ContinuousTime-Equality} and \eqref{LambdaPLeqLambdaD-Continuous}, $\lambda_{\mathrm d}(A) = \lambda_{\mathrm d}(A^{(\overline s)})$. By Proposition~\ref{PropFixedP-ContinuousTime} and \eqref{eq:lambdaDGeqBlocks}, we deduce that, for every $k \in \mathbb N$, $i_1, \dotsc, i_k \in \llbracket 1, N\rrbracket$, and $t_1, \dotsc, t_k \geqslant 0$, it holds
\[
\lambda_{\mathrm d}(A) = \lambda_{\mathrm d}(A^{(\overline s)}) = \frac{\log\spr\Bigl(e^{A_{i_k}^{(\overline s)} t_k} \dotsm e^{A_{i_1}^{(\overline s)} t_1}\Bigr)}{t_1 + \dotsb + t_k} \leqslant \frac{\log\spr\Bigl(e^{A_{i_k} t_k} \dotsm e^{A_{i_1} t_1}\Bigr)}{t_1 + \dotsb + t_k} \leqslant \lambda_{\mathrm d}(A),
\]
yielding \ref{CoroFixedP-ContinuousTime-Spr}.
\end{proof}

\begin{remk}
As a byproduct of the block-decomposition argument in the above proof, we get another statement equivalent to \ref{CoroFixedP-ContinuousTime-Equality} and \ref{CoroFixedP-ContinuousTime-Spr}: there exist a linear change of variables and $p \in \llbracket 1, d\rrbracket$ such that, for every $i \in \llbracket 1, N\rrbracket$,
\[
A_i - \lambda_{\mathrm d}(A) \id = 
\begin{pmatrix}
\ast & \ast \\
0 & B_i \\
\end{pmatrix}
\]
with $B_i$ a $p \times p$ skew-symmetric matrix.

Moreover, any of the previous statements implies that
\[
\lambda(M) = {\lambda_{\mathrm d}(A)}, \qquad \forall M \in \co(A).
\]
Indeed, write $M = \beta_1 A_1 + \dotsb + \beta_N A_N$ with $\beta_1, \dotsc, \beta_N \in [0, 1]$ and $\beta_1 + \dotsb + \beta_N = 1$. Take $k = N$ and $i_j = j$ and $t_j = t \beta_j$ for $j \in \llbracket 1, N\rrbracket$ in \ref{CoroFixedP-ContinuousTime-Spr}. The conclusion follows by letting $t \to 0$.
\end{remk}

Finally, we can turn to the proof of Theorem~\ref{TheoPNotSC-ContinuousTime}.

\begin{proof}[Proof of Theorem~\ref{TheoPNotSC-ContinuousTime}]
According to Proposition~\ref{prop:condition-initiale}, \ref{TheoPNotSC-ContinuousTime-Equality} is equivalent to the fact that for all $i\in \cco 1,R\ccf$ such that $\mathcal I(i)$ is accessible from $\nu$, $\lambda_{\mathrm d}(A) = \lambda_{\mathrm p}(\nu^{[i]}, \mu, P, A)$. Replacing $P$ by $P_i$ and $A$ by $(A_j)_{j\in\mathcal I(i)}$, we are in a strongly connected case and Proposition~\ref{CoroFixedP-ContinuousTime} concludes.
\end{proof}

\section{Compactification of the space of Markov processes}\label{sec:compactification}

The aim of this section is to prove Theorem~\ref{Thm:CV-mu-pas-borne}, that is, that any sequence $(\nu_n,\mu_n,P_n)_{n\in\N}$ of Markov processes for $A=(A_1,\dots,A_N)$ admits a subsequence converging in law to a convexified Markov process for $A$. 

Notice that an important difference from more classical averaging results (such as  \cite[Section~2.5]{Benaim2014Stability} or \cite[Corollary~2.15]{benaim2019}) is that $P_n$ is not fixed and, in particular, the jump rates $\mu_n (P_n)_{i,j}$ may have different asymptotic behaviours as $n$ goes to infinity depending on $i,j$, which is why, when $\lim_{n\to \infty}\mu_n=+\infty$, the limit process is not necessarily a deterministic ODE $\dot x  = M x$ for some $M \in \co(A)$.

\subsection{Convexified Markov processes as limits of Markov processes}

First, let us show that any convexified Markov process for $A$ can be obtained as the limit of a sequence  of Markov processes for $A$. 
For the sake of clarity, let us stress that we only consider here the convergence of the continuous component $x(\cdot)$.

\begin{prop}\label{Prop:CVversAveraged}
Let $(x,\sigma)$ be a convexified Markov process for $A$. Then there exists a sequence $(x_n,\sigma_n)_{n\in\mathbb N}$ of Markov processes for $A$ such that, for all $T,\delta>0$,
\begin{equation*}
\mathbb P \po \sup_{t\in [0,T]} |x(t)-x_n(t)| > \delta \pf \underset{n\rightarrow+\infty}\longrightarrow 0\,.
\end{equation*}
\end{prop}
\begin{proof}
Denote  by $B_1,\dots,B_k$ the modes of $(x,\sigma)$ with $B_r = \sum_{j\in I_r} \pi_r(j) A_j$, where $I_1,\dots,I_k$ are 
pairwise disjoint nonempty subsets of $ \llbracket 1, N\rrbracket$ and, for all $r\in\cco 1,k\ccf$, $\pi_r$ is a probability measure on $I_r$. In particular, $(\sigma(t))_{t\geqslant 0}$ is a continuous-time Markov chain on $\cco 1,k\ccf$. Denote by $\mu$ its jump rate, $P$ its transition matrix, and  $\nu$ its initial probability law. Let $T_0 = 0$ and $(T_m)_{m\geqslant 0}$ be the jump times of $\sigma$, so that $(T_{m+1} - T_m)_{m\in\N}$ is an i.i.d.\ sequence of random variables distributed according to the exponential law with parameter $\mu>0$. We will construct for all $n\in\N$ a Markov chain $(\sigma_n(t))_{t\geqslant 0}$ on $\cco 1,N\ccf$ such that $\sigma_n(t) \in I_{\sigma(t)}$ for all $t\geqslant 0$ and which is moreover re-sampled at rate $n$ according to $\pi_{\sigma(T_m)}$ between consecutive slow jump times $T_m$ and $T_{m+1}$. 

More precisely, let $(N_t)_{t\geqslant 0}$ be a standard Poisson process with intensity $1$ and  $\mathcal U=(\mathcal U_{p,r})_{p\in\N,r\in\cco 1,k\ccf}$ be a family of independent random variables such that, for all $p\in \N$ and $r\in \cco 1,k\ccf$, $\mathcal U_{p,r}$ takes values in $I_r$ and is distributed according to $\pi_r$, with moreover $(N_t)_{t\geqslant 0}$, $\mathcal U$, and $\sigma$ independent. For all $n\in\N$ and $t\geqslant 0$, set $M^n_t = N_{nt} + \sum_{m\in \N} \1_{t\geqslant T_m}$, so that $(M^n_t)_{t\geqslant 0}$ is a Poisson process with intensity $\mu +n$ such that that all jumps of $\sigma$ are jumps of $M^n$. For all $n\in\N$ and $t\geqslant 0$, set $\sigma_n(t) = \mathcal U_{M_t^n,\sigma(t)}$.

Then, for all $n \in \N$, $(\sigma_n(t))_{t\geqslant 0}$ is a Markov chain on $\cco 1,N\ccf$ with initial condition $\nu_n$ and jump rates $\lambda_n(i,j)$ for $i,j\in\cco 1,N\ccf$ given as follows:
\begin{itemize}
\item for all $i\in \cco 1,N\ccf$, $\nu_n(i) = \nu(r)\pi_r(i)$ if $i\in I_r$ with $r\in\cco 1,k\ccf$;
\item for all $r\in \cco 1,k\ccf$ and all $i,j \in I_r$, $\lambda_n(i,j) = n \pi_r(j)$;
\item for all distinct $r,s\in\cco 1,k\ccf$ and all $i\in I_r$, $j\in I_s$, $\lambda_n(i,j) = \mu \pi_s(j) P(r,s)$;
\item for all $i \notin I:=\bigcup_{r=1}^k I_r$ and all $j\in \cco 1,N\ccf$, $\lambda_n(i,j)=\lambda_n(j,i) = 0$.
\end{itemize}

Remark that, for all $n\in\N$ and $t\geqslant 0$,    $\sigma_n(t) \in I_{\sigma(t)}$ (and in particular  $\sigma_n(t) \in I$), so that $\sigma$ is completely determined by $\sigma_n$.

The proof  is then similar to \cite[Lemma~2.14]{benaim2019}. From \cite[Chapter~2, Theorem~1.3]{FreidlinWentzell}, it is sufficient to prove that for all $t_0,T>0$, $\int_{t_0}^{t_0+T}A_{\sigma_n(s)}\dd s$ converges in probability as $n\rightarrow \infty$ towards $\int_{t_0}^{t_0+T}B_{\sigma(s)}\dd s$, uniformly with respect to $t_0$. As in the proof of \cite[Lemma~2.14]{benaim2019}, it is thus sufficient to prove that, for all $r\in \cco 1,k\ccf$ and all $j\in I_r$, $\int_{t_0}^{t_0+T}\1_{\sigma_n(s)=j}\dd s$ converges in probability as $n\rightarrow \infty$ towards $\pi_r(j)\int_{t_0}^{t_0+T}  \1_{\sigma(s)=r}\dd s$, uniformly with respect to $t_0$. 
By the Markov property, for all $\delta>0$,
\begin{multline*}
  \mathbb P\po \left| \int_{t_0}^{t_0+T}\po \1_{\sigma_n(s)=j}  - \pi_r(j)   \1_{\sigma(s)=r}\pf \dd s\right| > \delta \pf \\
\ \leqslant \ \sup_{i\in I} \mathbb P_i\po   \left| \int_{0}^{T} \po \1_{\sigma_n(s)=j}  -  \pi_r(j)   \1_{\sigma(s)=r}\pf\dd s\right| > \delta   \pf, 
\end{multline*}
where we recall that the subscript $i$ denotes the conditioning $\sigma_n(0)=i$.

Let $R$ be a positive integer. For all $b\in\cco 0,R-1\ccf$ and all $t\geqslant bT /R$, denote
\[\tilde \sigma_n^b(t) = \mathcal U_{N_{nt} - N_{nbT /R} + M^n_{bT /R},\sigma(bT /R) }.\]
In other words, $\tilde \sigma_n^b(t)$ is initialized at time $bT/R$ with $\tilde \sigma_n^b(bT/R) = \sigma_n(bT/R)$ and then is re-sampled on $I_{\sigma(bT/R)}$ at each jump of $N_{nt}$. In particular, up to $\inf\{T_m\mid m\in\N,\;T_m>bT/R\}$, the first slow jump time after time $bT/R$, we have $\sigma_n = \tilde \sigma_n^b$. 
In particular, $\sigma_n$ and $\tilde \sigma_n^b$ coincide on the interval 
$[bT/R,(b+1)T/R]$ if the latter does not contain any slow jump time $T_m$. 
Hence, using the identity $\pi_r(j) \1_{r'=r}=\pi_{r'}(j)$, we deduce the bound
\begin{align*}
 \Big| \int_{0}^{T} \big( \1_{\sigma_n(s)=j}  - {} &  \pi_r(j)   \1_{\sigma(s)=r}\big)\dd s\Big| \\
 & \leqslant  \sum_{b=0}^{R-1 }  \left| \int_{bT/R}^{(b+1)T/R} \po \1_{\sigma_n(s)=j}  -  \pi_r(j)   \1_{\sigma(s)=r}\pf\dd s\right|\\
 & \leqslant   \frac{T}{R} \sum_{m\in\N} \1_{T_m\leqslant  T} + \sum_{b=0}^{R-1 }  \left| \int_{bT/R}^{(b+1)T/R} \po \1_{\tilde \sigma_n^b(s)=j}  -  \pi_{\sigma(bT/R)}(j)\pf \dd s   \right|.
\end{align*}
Consider the events
\[\mathcal A_b \ = \ \left\{  \left| \int_{bT/R}^{(b+1)T/R} \po \1_{\tilde \sigma_n^b(s)=j}  -  \pi_{\sigma(bT/R)}(j)\pf \dd s   \right| > \frac{\delta}{2R}\right\} \]
for all $b\in\cco 0,R-1\ccf$ and
\[\mathcal A_R \ = \ \left\{ \frac{T}{R} \sum_{m\in\N} \1_{T_m\leqslant  T} > \frac{\delta}{2}\right\} \,.\]
Then
\[\left\{\left| \int_{0}^{T} \po \1_{\sigma_n(s)=j}  -  \pi_r(j)   \1_{\sigma(s)=r}\pf\dd s\right| > \delta\right\} \ \subset\  \bigcup_{b=0}^{R} \mathcal A_b\]
and thus
\begin{eqnarray*}
\mathbb P_i\po   \left| \int_{0}^{T} \po \1_{\sigma_n(s)=j}  -  \pi_r(j)   \1_{\sigma(s)=r}\pf\dd s\right| > \delta   \pf  & \leqslant &   \sum_{b=0}^{R} \mathbb P_i(\mathcal A_b)\,.
\end{eqnarray*}
Since the number of jumps occurring before time $T$ follows a Poisson distribution with intensity $T\mu$,
\[\mathbb P_i(\mathcal A_R) \ \leqslant \ \frac{2 T \mathbb E\po \sum_{m\in\N} \1_{T_m\leqslant  T}  \pf }{\delta R} \ \leqslant \ \frac{2T^2 \mu}{\delta R}\,.\]
Moreover, conditioning on the value $\sigma_n(bT/R)$, we get that for all $i\in I$ and $b\in \cco 0,R-1\ccf$,
\begin{eqnarray*}
\mathbb P_i\po \mathcal A_b\pf & \leqslant  & \sup_{u\in I} \mathbb P_u \po  \left| \int_{0}^{T/R} \po \1_{\tilde \sigma_n^0(s)=j}  -  \pi_{\sigma(0)}(j)\pf \dd s   \right| > \frac{\delta}{2R} \pf\\
& \leqslant & \frac{4R^2}{\delta^2}  \sup_{u\in I} \mathbb E_u \po  \left| \int_{0}^{T/R} \po \1_{\tilde \sigma_n^0(s)=j}  -  \pi_{\sigma(0)}(j)\pf \dd s   \right|^2 \pf\,.
\end{eqnarray*}
It only remains to prove that for all $u\in I$ the expectation vanishes as $n\rightarrow +\infty$. Indeed, in that case, we obtain that for all $T,\delta,R>0$,
\[\limsup_{n\rightarrow +\infty} \sup_{i\in I} \mathbb P_i\po   \left| \int_{0}^{T} \po \1_{\sigma_n(s)=j}  -  \pi_r(j)   \1_{\sigma(s)=r}\pf\dd s\right| > \delta   \pf  \ \leqslant \ \frac{2T^2 \mu}{\delta R}\,,\]
yielding the conclusion, since the left-hand side does not depend on $R$, which can thus be taken arbitrarily large.

Let us fix $r\in\cco 1,k\ccf$ and $u\in I_{r}$, and work conditionally to $\{\sigma_n(0)=u\}$. Under this event, $\tilde \sigma_n^0$ is simply a Markov chain starting at $u$ and re-sampled according to $\pi_r$ at rate $n$. We are back to a problem similar to \cite[Lemma~2.14]{benaim2019}, and follow the same proof.
  In particular, for all $j\in I_r$,
\[\mathbb P_u (\tilde \sigma_n^0 (t) = j) \ =\  e^{-nt}\1_{j=u} + (1-e^{-nt}) \pi_r(j)\,,\]
and thus 
\[\left| \mathbb E_u \po \int_0^{T/R} \po \1_{\tilde \sigma_n^0(s)=j}  -  \pi_{r}(j)\pf \dd s  \pf\right|  \ \leqslant \ \int_0^{T/R} e^{-ns}\dd s \ \leqslant \ \frac1n\,.\]
Similarly
\begin{eqnarray*}
\mathbb E_u\po \po \int_{0}^{T/R} \1_{\tilde \sigma_n^0(s)=j}\dd s\pf^2 \pf & = & \int_{0}^{T/R} \int_{0}^{T/R} \mathbb P_u(\tilde \sigma_n^0(s)=j,\ \tilde \sigma_n^0(t)=j) \dd s \dd t\\
& = & 2\int_{0}^{T/R} \int_{0}^{t} \mathbb P_u(\tilde \sigma_n^0(s)=j,\ \tilde \sigma_n^0(t)=j) \dd s \dd t \,.
\end{eqnarray*}
For all $s<t$,
\begin{align*}
|\mathbb P_u&\po \tilde \sigma_n^0(s)=j,\ \tilde \sigma_n^0(t)=j\pf- \pi_r(j)^2 | \\
 = {}& \left| \po e^{-n s}\1_{j=u} + (1-e^{-ns}) \pi_r(j)\pf \po e^{-n(t-s)}  + (1-e^{-n(t-s)}) \pi_r(j)\pf- \pi_r(j)^2 \right|\\
 \leqslant  {}&{
e^{-n t}\1_{j=u}+e^{-n s}(1-e^{-n(t-s)})\1_{j=u}\pi_r(j)+ (1-e^{-ns})e^{-n(t-s)}\pi_r(j)}\\
& {+
\left| (1-e^{-ns}) (1-e^{-n(t-s)})-1\right|\pi_r(j)^2 }\\
 \leqslant {}&  {e^{-nt}+ } 2 e^{-ns} + 2 e^{-n(t-s)}\,,
\end{align*}
so that
\[\left|\mathbb E_u\po \po \int_{0}^{T/R} \1_{\tilde \sigma_n^0(s)=j}\dd s\pf^2 \pf - T^2\pi_r(j)^2/R^2\right| \ \underset{n\rightarrow +\infty}\longrightarrow \ 0 \,.\] 
 This concludes since, denoting $Z_n = \int_{0}^{T/R} \1_{\tilde \sigma_n^0(s)=j}\dd s$, we have then obtained that
\[
\begin{split}
\mathbb E_u \po \left| Z_n - T \pi_r(j)/R \right|^2 \pf & = \mathbb E_u(Z_n^2) - 2 T\pi_r(j) \mathbb E_u(Z_n)/R + T^2\pi_{r}(j)^2/R^2\\
& \underset{n\rightarrow +\infty}\longrightarrow T^2 \pi_r(j)^2/R^2 - 2 T^2 \pi_r(j)^2/R^2  + T^2 \pi_{r}(j)^2/R^2  \ = \ 0\,. \qedhere
\end{split}
\]
\end{proof}

\subsection{Convexified Markov processes compactify the space of Markov processes}

We now want to prove a converse of Proposition~\ref{Prop:CVversAveraged}, namely that, from any sequence  $(x_n)_{n\in\N}$ of Markov processes for $A$, we can extract a subsequence that converges to a
convexified process, cf.\ Theorem~\ref{Thm:CV-mu-pas-borne} below. We start by treating separately the simple case where the maximal jump rate $\mu_n$ is bounded.

\begin{prop}\label{Prop:CV-mu-borne}
Consider a sequence of parameters $(\nu_n,\mu_n,P_n)_{n\in\N}$ of Markov processes for $A$. Suppose that $(\mu_n)_{n\in\N}$ is bounded. Then, up to extracting a subsequence, there exist parameters $(\nu,\mu,P)$ of  a Markov process for $A$ such that the following holds:  There exist Markov processes $(x_n,\sigma_n)_{n\in\N}$ and $(x,\sigma)$ associated respectively with $(\nu_n,\mu_n,P_n)_{n\in\N}$ and $(\nu,\mu,P)$ such that, for all $T>0$ and $\delta>0$,
\begin{eqnarray*}
\mathbb P \po   \sup_{t\in [0,T]}|x(t) - x_n(t)| > \delta \pf & \underset{n\rightarrow+\infty}\longrightarrow & 0\,.
\end{eqnarray*}
\end{prop}

\begin{proof}
Up to an extraction we can suppose that $\nu_n$, $\mu_n$, and $P_n$ have limits as $n\rightarrow +\infty$ (in fact it would 
have been sufficient to assume that $\liminf_{n\rightarrow +\infty} \mu_n < +\infty$), that we denote by $\nu$, $\mu$, and $P$, respectively.

We are going to prove that for all $T,\varepsilon>0$, there exists $n_0\in\N$ such that, for all $n\geqslant n_0$, there exist Markov processes $(x_n,\sigma_n)$ and $(x,\sigma)$ associated respectively with $(\nu_n,\mu_n,P_n)$ and $(\nu,\mu,P)$ such that
\[\mathbb P \po   x(t) = x_n(t),\, \forall t\in[0,T] \pf \ \geqslant \ 1 - \varepsilon\,.\]
Remark that, in this statement, $(x,\sigma)$ may depend on $n$. Nevertheless, this yields the convergence of the distribution of $(x_n(t))_{t\geqslant0}$ to the distribution of $(x(t))_{t\geqslant 0}$ on all compact time intervals. The result then follows from Skorokhod's representation theorem.

The proof relies on a synchronous coupling of the Markov chains, namely, for each $n\in\N$, 
we can define simultaneously two Markov processes $(z_n(t))_{t\geqslant0}:=(x_n(t),\sigma_n(t))_{t\geqslant 0}$ and  $(z(t))_{t\geqslant0}:=(x(t),\sigma(t))_{t\geqslant 0}$ associated respectively with $(\nu_n,\mu_n,P_n)$  and $(\nu,\mu,P)$ on the same probability space in such a way that they have the same initial value with maximal probability (i.e., $\mathbb P(z_n(0)\neq z(0)) = \abs{\nu-\nu_n}_1 / 2$) and that, as long as they stay at the same position, they jump as much as possible at the same times and to the same locations. The precise construction of this coupling is given in \cite[Section~6]{DurmusGuillinMonmarche}, to which we refer for details. The Markov generator $\mathcal L$ on $\mathbb R^d\times \cco 1,N\ccf$ associated with  $(z(t))_{t\geqslant0}$ is given by
\[\mathcal L f(x,\sigma) \ = \ (A_\sigma x) \cdot \nabla_x f(x,\sigma) + \mu \sum_{j=1}^N (P)_{\sigma,j} \po f(x,j)-f(x,\sigma)\pf \,,\]
and similarly for the generator $\mathcal L_n$ of $z_n$. Then
\[\epsilon_n \ := \ \sup_{\|f\|_\infty \leqslant 1} \|\mathcal Lf-\mathcal L_n f\|_{\infty} \ \leqslant \ 2|\mu_n - \mu| + 2\mu\|P_n-P\|_1 \ \underset{n\rightarrow +\infty}\longrightarrow \ 0\,.\]
From \cite[Theorem~11]{DurmusGuillinMonmarche},
\begin{equation*}
\mathbb P \po \exists t\in[0,T],\    x(t) \neq x_n(t)\pf \leqslant \mathbb P \po z_n(0)\neq z(0)\pf + 1 - e^{-\varepsilon_n T}\ \underset{n\rightarrow +\infty}\longrightarrow \ 0\,. \qedhere
\end{equation*}
\end{proof}

For the rest of this section we consider a given sequence $(\nu_n,\mu_n,P_n)_{n\in\N}$ of Markov processes for $A$ such that, for at least a pair $(i,j)$ of distinct elements of  $\cco 1,N\ccf$, the jump rate $\lambda_n(i,j) = \mu_n (P_n)_{i,j}$ is unbounded.
  We will repeatedly consider successive  extractions of this sequence and keep writing them  $(\nu_n,\mu_n,P_n)_{n\in\N}$. First, up to extracting a subsequence, we suppose that $\lambda_n(i,j)\rightarrow +\infty$ for some distinct $i,j$.

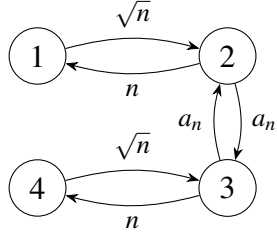
\begin{figure}[ht]
\centering
\begin{tikzpicture}
\node[circle, draw] (1) at (0, 0) {$1$};
\node[circle, draw] (2) at (2.5, 0) {$2$};
\node[circle, draw] (3) at (2.5, -1.75) {$3$};
\node[circle, draw] (4) at (0, -1.75) {$4$};
\draw[-Stealth] (1) to[out=15, in=165] node[midway, above] {\footnotesize $\sqrt{n}$} (2);
\draw[-Stealth] (2) to[out=-165, in=-15] node[midway, below] {\footnotesize $n$} (1);
\draw[-Stealth] (2) to[out=-75, in=75] node[midway, right] {\footnotesize $a_n$} (3);
\draw[-Stealth] (3) to[out=105, in=-105] node[midway, left] {\footnotesize $a_n$} (2);
\draw[-Stealth] (4) to[out=15, in=165] node[midway, above] {\footnotesize $\sqrt{n}$} (3);
\draw[-Stealth] (3) to[out=-165, in=-15] node[midway, below] {\footnotesize $n$} (4);
\end{tikzpicture}
\caption{A chain on $\llbracket 1,4\rrbracket$ whose jump rates all go to $+\infty$.}
\label{fig:chaine4}
\end{figure}

As in the proof of Proposition~\ref{Prop:CVversAveraged}, we would like to separate in the chain fast transitions that happen in arbitrarily small time as $n\rightarrow +\infty$ and slow transitions. It is not sufficient to consider the pairs $(i,j)$ such that $\lambda_n(i,j) \rightarrow +\infty$ or such that $\lambda_n(i,j)$ is of the order of $\mu_n$. Indeed, consider the example given in Figure~\ref{fig:chaine4}, where the values over the arrows denote the jump rates and we assume that $a_n$ goes to infinity as $n \to +\infty$ with the assumption that $\lim_{n\to+\infty} \frac{a_n}n =0$. For instance, starting from the state $2$, the Markov chain will go to $1$ with high probability (for large $n$). Then, each time it will go back to state $2$, it will have a probability $a_n /(n+a_n )$ to go to state $3$, from which it will go very fast to state $4$ with high probability. Since the time taken by transitions from $2$ to $1$, of order $1/n$, is negligible with respect to the time taken by transitions from $1$ to $2$ which is of order $1/\sqrt n$, and since the number of transitions from $1$ to $2$ before the chain reaches $3$ follows a geometric law with parameter $a_n /(n+a_n )$, the typical time to see a transition between $1$ and $4$ is of order $\sqrt n/a_n $. If $a_n = n^{1/3}$, for large $n$, it is unlikely to see such a transition before a given time $T$ (independent of $n$), so that the corresponding Markov process $x_n$ is expected to converge to the deterministic solution of $\dot x= A_1 x$. If $a_n = n^{2/3}$, transitions between $1$ an $4$ get arbitrarily fast for large $n$ and  a fast averaging phenomenon  leads to $\dot x= [(A_1+A_4)/2]x$ (the time spent in $2$ and $3$ being negligible). If $a_n=\sqrt n$, transitions between $1$ and $4$ occur at a rate of order $1$, so the limit process is $\dot x= A_{\sigma}x$, where $\sigma$ is an irreducible Markov chain on $\{1,4\}$.

Trying to adapt this analysis to a general chain leads to a recursive construction of several timescales at which different transitions occur. Such a rigorous construction is precisely the topic of the work \cite{Landim2016} by Landim and Xu (itself based on \cite{BELTRAN20111633} which deals with reversible Markov chains), upon which we will rely. Nevertheless, stated as they are, the results of \cite{Landim2016} do not fully match our needs. For this reason, we sightly reformulate them below. 

Let us check that the assumptions of \cite{Landim2016} are satisfied (at least up to extracting a subsequence). The first one is that the chain is strongly connected for all $n\in\N$. This does not necessarily hold in our case, but we will be able to reduce the problem to this case by a standard argument, see the proof of Theorem~\ref{Thm:CV-mu-pas-borne}. For this reason, we can suppose that $P_n$ is strongly connected for all $n\in\N$. 

Up to extracting a subsequence, we can suppose that
\begin{equation}\label{cond:lambda_n}
\forall (i,j)\in\cco 1,N\ccf^2\,,\qquad 
\begin{cases}
\text{ either }&\lambda_n(i,j)=0\ \forall n\in\N,\\
\text{ or }&\lambda_n(i,j)>0\ \forall n\in\N,
\end{cases}
\end{equation}
and we denote $\mathbb B = \{(i,j)\in\cco 1,N\ccf^2 \mid \lambda_n(i,j)>0\ \forall n\in\N\}$. Notice that, the chains being strongly connected, necessarily $\mathbb B \neq \emptyset$.
   
\begin{defi}
For $r\geqslant 2$, a family $\{(a_n^i)_{n\in\N}\}_{i\in\cco 1,r\ccf}$ of positive sequences is said to be \emph{ordered} if $\arctan(a_n^i/a_n^j)$ converges  as $n\rightarrow +\infty$ for all $i,j\in\cco 1,r\ccf$.

For a pair of ordered positive sequences we write $a_n\ll b_n$ (resp., $\simeq,\gg$) if $a_n/b_n\rightarrow 0$ (resp., $1$, $+\infty$) as $n\rightarrow +\infty$.
\end{defi}

Two positive sequences form an ordered pair up to extracting a subsequence, and   the same is true for a finite family of  sequences. As a consequence,     denoting by $\mathfrak{A}_m$, for all $m\in\N$, the set of functions $k:\mathbb B \rightarrow \N$ such that $\sum_{(i,j)\in\mathbb B} k(i,j) = m$, we see that, up to extracting a subsequence, by a diagonal argument,
\begin{equation}\label{Eq:CondOrdered}
\forall m\in\N,\ \left\{\left( \prod_{(i,j)\in\mathbb B} \lambda_n(i,j)^{k(i,j)}\right)_{n\in\N}\right\}_{k\in\mathfrak{A}_m} \text{ is ordered,}
\end{equation}
which is \cite[Assumption~2.6]{Landim2016}.

Let us now describe the consequences of this, established in \cite{Landim2016}. The following result is an adaptation from Theorems~2.1, 2.7, and 2.12 of \cite{Landim2016}. As such an adaptation requires the introduction of several definitions and notations, it is postponed to Appendix~\ref{Sec:MarkovDecomposition}.

\begin{theo}\label{Thm:Landim}
Consider for all $n\in\N$ a strongly connected Markov chain $(\sigma_n(t))_{t\geqslant 0}$ on $\cco 1,N \ccf$ with jump rates $(\lambda_n(i,j))_{i,j\in\cco 1,N\ccf}$. Under conditions~\eqref{cond:lambda_n} and \eqref{Eq:CondOrdered}, there exist $\mathfrak p\geqslant 1$, a decreasing sequence $ \mathfrak{n}_1,\dots,\mathfrak{n}_{\mathfrak{p}+1}$
in $\llbracket 1,N\rrbracket$,
a family of $\mathfrak{p}+1$ partitions $\{\mathcal E_1^j,\dots,\mathcal E_{\mathfrak{n}_j}^j,\Delta^j\}$ of $\cco 1,N\ccf$,  
$j\in \cco 1,\mathfrak{p}+1\ccf$,
and $\mathfrak{p}$ positive sequences  $ \theta^j = (\theta^j_n)_{n\in\N}$, $j\in \cco 1,\mathfrak{p}+1\ccf$, with the following properties:
\begin{enumerate}
\item \label{lan1_timescale} The timescales $\theta^j$ are increasing with respect to $j$, in the sense that for all $j\in\cco 1,\mathfrak{p}-1\ccf$, $\theta_n^{j} \ll \theta_n^{j+1}$. Moreover, the fastest
 timescale $\theta^1$ is obtained by taking $1/\theta_n^1 =\sum_{i,j=1}^N \lambda_n(i,j)$ for all $n\in\N$.
\item \label{lan2_CVMarkov}   Denote, for all $i \in \llbracket 1, N\rrbracket$ and $j\in\cco 1,\mathfrak{p}+1\ccf$,
\[\Psi^j(i) \ = \ \sum_{x=1}^{\mathfrak{n}_j} x \1_{i\in \mathcal E_x^j} \,, \]
which we call the \emph{coarse-grained variable at level $j$}. For all  $j\in\cco 1,\mathfrak{p}\ccf$, the transitions of $\Psi^j(\sigma_n)$ occur at the timescale $\theta^j$, and are approximately Markovian, in the sense that there exists a Markov chain $(X^j(t))_{t\geqslant 0}$ on $\cco 1,\mathfrak{n}_j\ccf$ that has   at least one nonzero jump rate and  such that, for all $x,y\in\cco 1,\mathfrak{n}_j\ccf$, $i\in \mathcal E_x^j$, and $t>0$,
\begin{equation}\label{eq:LandimCVX}
\mathbb P_i \po \Psi^j\po \sigma_n(t\theta_n^j)\pf =y\pf \ \underset{n\rightarrow +\infty}\longrightarrow \ \mathbb P_x\po X^j(t) =y\pf 
\end{equation}
and, for all $\delta>0$,
\begin{equation}\label{eq:LandimCVX_integral}
\mathbb P_i \po \left| \int_0^t \po \1_{\Psi^{j}(\sigma_n(s\theta_n^j)) = y} - \1_{  X^{j}(s) = y}\pf   \dd s\right| > \delta\pf  \  \underset{n\rightarrow +\infty}\longrightarrow  \ 0\,.
\end{equation}
More generally, \eqref{eq:LandimCVX} and \eqref{eq:LandimCVX_integral} still hold   if $\theta_n^j$ is replaced by $\tilde \theta_n^j$ with $\tilde \theta_n^j\simeq \theta_n^j$.
\item \label{lan3_Pastransition} There are no transitions between two successive timescales $\theta^j$ and $\theta^{j+1}$ or at a larger timescale than $\theta^{\mathfrak{p}}$, in the sense that, for all $j \in\cco 1, \mathfrak{p}\ccf$,  for all positive sequence $(\alpha_n)_{n\in\N}$ with $\theta_n^j \ll \alpha_n\ll \theta_n^{j+1}$ (where, for $j=\mathfrak{p}$, we set $\theta_n^{\mathfrak p+1}=+\infty$ for all $n\in\N$),  for all $x,y\in\cco 1,\mathfrak{n}_{j+1}\ccf$, $i\in \mathcal E_x^{j+1}$, and $t>0$,
\[\mathbb P_i \po \Psi^{j+1}\po \sigma_n(t\alpha_n)\pf =y\pf \ \underset{n\rightarrow +\infty}\longrightarrow \ \1_{x=y}  
\] 
and for all $\delta>0$,
\begin{equation*}
\mathbb P_i \po \abs*{ \int_0^t  \1_{\Psi^{j+1}(\sigma_n(s\alpha_n)) = y}    \dd s - t\1_{x = y}} > \delta\pf  \  \underset{n\rightarrow +\infty}\longrightarrow  \ 0\,.
\end{equation*}
\item \label{lan4Deltanegligeable} The time spent in $\Delta^j$ at a scale larger than $\theta^{j-1}$ is negligible, in the sense that    for all   $j\in\cco 2,\mathfrak p\ccf$, all positive sequence $(\alpha_n)_{n\in\N}$ with $\theta_n^{j-1} \ll \alpha_n$ and $t\geqslant 0$,
\[ \max_{i \in \cco 1,N\ccf} \mathbb E_i \po \int_0^t \1_{\sigma_n(s\alpha_n) 
\in \Delta^j} \dd s \pf \ \underset{n\rightarrow +\infty}\longrightarrow \ 0\,.\]
\item \label{lan5coarser} For $j=1$,  we have $\mathfrak{n}_1=N$,
$\mathcal E_i^1=\{i\}$ for every $i\in\llbracket 1,N\rrbracket$
 (i.e., all the points are separated), and $\Delta^1=\emptyset$. Then, the partitions get  coarser and are given by the recurrence classes of the limit chains. More precisely, for all  $j\in\cco 1,\mathfrak{p}\ccf$, the limit chain $X^j$ admits $\mathfrak{n}_{j+1}$ recurrence classes $\mathcal C_1^j,\dots,\mathcal C_{\mathfrak{n}_{j+1}}^j$, and for all $x\in \cco 1,\mathfrak{n}_{j+1}\ccf$, $\mathcal E_x^{j+1} = \bigcup_{y\in \mathcal C_x^j} \mathcal E_y^j$. Similarly, the set $\Delta^j$ is increasing, and is given by the transient points of $X^j$. More precisely, denoting $\mathcal T_j$ the set of transient points of $X^j$, then $\Delta^{j+1} = \Delta^j \cup \po \bigcup_{y\in\mathcal T_j} \mathcal E_y^j\pf$. In particular, for all $j\in\cco 1,\mathfrak{p}+1\ccf$ and all $x\in\cco 1,\mathfrak{n}_j\ccf$, $\mathcal E_x^j \neq \emptyset$ (while possibly $\Delta^j = \emptyset$). The last partition is trivial, in the sense that $\mathfrak{n}_{\mathfrak{p}+1}=1$.
 \item \label{lan6escape_expo} For all $j \in \cco 1,\mathfrak{p}+1\ccf$ and all $x\in \cco 1,\mathfrak{n}_j\ccf$, consider the escape time $\tau_x^j = \inf\{t>0 \mid  \sigma_n(t) \notin \mathcal E_x^j \cup \Delta^j\}$. For all $j \in \cco 1,\mathfrak{p}\ccf$,   $x\in \cco 1,\mathfrak{n}_j\ccf$, and all initial conditions $i\in\mathcal E_x^j$, $\tau_x^j/ \theta^j_n$ converges in law towards an exponential distribution with some parameter $r\geqslant 0$ (where $r=0$ means that $\tau_x^j  /\theta^j_n \rightarrow +\infty$ in probability). For $j=\mathfrak{p}+1$, for all $x\in \cco 1,\mathfrak{n}_{\mathfrak{p}+1}\ccf$ and all  initial conditions $i\in\mathcal E_x^{\mathfrak{p}+1}$, almost surely $\tau_x^{\mathfrak{p}+1} = +\infty$. 
\end{enumerate}
\end{theo}

\begin{remk}
\label{Remk:Landim}
To help clarify the notations introduced in Theorem~\ref{Thm:Landim}, its construction is illustrated in Figure~\ref{FigCM}.

Figure~\ref{FigCM}~(a) illustrates the first step of the construction, representing in blue the states of $\sigma_n$, i.e., the elements of $\cco 1,N\ccf$ identified with $\mathcal E_1^1,\dots,\mathcal E_N^1$. We only represent the fastest transitions of $\sigma_n$ (black arrows), i.e., the transition whose rates are of the same order as the total jump rate of $\sigma_n$, which we call $1/\theta_n^1$. This defines recurrence classes (in dashed lines) and transient states. We call $\Delta^2$ the set of transient states, and $\mathcal E_1^2,\mathcal E_2^2,\mathcal E_3^2,\mathcal E_4^2$ the recurrence classes. After rescaling the time by $\theta_n^1$, $\sigma_n$ converges to some Markov chain $X^1$ on $\cco 1,N\ccf$.

Figure~\ref{FigCM}~(b) represents the second step of the construction. At a timescale larger than $\theta_n^1$, the time spent in $\Delta^2$, the transient states of $X^1$, is negligible and $\sigma_n$ is averaged  within the recurrence classes, so that we can approximately consider that the ``macroscopic states'' are the recurrence classes $\mathcal E_1^2,\mathcal E_2^2,\mathcal E_3^2,\mathcal E_4^2$ (corresponding to the coarse-grained variable $\Psi^2(\sigma_n)$). We can now consider the first timescale $\theta_n^2$ at which transitions between these macroscopic states (the black arrows in Figure~\ref{FigCM}~(b)) occur. Remark that, during any such a transition, $\sigma_n$ may have to cross $\Delta^2$. These transitions define a new Markov chain over the macroscopic states. More precisely, after rescaling the time by $\theta_n^2$, $\Psi^2(\sigma_n)$ converges to some Markov chain $X^2$ on $\{1,2,3,4\}$. In this example, the
recurrence classes of $X^2$ are $\mathcal C^2_1=\{1,2\}$ and $\mathcal C_2^2=\{4\}$, corresponding to the sets  $\mathcal E_1^3 = \mathcal E_1^2 \cup \mathcal E_2^2$ and $\mathcal E_2^3 = \mathcal E_4^2$ (in dashed lines). The last macroscopic state, $\mathcal E_3^2$, is transient for $X^2$, so $\Delta^3 = \Delta^2 \cup \mathcal E_3^2$ is the set of states that are negligible for any timescale larger than $\theta_n^2$.

The third step of the construction is represented in Figure~\ref{FigCM}~(c). At a timescale larger than $\theta_n^2$, the time spent in $\Delta^3$ is negligible and the chain $\sigma_n$ is averaged within either $\mathcal E_1^3$ or $\mathcal E_2^3$. So, at this scale, there are two macroscopic states. The next timescale is given by the transitions between them (afterwards, only one class remains and the construction stops, in other words here $\mathfrak{p}=3$). Besides, we are not interested in these transitions if they occur at a timescale larger than $\mathcal O(1)$ since, in this case, they are not seen in the limit convexified process.
\end{remk}

\definecolor{myblue}{rgb}{0.4471, 0.6235, 0.8118}
\begin{figure}
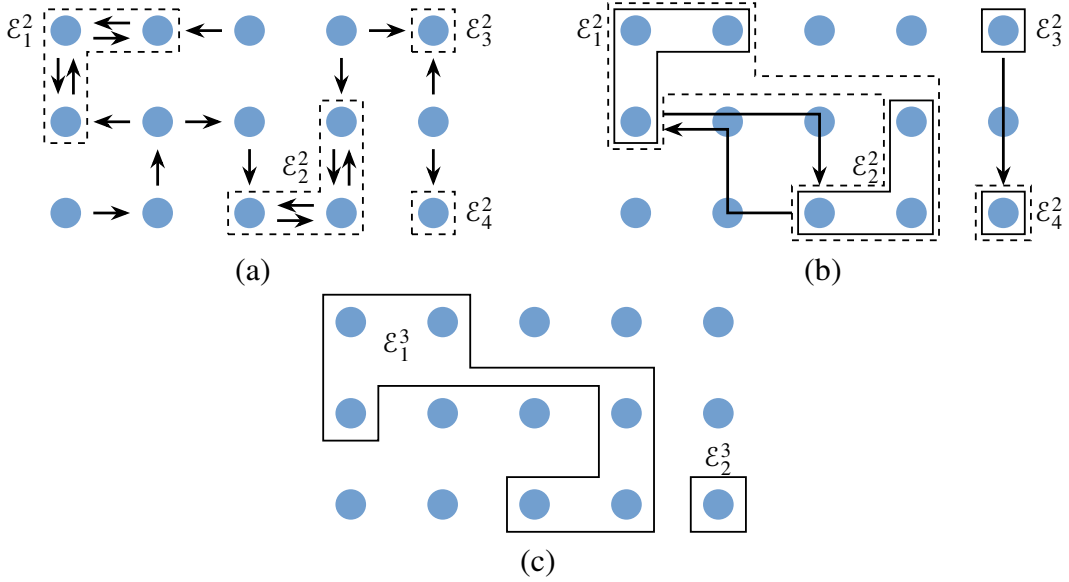
 
\centering
\begin{tabular}{@{} >{\centering} m{0.5\textwidth} @{} >{\centering} m{0.5\textwidth} @{}}
\resizebox{0.45\textwidth}{!}{\begin{tikzpicture}[x=40pt, y=40pt]
\input{fig_construction_states.tex}
\draw[very thick, -Stealth] (0.3, 0) -- (0.7, 0);
\draw[very thick, -Stealth] (1, 0.3) -- (1, 0.7);
\draw[very thick, -Stealth] (0.7, 1) -- (0.3, 1);
\draw[very thick, -Stealth] (1.3, 1) -- (1.7, 1);
\draw[very thick, -Stealth] ({1/12}, 1.3) -- ({1/12}, 1.7);
\draw[very thick, -Stealth] ({-1/12}, 1.7) -- ({-1/12}, 1.3);
\draw[very thick, -Stealth] (0.3, {2-1/12}) -- (0.7, {2-1/12});
\draw[very thick, -Stealth] (0.7, {2+1/12}) -- (0.3, {2+1/12});
\draw[very thick, -Stealth] (1.7, 2) -- (1.3, 2);
\draw[very thick, -Stealth] (2, 0.7) -- (2, 0.3);
\draw[very thick, -Stealth] (2.3, {-1/12}) -- (2.7, {-1/12});
\draw[very thick, -Stealth] (2.7, {1/12}) -- (2.3, {1/12});
\draw[very thick, -Stealth] ({3+1/12}, 0.3) -- ({3+1/12}, 0.7);
\draw[very thick, -Stealth] ({3-1/12}, 0.7) -- ({3-1/12}, 0.3);
\draw[very thick, -Stealth] (3, 1.7) -- (3, 1.3);
\draw[very thick, -Stealth] (3.3, 2) -- (3.7, 2);
\draw[very thick, -Stealth] (4, 1.3) -- (4, 1.7);
\draw[very thick, -Stealth] (4, 0.7) -- (4, 0.3);

\draw[thick, dashed] ({1+7/30}, {2+7/30}) -- ({-7/30}, {2+7/30}) -- ({-7/30}, {1-7/30}) -- ({7/30}, {1-7/30}) -- ({7/30}, {2-7/30}) -- ({1+7/30}, {2-7/30}) -- cycle;
\draw[thick, dashed] ({2-7/30}, {-7/30}) -- ({3+7/30}, {-7/30}) -- ({3+7/30}, {1+7/30}) -- ({3-7/30}, {1+7/30}) -- ({3-7/30}, {7/30}) -- ({2-7/30}, {7/30}) -- cycle;
\draw[thick, dashed] ({4-7/30}, {-7/30}) -- ({4-7/30}, {7/30}) -- ({4+7/30}, {7/30}) -- ({4+7/30}, {-7/30}) -- cycle;
\draw[thick, dashed] ({4-7/30}, {2-7/30}) -- ({4-7/30}, {2+7/30}) -- ({4+7/30}, {2+7/30}) -- ({4+7/30}, {2-7/30}) -- cycle;

\node at (-0.5, 2) {$\mathcal E_1^2$};
\node at (2.5, 0.5) {$\mathcal E_2^2$};
\node at (4.5, 2) {$\mathcal E_3^2$};
\node at (4.5, 0) {$\mathcal E_4^2$};
\end{tikzpicture}} & \resizebox{0.45\textwidth}{!}{\begin{tikzpicture}[x=40pt, y=40pt]
\input{fig_construction_states.tex}
\draw[very thick, -Stealth] (1.7, 0) -- (1, 0) -- (1, {1-1/12}) -- (0.3, {1-1/12});
\draw[very thick, -Stealth] (0.3, {1+1/12}) -- (2, {1+1/12}) -- (2, 0.3);
\draw[very thick, -Stealth] (4, 1.7) -- (4, 0.3);

\draw[thick] ({1+7/30}, {2+7/30}) -- ({-7/30}, {2+7/30}) -- ({-7/30}, {1-7/30}) -- ({7/30}, {1-7/30}) -- ({7/30}, {2-7/30}) -- ({1+7/30}, {2-7/30}) -- cycle;
\draw[thick] ({2-7/30}, {-7/30}) -- ({3+7/30}, {-7/30}) -- ({3+7/30}, {1+7/30}) -- ({3-7/30}, {1+7/30}) -- ({3-7/30}, {7/30}) -- ({2-7/30}, {7/30}) -- cycle;
\draw[thick] ({4-7/30}, {-7/30}) -- ({4-7/30}, {7/30}) -- ({4+7/30}, {7/30}) -- ({4+7/30}, {-7/30}) -- cycle;
\draw[thick] ({4-7/30}, {2-7/30}) -- ({4-7/30}, {2+7/30}) -- ({4+7/30}, {2+7/30}) -- ({4+7/30}, {2-7/30}) -- cycle;

\draw[thick, dashed] (1.7, -0.3) -- (3.3, -0.3) -- (3.3, 1.5) -- (1.3, 1.5) -- (1.3, 2.3) -- (-0.3, 2.3) -- (-0.3, 0.7) -- (0.3, 0.7) -- (0.3, 1.3) -- (2.7, 1.3) -- (2.7, 0.3) -- (1.7, 0.3) -- cycle;
\draw[thick, dashed] (3.7, -0.3) -- (3.7, 0.3) -- (4.3, 0.3) -- (4.3, -0.3) -- cycle;

\node at (-0.5, 2) {$\mathcal E_1^2$};
\node at (2.5, 0.5) {$\mathcal E_2^2$};
\node at (4.5, 2) {$\mathcal E_3^2$};
\node at (4.5, 0) {$\mathcal E_4^2$};
\end{tikzpicture}} \tabularnewline
(a) & (b) \tabularnewline
\end{tabular}
\resizebox{0.45\textwidth}{!}{\begin{tikzpicture}[x=40pt, y=40pt]
\input{fig_construction_states.tex}
\draw[thick] (1.7, -0.3) -- (3.3, -0.3) -- (3.3, 1.5) -- (1.3, 1.5) -- (1.3, 2.3) -- (-0.3, 2.3) -- (-0.3, 0.7) -- (0.3, 0.7) -- (0.3, 1.3) -- (2.7, 1.3) -- (2.7, 0.3) -- (1.7, 0.3) -- cycle;
\draw[thick] (3.7, -0.3) -- (3.7, 0.3) -- (4.3, 0.3) -- (4.3, -0.3) -- cycle;

\node at (-0.5, 2) {\hphantom{$\mathcal E_1^2$}}; 
\node at (4.5, 2) {\hphantom{$\mathcal E_3^2$}}; 

\node at (0.5, 1.75) {$\mathcal E_1^3$};
\node at (4, 0.5) {$\mathcal E_2^3$};
\end{tikzpicture}}

(c)
\caption{Construction of Theorem~\ref{Thm:Landim}, described in Remark~\ref{Remk:Landim}.}
\label{FigCM}
\end{figure}

From now on, we suppose that the sequence $(\nu_n,\mu_n,P_n)_{n\in\N}$  is such that $\mu_n\rightarrow +\infty$, that $P_n$ is strongly connected for all $n\in\N$, and that \eqref{cond:lambda_n} and \eqref{Eq:CondOrdered} hold, so that Theorem~\ref{Thm:Landim} holds. Up to extracting a subsequence, we assume that for all $j\in\cco 1,\mathfrak p\ccf$, the sequence $\theta^j$ is monotone (and in particular admits a limit in $[0,+\infty]$). Since we assumed that $\mu_n\rightarrow +\infty$, necessarily $\theta^1_n \rightarrow 0$. Let $h = \max\{j\in\cco 1,\mathfrak{p}\ccf \mid  \theta^j_n \rightarrow 0\}$, so that $\theta^h$ is the slowest of all the fast scales of the chain. We now have to distinguish whether there are slow transitions (i.e., occurring at a time of order $1$ with respect to $n$) or not.
\begin{itemize}
\item \textbf{Case 1.} If $h=\mathfrak{p}$ then by   Theorem~\ref{Thm:Landim}~\ref{lan3_Pastransition} , for any sequence $\alpha_n \simeq 1 \gg \theta^{\mathfrak{p}}_n$, for all $t>0$, $\Psi^{h+1}( \sigma_n(t\alpha_n))$ converges in law to the value at time $t$ of the constant Markov chain on $\{1\}$ (since, by Theorem~\ref{Thm:Landim}\ref{lan5coarser}, $\mathfrak{n}_{\mathfrak{p}+1} = 1$).
\item \textbf{Case 2.} If $h< \mathfrak{p}$ and $\theta^{h+1}_n$ does not converge to $+\infty$ as $n\rightarrow +\infty$, then it converges to some $\theta_*>0$. In particular, for all $\alpha_n\simeq 1$ and $t>0$, $\Psi^{h+1}(\sigma_n(t\alpha_n))$ converges in law towards  $X^{h+1}(t/  \theta_*)$. 
\item \textbf{Case 3.} If $h<\mathfrak{p}$ and $\theta^{h+1}_n\rightarrow +\infty$ as $n\rightarrow +\infty$ then, as in the first case, by  Theorem~\ref{Thm:Landim}~\ref{lan3_Pastransition}, for any sequence $(\alpha_n)_{n\in\N}$ with $\theta_n^h \ll \alpha_n \simeq 1 \ll \theta^{h+1}_n$, for all $t>0$, $\Psi^{h+1}( \sigma_n(t\alpha_n))$ converges in law to the value at time $t$ of the constant Markov chain on $\cco 1,\mathfrak{n}_{h+1}\ccf$.
\end{itemize}

Summarizing the above three cases, there is a Markov chain  on $\cco 1,\mathfrak{n}_{h+1}\ccf$, that we denote $(\tilde X^{h+1}(t))_{t\geqslant 0}$, such that, for all sequence $\alpha_n\simeq 1$,  $t>0$, $x,y\in\cco 1,\mathfrak{n}_{h+1}\ccf$ and $i\in\mathcal E_x^{h+1}$,
\[\mathbb P_i \po \Psi^{h+1}\po \sigma_n(t\alpha_n)\pf =y\pf \ \underset{n\rightarrow +\infty}\longrightarrow \ \mathbb P_x\po \tilde X^{h+1}(t) =y\pf.
\]
Reasoning similarly, we can ensure, in addition, that for all $\delta>0$,
\begin{equation}\label{Eq:CVintegraleXtilde}
\mathbb P_i \po \left| \int_0^t \po \1_{\Psi^{h+1}(\sigma_n(s\alpha_n)) = y} - \1_{  \tilde X^{h+1}(s) = y}\pf   \dd s\right| > \delta\pf  \  \underset{n\rightarrow +\infty}\longrightarrow  \ 0\,.
\end{equation}
Moreover, from Theorem~\ref{Thm:Landim}~\ref{lan4Deltanegligeable}, for all sequences $\alpha_n \simeq 1$ and $t\geqslant 0$,
\begin{equation}\label{Eq:Negligeable}
\max_{i\in\cco 1,N\ccf} \mathbb E_i \po \int_0^t \1_{\sigma_n(s \alpha_n)\in \Delta^{h+1}} \dd s\pf \ \underset{n\rightarrow +\infty}\longrightarrow \ 0\,.
\end{equation}

By analogy with the proof of Proposition~\ref{Prop:CVversAveraged}, the transitions of $\tilde X^{h+1}$ will play the role of the slow transition (at rate $\mu$ in  Proposition~\ref{Prop:CVversAveraged}). It remains to identify a Markov chain that plays the same role as $\tilde \sigma_n^b$ in the proof of Proposition~\ref{Prop:CVversAveraged}, namely a Markov chain that is equal to $\sigma_n$ up to the first slow transition and that mixes fast within the sets $\mathcal E_x^{h+1}$, $x\in\cco 1,\mathfrak{n}_{h+1}\ccf$. We could consider the trace process of $\sigma_n$ on $\mathcal E_x^{h+1}$ (see Section~\ref{Sec:MarkovDecomposition} for the definition), but for simplicity (in order to avoid the question of time change) we will consider another process.
 
For all $z\in \cco 1,\mathfrak{n}_{h+1}\ccf$, fix some $i_z \in \mathcal E_z^{h+1}$ (for instance, $i_z = \min \mathcal E_z^{h+1}$). We denote by $\tilde \sigma_n^z$ the Markov chain on $\Ebz:= \mathcal E_z^{h+1} \cup \Delta^{h+1}$ with jump rates $\tilde \lambda_n$ defined as follows:
\begin{itemize}
\item $\tilde \lambda_n(i,j) = \lambda_n(i,j)$ for all $i,j\in\Ebz$ with $j\neq i_z$;
\item $\tilde \lambda_n(i,i_z) = \lambda_n(i,i_z) + \sum_{k\notin \Ebz} \lambda_n(i,k)$ for all $i\in \Ebz$.
\end{itemize}
In other words, starting at a point of $\Ebz$, $\tilde \sigma_n^z$ has the same transitions as $\sigma_n$ except that all transitions that would leave $\Ebz$ are replaced by a transition to $i_z$. By  Theorem~\ref{Thm:Landim}~\ref{lan6escape_expo}, from an initial condition in $\mathcal E^{h+1}_z$, the first transition of $\sigma_n$ out of $\Ebz$ occurs at a rate at most of order $1$ (in $n$ as $n\to\infty$). On the other hand, let us prove that the mixing time of $\tilde \sigma_n^z$ is of order $\theta_n^h$ (in the sense of Equation~\eqref{eq:mixtime} below). Denote by $(P_t^{n,z})_{t\geqslant 0}$ the semigroup associated with $\tilde \sigma_n^z$, i.e., $P_t^{n,z}(i,j) = \mathbb P(\tilde \sigma_n^z(t)=j|\tilde \sigma_n^z(0)=i) $ for all $i,j\in\Ebz$.
 
\begin{prop}
\label{prop:semigroup-pi}
There exist $\rho,C>0$ such that for all $z\in \cco 1,\mathfrak{n}_{h+1}\ccf$ and all $n \in\N$, $(P_t^{n,z})_{t\geqslant 0}$ admits a unique invariant probability measure $\pi_n^z$ on $\Ebz$ and, for all $t\geqslant 0$ and all $i\in \Ebz$,
\begin{equation}\label{eq:mixtime}
|P_t^{n,z}(i,\cdot) - \pi_n^z| \ \leqslant \ C e^{-\rho t/\theta_n^h}\,.
\end{equation}
\end{prop}
\begin{proof}
The statement does not depend on the norm used in the left hand side of \eqref{eq:mixtime}. In the following, we consider the total variation distance $\abs{x} = \sum_{i=1}^N \abs{x_i}$.
 
Since there is a finite number of values $z\in \cco 1,\mathfrak{n}_{h+1}\ccf$, it is sufficient to prove the result for each one, so let $z$ be fixed. 

Assuming that $P_{4\theta_n^h}^{n,z}$ satisfies the Doeblin condition 
\begin{eqnarray}\label{Eq:DoeblinFast}
\liminf_{n\rightarrow +\infty} \min_{i\in\Ebz} P_{4\theta_n^h}^{n,z}(i,i_z) & > &0\,,
\end{eqnarray}
we deduce the conclusion of the proposition by the following classical argument. By \eqref{Eq:DoeblinFast}, $P^{n,z}$ has a unique recurrence class, since all points can go to $i_z$, and thus a unique invariant measure $\pi_n^z$ (supported by this recurrence class). Let $\alpha\in (0,1)$ be such that, for $n$ large enough,
 \[\min_{i\in\Ebz}   P_{4\theta_n^h}^{n,z}(i,i_z) \geqslant \ \alpha\,.\]
For the the total variation, we have (see e.g.\ \cite[Theorem~S7]{DurmusGuillinMonmarche2} with $V=1$)
\[\sup_{i\in \Ebz} |P_{t}^{n,z}(i,\cdot) - \pi_n^z| \ \leqslant \ 2 (1-\alpha)^{\lfloor t/(4\theta_n^h) \rfloor} \ \leqslant \  \frac{2}{1-\alpha}e^{t \ln(1-\alpha) /(4\theta_n^h)  }\,,\]
which concludes the proof of \eqref{eq:mixtime}.

We are left to prove that \eqref{Eq:DoeblinFast} holds true and for that purpose we introduce some useful notations. In the sequel we say that a sequence $(a_n)_{n\in\N}$ is nonvanishing if $\liminf_{n\rightarrow +\infty} a_n \allowbreak >0$. Let $x_1,\dots,x_{h+1}$ be such that
\[i_z \in \mathcal E_{x_1}^1 \subset \dots \subset \mathcal E_{x_{h+1}}^{h+1} \,.\]
In other words, $x_1= i_z$, $x_{h+1}=z$, and for all $j\in\cco 1,h+1\ccf$ and $n\in \N$, $x_j=\Psi^j(i_z)$. Denote by $(P_t^n)_{t\geqslant 0}$ the semigroup associated with $\sigma_n$. Also, set $\theta_n^0 = 0$ for all $n\in \N$. The strategy of the proof of \eqref{Eq:DoeblinFast} consists in finding a path from any $i \in\Ebz$ to $i_z$ that has a nonvanishing probability. An informal description of such a path is given in Figure~\ref{FigChemin}.
 
\definecolor{bgblue}{rgb}{0.3490, 0.4353, 0.8941}
\definecolor{bggreen}{rgb}{0.3490, 0.8941, 0.4353}
\definecolor{bgyellow}{rgb}{0.9451, 0.8863, 0.2196}
\definecolor{nodepurple}{rgb}{0.3608, 0.1765, 0.5686}
\definecolor{nodepink}{rgb}{0.9686, 0.6314, 0.6039}
\begin{figure}
\centering
\resizebox{0.67\textwidth}{!}{\begin{tikzpicture}[x=40pt, y=40pt]

\draw[fill=bgblue] ({-11/30}, {-11/30}) -- ({2+11/30}, {-11/30}) -- ({2+11/30}, {2+11/30}) -- ({-11/30}, {2+11/30}) -- cycle;
\draw[fill=bgblue] ({4-11/30}, {-11/30}) -- ({5+11/30}, {-11/30}) -- ({5+11/30}, {11/30}) -- ({4-11/30}, {11/30}) -- cycle;
\draw[fill=bgblue] ({4-11/30}, {1-11/30}) -- ({5+11/30}, {1-11/30}) -- ({5+11/30}, {2+11/30}) -- ({4-11/30}, {2+11/30}) -- cycle;

\draw[fill=bggreen] ({-9/30}, {-9/30}) -- ({1+9/30}, {-9/30}) -- ({1+9/30}, {9/30}) -- ({-9/30}, {9/30}) -- cycle;
\draw[fill=bggreen] ({-9/30}, {1-9/30}) -- ({1+9/30}, {1-9/30}) -- ({1+9/30}, {2+9/30}) -- ({-9/30}, {2+9/30}) -- cycle;
\draw[fill=bggreen] ({2-9/30}, {-9/30}) -- ({2+9/30}, {-9/30}) -- ({2+9/30}, {2+9/30}) -- ({2-9/30}, {2+9/30}) -- cycle;
\draw[fill=bggreen] ({4-9/30}, {-9/30}) -- ({4+9/30}, {-9/30}) -- ({4+9/30}, {9/30}) -- ({4-9/30}, {9/30}) -- cycle;
\draw[fill=bggreen] ({5-9/30}, {-9/30}) -- ({5+9/30}, {-9/30}) -- ({5+9/30}, {9/30}) -- ({5-9/30}, {9/30}) -- cycle;
\draw[fill=bggreen] ({4-9/30}, {1-9/30}) -- ({5+9/30}, {1-9/30}) -- ({5+9/30}, {2+9/30}) -- ({4-9/30}, {2+9/30}) -- cycle;

\draw[fill=bgyellow] ({-7/30}, {-7/30}) -- ({1+7/30}, {-7/30}) -- ({1+7/30}, {7/30}) -- ({-7/30}, {7/30}) -- cycle;
\draw[fill=bgyellow] ({-7/30}, {1-7/30}) -- ({1+7/30}, {1-7/30}) -- ({1+7/30}, {1+7/30}) -- ({-7/30}, {1+7/30}) -- cycle;
\draw[fill=bgyellow] ({-7/30}, {2-7/30}) -- ({1+7/30}, {2-7/30}) -- ({1+7/30}, {2+7/30}) -- ({-7/30}, {2+7/30}) -- cycle;
\draw[fill=bgyellow] ({2-7/30}, {-7/30}) -- ({2+7/30}, {-7/30}) -- ({2+7/30}, {7/30}) -- ({2-7/30}, {7/30}) -- cycle;
\draw[fill=bgyellow] ({4-7/30}, {-7/30}) -- ({4+7/30}, {-7/30}) -- ({4+7/30}, {7/30}) -- ({4-7/30}, {7/30}) -- cycle;
\draw[fill=bgyellow] ({5-7/30}, {-7/30}) -- ({5+7/30}, {-7/30}) -- ({5+7/30}, {7/30}) -- ({5-7/30}, {7/30}) -- cycle;
\draw[fill=bgyellow] ({2-7/30}, {1-7/30}) -- ({2+7/30}, {1-7/30}) -- ({2+7/30}, {2+7/30}) -- ({2-7/30}, {2+7/30}) -- cycle;
\draw[fill=bgyellow] ({4-7/30}, {1-7/30}) -- ({5+7/30}, {1-7/30}) -- ({5+7/30}, {2+7/30}) -- ({4-7/30}, {2+7/30}) -- cycle;

\input{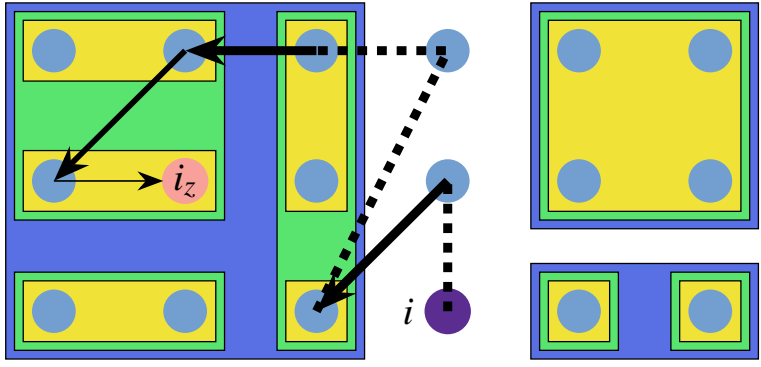}
\foreach \y in {0, ..., 2} {
	\fill[myblue] (5, \y) circle[radius={1/6}];
}
\fill[nodepurple] (3, 0) circle[radius={1/6+0.01}];
\fill[nodepink] (1, 1) circle[radius={1/6+0.01}];

\draw[line width=2.75, dashed] (3, 0) -- (3, 1);
\draw[line width=2.75, {-Stealth[scale=0.8]}] (3, 1) -- (2, 0);
\draw[line width=2.75, dashed] (2, 0) -- (3, 2) -- (2, 2);
\draw[line width=2.75, {-Stealth[scale=0.8]}] (2, 2) -- (1, 2);
\draw[line width=1.75, -Stealth] (1, 2) -- (0, 1);
\draw[line width=0.5, {-Stealth[scale=1.5]}] (0, 1) -- ({1-1/6}, 1);

\node at (2.7, 0) {$i$};
\node at (1, 1) {$i_z$};
\end{tikzpicture}}
\caption{At the timescale $\theta^1$, the state $i_z$, which is also $\mathcal E_{i_z}^1$, is averaged in some recurrence class $\mathcal E_{x_2}^2$ (in yellow). Then, at the slower timescale $\theta^2$, this class is itself averaged with other yellow classes, in a coarser $\mathcal E_{x_3}^3$ (in green) which is itself averaged at the slower timescale $\theta^3$ with some other green classes to give a coarser $\mathcal E_{x_4}^4$ (in blue). Here, $h=3$, i.e., there is no more fast averaging (there may be slow transitions  at order $1$ between several blue classes, but we do not consider them here). Now, starting from some point $i$ (in $\Delta^4$ here), after a time $\theta^3$, the chain  $\sigma_n$ has a positive (nonvanishing) probability to be in one of the blue classes (and thus $\tilde \sigma_n^z$ to be in the blue class of $i_z$, since transitions to other blue classes are replaced by a transition to $i_z$; and after reaching this blue class the transitions of $\sigma_n$ and $\tilde \sigma_n^z$ are similar in the timescale $\theta^3$ since a transition to a different blue class is unlikely). Then, starting in the blue class of $i_z$, i.e., in some green class within $\mathcal E_{x_4}^4$,   the chain has a positive probability to be in $\mathcal E_{x_3}^3$,  the green class of $i_z$, in a time $\theta^3$ (or equivalently $\theta^3 - \theta^2$) as this is the timescale of motion between green classes and all green classes within a given blue class are linked at this timescale (by definition of the blue classes as recurrence classes). Similarly, from there, the chain has a positive probability to be in $\mathcal E_{x_2}^2$, the yellow class of $i_z$, in a time $\theta^2$ (or $\theta^2-\theta^1$). Finally, starting from any point in $\mathcal E_{x_2}^2$, it has a positive probability to be at $i_z$ in a time $\theta^1$. The arrows represent the transitions (the timescale being represented by the thickness of the arrow). The dotted lines are here to remind that a ``macroscopic transition'' happening between two classes (or from the point $i\in\Delta^4$ to some point of the blue class $\mathcal E_{x_4}^4$) may be constituted of many (possibly fastest) ``microscopic'' transitions (that may cross  $\Delta^4$, which explains why $\tilde \sigma_n^z$ is defined on $\overline{\mathcal E}_{x_4}^4$ and not on  $ \mathcal E_{x_4}^4$).}\label{FigChemin}
\end{figure}

We decompose the argument in several steps. 
 
\medskip

\noindent\textbf{Step 1.} Let us prove that, from any initial condition in $\Ebz$, the probability that $\tilde\sigma^z_n$ hits $\mathcal E_z^{h+1}$ before the time $\theta^h_n$ is nonvanishing. More precisely denote 
\[\tau_* = \inf\{t\geqslant 0 \mid \sigma_n(t) \notin \Delta^{h+1}\} = \inf\{t\geqslant 0\mid \tilde \sigma_n^z(t) \in\mathcal E_z^{h+1}\} \,,\]
where $\sigma_n$ and $\tilde \sigma^z_n$ are synchronously coupled (as in the proof of Proposition~\ref{Prop:CV-mu-borne}, see \cite[Section~6]{DurmusGuillinMonmarche} for details) with the same initial condition. It is clear that, for all  initial conditions $i\in\Delta^{h+1}$,
\begin{equation}\label{eq:tau*}
\liminf_{n\rightarrow +\infty} \mathbb P_i(\tau_* \leqslant \theta_n^h) >0\,. 
\end{equation}
Indeed, the time spent in $\Delta^{h}$ is negligible at the timescale $\theta_n^h$ (Theorem~\ref{Thm:Landim}~\ref{lan4Deltanegligeable}), so the probability for $\sigma_n$ to leave $\Delta^h$ before time $\theta^h_n/2$ goes to $1$ as $n\rightarrow +\infty$. Then, starting from a point   $i'\in \Delta^{h+1}\setminus \Delta^h$ (which corresponds to the transient points of $X^h$),  we can use Theorem~\ref{Thm:Landim}~\ref{lan2_CVMarkov} to see that
\[\mathbb P_{i'} \po \tau_* \leqslant \theta_n^h/2\pf \geqslant \mathbb P_{i'} \po \Psi^h(\sigma_n(\theta^h_n/2)) \notin \mathcal T_h\cup\{0\}\pf \underset{n\rightarrow +\infty}\longrightarrow \mathbb P_{\Psi^h(i')} \po X^h(1/2) \notin \mathcal T_h\pf >0\]
(where we recall that, by definition, $\Psi^h(\sigma)=0$ means that $\sigma \in \Delta^h$). The strong Markov property concludes the proof of \eqref{eq:tau*}. 
 
\medskip
  
\noindent  \textbf{Step 2.} Starting from an initial condition  $i\in \mathcal E_z^{h+1}$, we consider again the synchronous coupling of $\sigma_n$ and $\tilde \sigma_n^z$. The two chains are then equal up to the first time at which $\sigma_n$ hits $\mathcal E_y^{h+1}$ for some $y\neq z$. According to Theorem~\ref{Thm:Landim}~\ref{lan6escape_expo}, this occurs at a time of order at least 1. In other words, for any $M>0$, the probability that this happens before time $M\theta^h_n$ goes to $0$ as $n\rightarrow +\infty$. Hence, for all $M>0$,
\[  \sup_{s\in [0,M\theta_n^{h}]}  | P_{s}^{n}(i,u)-P_s^{n,z}(i,u)| \ \leqslant \    \mathbb P_i \po \exists s\in [0,M\theta_n^{h}],\ \sigma_n(s)\neq \tilde\sigma_n^z(s) \pf \ \underset{n\rightarrow +\infty}\longrightarrow \ 0\,.\]
For this reason, we can focus on $\sigma_n$ rather than $\tilde\sigma^z_n$.

\medskip
  
\noindent\textbf{Step 3.} Let us prove that, for all $j\in\cco 1,h\ccf$,
\begin{eqnarray}\label{Eq:DoeblinLocal}
\liminf_{n\rightarrow +\infty} \min_{i\in\mathcal E_{x_{j+1}}^{j+1}} \sum_{u\in \mathcal E_{x_{j}}^{j}}  P_{\theta_n^j-\theta_n^{j-1}}^{n}(i,u) &>& 0\,. 
\end{eqnarray}
Indeed, for all $s>0$, $n\in\N$,
\[\sum_{u\in \mathcal E_{x_{j}}^{j}}  P_{s}^{n}(i,u)  \ = \ \mathbb P_i\po \Psi^j(\sigma_n(s)) = x_j\pf \,.\]
 Since $\theta_n^j - \theta_n^{j-1} \simeq \theta_n^j$,   Theorem~\ref{Thm:Landim}~\ref{lan2_CVMarkov} implies that for all $y\in\cco 1,\mathfrak n_j\ccf$ and $i\in\mathcal E_y^j$,
\[\sum_{u\in \mathcal E_{x_{j}}^{j}}  P_{\theta_n^j-\theta_n^{j-1}}^{n}(i,u) \underset{n\rightarrow +\infty}\longrightarrow \mathbb P_y\po X^j(1) = x_j\pf\,.\]
Recall from  Theorem~\ref{Thm:Landim}~\ref{lan5coarser} that  $\mathcal E_{x_{j+1}}^{j+1} = \bigcup_{y\in \mathcal C_{x_{j+1}}^{j}} \mathcal E_y^{j}$, where $\mathcal C_{x_{j+1}}^{j}$ is the  $x_{j+1}$-th recurrence class of the limit chain $X^{j}$. In particular, $x_{j} \in \mathcal C_{x_{j+1}}^{j}$ and, by definition of a recurrence class, for all $x,y\in   \mathcal C_{x_{j+1}}^{j}$, $\mathbb P_{x}(X^j(1)=y)>0$. This concludes the proof of \eqref{Eq:DoeblinLocal}.

\medskip
  
\noindent\textbf{Step 4.} We now deduce from \eqref{Eq:DoeblinLocal} that, for all $j\in\cco 1,h\ccf$,
\begin{eqnarray}\label{Eq:DoeblinLocal2}
\liminf_{n\rightarrow +\infty} \min_{i\in\mathcal E_{x_{j+1}}^{j+1}}    P_{\theta_n^j}^{n}(i,i_z) &>& 0\,. 
\end{eqnarray}
Indeed, reasoning by induction, the case $j=1$ is just \eqref{Eq:DoeblinLocal} with $j=1$ since $\mathcal E_{x_1}^1 = \{i_z\}$ and $\theta_n^0 = 0$. Then, using the Markov property, for $i\in\mathcal E_{x_{j+1}}^{j+1}$,
\[P_{\theta_n^j}^{n}(i,i_z) \ \geqslant \ \sum_{u\in \mathcal E_{x_{j}}^{j}}  P_{\theta_n^j-\theta_n^{j-1}}^{n}(i,u) P_{\theta_n^{j-1}}^n(u,i_z)\,,\]
and the conclusion follows from \eqref{Eq:DoeblinLocal} and, by induction, \eqref{Eq:DoeblinLocal2} for $j-1$.

\medskip

\noindent\textbf{Step 5.} Let us prove by induction on $j$ that, for all $j\in\cco 1,h\ccf$, 
\begin{equation}\label{eq:cvtildeinf}
\forall M>m>1\,,\qquad \liminf_{n\rightarrow +\infty}  \inf_{s\in[m,M]}     P_{s\theta_n^j}^{n}(i_z,i_z)  \ >\ 0\,. 
\end{equation}
For $j=1$, this follows from the fact that $1/\theta^1_n$ is of the same order as the total jump rate $\mu_n$ of the chain, so that for all $M>0$ there is a nonvanishing probability that no jump occurs in $[0,M\theta_n^1]$, and then
\[ \inf_{s\in[0,M]}     P_{s\theta_n^j}^{n}(i_z,i_z)  \ \geqslant \ \mathbb P_{i_z} \po \sigma_n(t) = i_z\ \forall t\in[0,M\theta_n^1]\pf\,.\]
Now, suppose that \eqref{eq:cvtildeinf} is true for $j-1$ for some $j\in\cco 2,h\ccf$. Fix $M>m>1$. For $M'>0$ to be chosen later on, consider the events
\begin{align*}
E_n &= \left\{\sigma_n(t) \in \mathcal E_{x_j}^j \cup \Delta^j \ \forall t\in[0,M\theta^j_n]\right\}, \qquad 
F_n = \left\{\tau' \leqslant M' \theta^{j-1}\right\},  
\end{align*}
where $\tau' = \inf\{t\geqslant 0\mid \sigma_n(t) \notin \Delta^j\}$. From Theorem~\ref{Thm:Landim}~\ref{lan6escape_expo},
\[\varepsilon \ := \ \liminf_{n\rightarrow +\infty} \mathbb P_{i_z}(E_n) \ > \ 0\,.\]
Besides, from Theorem~\ref{Thm:Landim}~\ref{lan2_CVMarkov} and \ref{lan5coarser}, for all $i\in \Delta^j \setminus \Delta^{j-1}$,
\begin{align*}
\liminf_{n\rightarrow +\infty} \mathbb P_i \po \tau' \leqslant M' \theta_n^{j-1} \pf & \geqslant\ \liminf_{n\rightarrow +\infty} \mathbb P_i \po \sigma_n(M'\theta_n^{j-1}) \notin \Delta^j\pf \\
&\geqslant \ {\min_{y\in\mathcal T^j}\mathbb P_y \po   X^j (M') \notin \mathcal T^j\pf\ } \underset{M'\rightarrow +\infty}\longrightarrow 1\,,
\end{align*}
and from Theorem~\ref{Thm:Landim}~\ref{lan4Deltanegligeable}, for all $i\in \Delta^{j-1}$,
\[\mathbb P_i\po \sigma_n(t) \in \Delta^{j-1} \ \forall t\in [0,\theta_n^{j-1}]\pf \ \underset{n\rightarrow +\infty}\longrightarrow 0\,, \]
so that, combining these two facts thanks to the Markov property and denoting $q_n =\min_{i\in\cco 1,N\ccf}\mathbb P_i \po \tau' \leqslant M' \theta_n^{j-1} \pf$, we get
\[ \ \liminf_{n\rightarrow +\infty} q_n \underset{M'\rightarrow +\infty}\longrightarrow 1\,.\]
From now on, we choose $M'$ large enough so that $q_n \geqslant 1-\varepsilon/2$ for all $n \geqslant n_0$ for some sufficiently large $n_0$. We also suppose that $n_0$ is large enough so that $(m-1)\theta_n^j - (2+M')\theta_n^{j-1}>0$ for all $n\geqslant n_0$, and implicitly up to the end of Step 4 we always assume $n\geqslant n_0$.

For $s\in [m,M]$ and all $n\geqslant n_0$  consider the time interval $I_n(s) = [(s-1)\theta_n^j - (2+M')\theta_n^{j-1},(s-1)\theta_n^j-2\theta_n^{j-1}]$ and the event
\[F_n(s) \  = \ \{\exists u\in I_n(s)\mid \sigma_n(u) \notin \Delta^j\}\,.\]
By the Markov property, $\mathbb P_i(F_n(s)) \geqslant q_n$ for all $i\in\cco 1,N\ccf$ and $s\in[m,M]$. Moreover, the event $F_n(s) \cap E_n$, which has a probability at least $\varepsilon/3$ for $n$ large enough, implies that there exists $u\in I_n(s)$ such that $\sigma_n(u) \in \mathcal E_{x_j}^j$. Conditioning with respect to $u$ and $\sigma_n(u)$ for such a time $u$, we get by the Markov property that, for all $s\in[m,M]$,
\[P_{s\theta_n^j}^{n}(i_z,i_z) \ \geqslant \ \mathbb P_{i_z}\po F_n(s) \cap E_n\pf \min_{i\in \mathcal E_{x_j}^j} P_{\theta_n^j}^n(i,i_z) \inf_{s'\in [2,2+M'] }   P_{s'\theta_n^{j-1}}^n (i_z,i_z) \,.\]
Thanks to the induction hypothesis, \eqref{Eq:DoeblinLocal2}, and the bound on $q_n$, we see that the three factors of the right-hand side are bounded from below by positive quantities independent of $s\in [m,M]$, which concludes the proof of \eqref{eq:cvtildeinf} for all $j\in\cco 1,h\ccf$.

\medskip

\noindent \textbf{Step 6.} We now prove the Doeblin condition \eqref{Eq:DoeblinFast}. From Step 2, \eqref{Eq:DoeblinLocal2} and \eqref{eq:cvtildeinf} are still true if $P^n$ is replaced by $P^{n, z}$. According to Step 1, from any initial condition $i\in\cco 1,N\ccf$, $\tilde \sigma_n^z$ has a nonvanishing probability to hit $\mathcal E_z^{h+1}$ before time $\theta_n^h$. Using the Markov property and \eqref{Eq:DoeblinLocal2} (with $P^n$ replaced by $P^{n,z}$), conditioning with respect to this hitting time $u$, the chain has a nonvanishing probability (independent of $u$) to be at $i_z$ at time $u+\theta_n^h$. Then, from \eqref{eq:cvtildeinf} (with $P^n$ replaced by $  P^{n,z}$), it has a nonvanishing probability (still independent of $u$) to be at $i_z$ at time $4\theta_n^z$, which concludes.
\end{proof}
 
Before stating the main result of the section, let us prove an additional technical property. 

\begin{prop}\label{Prop:mix}
For all $T>0$, all $z\in \cco 1,\mathfrak{n}_{h+1}\ccf$, and all $j \in \mathcal E_z^{h+1}$,
\[\max_{i\in \cco 1,N\ccf}\mathbb E_i \po \left| \int_0^T \po \1_{\sigma_n(s) = j} - \pi_n^z(j)\1_{\sigma_n(s)\in\mathcal E^{h+1}_z} \pf \dd s \right|  \pf \ \underset{n\rightarrow +\infty} \longrightarrow \ 0\,.\]
\end{prop}
\begin{proof}
Fix $T>0$,  $z\in \cco 1,\mathfrak{n}_{h+1}\ccf$, and   $j \in \mathcal E_z^{h+1}$. Let $R \in \mathbb N$, and notice that
\[
\left| \int_{0}^{T} \po \1_{\sigma_n(s)=j}  -   \pi_n^z(j)\1_{\sigma_n(s)\in\mathcal E^{h+1}_z}\pf\dd s\right| \leqslant \sum_{b=0}^{R-1 }  \left| \int_{bT/R}^{(b+1)T/R} \po \1_{\sigma_n(s)=j}  -   \pi_n^z(j)\1_{\sigma_n(s)\in\mathcal E^{h+1}_z}\pf\dd s\right|.
\]
For all $b\in\cco 0,R-1\ccf$, consider the stopping time $\tau_b = \inf\{t\geqslant bT/R \mid  \sigma_n(t) \notin \Delta^{h+1}\}$ and $Z_b = \Psi^{h+1}(\sigma_n(\tau_b))$. Let $\tau_b^+ = \inf\{t\geqslant \tau_b \mid  \sigma_n(t) \in\bigcup_{y\neq Z_b}\mathcal E^{h+1}_y\}$. According to  Theorem~\ref{Thm:Landim}~\ref{lan6escape_expo}, for all $b\in\cco 0,R-1\ccf$, $\tau_b^+ - \tau_b$ converges as $n\rightarrow +\infty$ towards an exponential distribution with some finite rate (possibly $0$). Thus there exist $K>0$ such that
\[\max_{i\in \cco 1,N\ccf}\mathbb P_i\po \tau_b^+ - \tau_b \leqslant T/R\pf \ \leqslant \ \frac{K}{R}\]
for all $n$ large enough and all  $R \in \mathbb N$. 

For each $b\in\cco 0,R-1\ccf$, consider a Markov chain $\tilde \sigma_n^{Z_b}$ and its associated semigroup $P^{n,Z_b}$, initialized at time $\tau_b$ by $\tilde \sigma_n^{Z_b}(\tau_b) = \sigma_n(\tau_b)$ and such that $\tilde \sigma_n^{Z_b}(s) = \sigma_n(s)$ for all $s\in [\tau_b,\tau_b^+)$. We have
\begin{multline*}
 \left| \int_{bT/R}^{(b+1)T/R} \po \1_{\sigma_n(s)=j}  -   \pi_n^z(j)\1_{\sigma_n(s)\in\mathcal E^{h+1}_z}\pf\dd s\right| \  \leqslant \ \frac TR \1_{\tau_b\geqslant (b+1)T/R} \\
 + \1_{\tau_b<(b+1)T/R} \po \tau_b - \frac{bT}{R}  + \left| \int_{\tau_b}^{(b+1)T/R} \po \1_{\sigma_n(s)=j}  -   \pi_n^z(j)\1_{\sigma_n(s)\in\mathcal E^{h+1}_z}\pf\dd s\right|\pf \displaybreak[0] \\ 
 \leqslant \ \frac TR \po  \1_{\tau_b\geqslant (b+1)T/R} + \1_{\tau_b^+ - \tau_b \leqslant T/R} \pf  \\
 + \1_{\tau_b<(b+1)T/R} \po \tau_b - \frac{bT}{R}  +\left| \int_{\tau_b}^{(b+1)T/R} \po \1_{\tilde \sigma_n^{Z_b}(s)=j}  -   \pi_n^{z}(j)\1_{\tilde \sigma_n^{Z_b}(s)\in\mathcal E^{h+1}_z}\pf\dd s\right|\pf \\
 \leqslant \ \frac TR \po  \1_{\tau_b\geqslant (b+1)T/R} + \1_{\tau_b^+ - \tau_b \leqslant T/R} \pf  \\
+ 2 \1_{\tau_b<(b+1)T/R}  \po \tau_b - \frac{bT}{R}\pf   + \left| \int_{\tau_b}^{\tau_b + T/R} \po \1_{\tilde \sigma_n^{Z_b}(s)=j}  -   \pi_n^{z}(j)\1_{\tilde \sigma_n^{Z_b}(s)\in\mathcal E^{h+1}_z}\pf\dd s\right| \,.
 \end{multline*}
 We bound separately the expectations of these terms. First, for all $i\in\cco 1,N\ccf$,
 \begin{eqnarray*}
\mathbb E_i \po \1_{\tau_b<(b+1)T/R}  \po \tau_b - \frac{bT}{R}\pf \pf & \leqslant & \mathbb E_i \po \int_{bT/R}^{(b+1)T/R} \1_{\sigma_n(s) \in \Delta^{h+1}}\dd s\pf\\
& \leqslant & \sup_{u\in\cco 1,N\ccf} \mathbb E_u \po \int_{0}^{T/R} \1_{\sigma_n(s) \in \Delta^{h+1}}\dd s\pf\,, 
 \end{eqnarray*}
 which vanishes as $n\rightarrow +\infty$ (this is \eqref{Eq:Negligeable}). Similarly,
\begin{eqnarray*}
 \mathbb P_i \po \tau_b \geqslant (b+1)T/R\pf &=&  \mathbb P_i\po \int_{bT/R}^{(b+1)T/R}    \1_{\sigma_n(s) \in \Delta^{h+1}}\dd s  = \frac TR \pf  \\
 & \leqslant & \frac{R}{T}\sup_{u\in\cco 1,N\ccf} \mathbb E_u \po \int_{0}^{T/R} \1_{\sigma_n(s) \in \Delta^{h+1}}\dd s\pf\,.
\end{eqnarray*} 
Also, remark that if $Z_b \neq z$ then for all $s\geqslant \tau_b$
\[ \1_{\tilde \sigma_n^{Z_b}(s)=j}  = 0 =   \pi_n^{z}(j)\1_{\tilde \sigma_n^{Z_b}(s)\in\mathcal E^{h+1}_z}\,.\]
Up to now we have obtained that, for all $i\in\cco 1,N\ccf$, for $n$ large enough,
\begin{multline*}
\mathbb E_i \po \left| \int_0^T \po \1_{\sigma_n(s) = j} - \pi_n^z(j)\1_{\sigma_n(s)\in\mathcal E^{h+1}_z} \pf \dd s \right|  \pf \ \\ \leqslant \ \frac{KT}{R} +  3R \sup_{u\in\cco 1,N\ccf} \mathbb E_u \po \int_{0}^{T/R} \1_{\sigma_n(s) \in \Delta^{h+1}}\dd s\pf\\
+ \ R \sup_{u\in \Ebz} \mathbb E_{u} \po \left| \int_{0}^{T/R} \po \1_{\tilde \sigma_n^{z}(s)=j}  -   \pi_n^z(j)\1_{\tilde \sigma_n^{z}(s)\in\mathcal E^{h+1}_z}\pf\dd s\right| \pf\,.
\end{multline*}
Since the left-hand side does not depend on $R$, it only remains to prove that, for any fixed $R$, the last term converges to zero  as $n\rightarrow +\infty$. Using Proposition~\ref{prop:semigroup-pi}, the proof 
is exactly the same as that given for Proposition~\ref{Prop:CVversAveraged}.
\end{proof}

We are now ready to prove the main result of this section, Theorem~\ref{Thm:CV-mu-pas-borne}.

\begin{proof}[Proof of Theorem~\ref{Thm:CV-mu-pas-borne}]
 Let us first show that it is sufficient to consider the case where $P_n$ is strongly connected for all $n\in\N$. Indeed,  instead of a chain $\sigma_n$ associated with some $(\nu_n,\mu_n,P_n)_{n\in\N}$, we can rather consider a chain $\tilde \sigma_n$ associated with $(\nu_n,\mu_n+e^{-n},(\mu_n P_n+e^{-n}Q)/(\mu_n+e^{-n}))_{n\in\N}$ where $Q$ is the matrix whose coefficients are all $1/N$. In other words, $\tilde \sigma_n$ behaves like $\sigma_n$ except that, at rate $e^{-n}$, it jumps to a position distributed uniformly over $\cco 1,N\ccf$. In particular, as in the proof of Proposition~\ref{Prop:CV-mu-borne}, considering the synchronous coupling of $\sigma_n$ and $\tilde \sigma_n$,  we get that, for all $T>0$,  
 \[\mathbb P\po \sigma_n(t) = \tilde \sigma_n(t)\ \forall t\in[0,T]\pf \geqslant  e^{-T e^{-n}}\,.\]
Hence, if we prove that, up to a subsequence, the Markov process associated with $\tilde\sigma_n$ converges in the Skorokod topology to some convexified process, then the same convergence holds for the Markov process associated with $\sigma_n$. For this reason, from now on we suppose that $P_n$ is strongly connected.

The case where  $(\mu_n)_{n\in\N}$ is   bounded has already been treated in Proposition~\ref{Prop:CV-mu-borne}, so we suppose that $\lambda_n(i,j)\rightarrow +\infty$ for some distinct $i,j\in\cco 1,N\ccf$ and we keep all the definitions and notations of the rest of the section. Up to extracting a subsequence, we assume that, for all $z\in\cco 1,\mathfrak{n}_{h+1}\ccf$, $\pi_n^z$ converges as $n\rightarrow +\infty$ to some law $\pi^z$. Let us prove that for all $T,\delta>0$, $z\in \cco 1,\mathfrak{n}_{h+1}\ccf$, and $j\in \mathcal E_z^{h+1}$,
\begin{equation}\label{eq:tobph}
\sup_{i\in \cco 1,N\ccf}\mathbb P_i \po \left| \int_0^T \po \1_{\sigma_n(s) = j} - \pi^z(j)\1_{\tilde X^{h+1}(s) = z}\pf   \dd s \right| > \delta  \pf \ \underset{n\rightarrow +\infty} \longrightarrow \ 0\,,
\end{equation}
from which the conclusion follows exactly as in the proof  of \cite[Lemma~2.14]{benaim2019} or of Proposition~\ref{Prop:CVversAveraged}. First, by \eqref{Eq:CVintegraleXtilde},
\[\sup_{i\in \cco 1,N\ccf}\mathbb P_i \po \pi^z(j) \left| \int_0^T \po \1_{\Psi^{h+1}(\sigma_n(s)) = z} - \1_{\tilde X^{h+1}(s) = z}\pf   \dd s \right| > \delta/2  \pf \ \underset{n\rightarrow +\infty} \longrightarrow \ 0\,.\]
Second,
\[\sup_{i\in \cco 1,N\ccf}\mathbb E_i \po |\pi^z(j)-\pi_n^z(j)|   \int_0^T   \1_{\Psi^{h+1}(\sigma_n(s)) = z}   \dd s    \pf \ \leqslant \ T \abs{\pi_n^z - \pi^z}_\infty   \ \underset{n\rightarrow +\infty} \longrightarrow \ 0\,.\]
Third,  Proposition~\ref{Prop:mix} reads 
\[\sup_{i\in \cco 1,N\ccf}\mathbb E_i \po \left| \int_0^T \po \1_{\sigma_n(s) = j} - \pi_n^z(j)\1_{\Psi^{h+1}(\sigma_n(s)) = z} \pf \dd s \right|  \pf \ \underset{n\rightarrow +\infty} \longrightarrow \ 0\,.\]
The combination of these three facts yields \eqref{eq:tobph}.
\end{proof}

We now extend Theorem~\ref{Thm:CV-mu-pas-borne} to the case where the processes $(x_n, \sigma_n)$ are themselves already convexified. The parameters of a convexified process are $(\nu,\mu,P,B)$ where $B=(B_1,\dots,B_k)$ are the modes and $(\nu,\mu,P)$ are the parameters of the Markov chain on $\cco 1,k\ccf$.

\begin{coro}
\label{Coro:CV-mu-pas-borne}
Consider a sequence $(x_n,\sigma_n)$ of convexified Markov processes for $A$ with parameters $(\nu_n,\mu_n,P_n, B_n)_{n\in\N}$. Up to extracting a subsequence, there exists a convexified Markov process $(x,\sigma)$ for $A$ such that, for all $T,\delta>0$,
\begin{eqnarray*}
\mathbb P \po \sup_{t\in [0,T]} |x(t)-x_n(t)| > \delta \pf & \underset{n\rightarrow+\infty}\longrightarrow & 0\,.
\end{eqnarray*}
\end{coro}

\begin{proof}
For each $n\in\N$, the set of matrices in $B_n$ is related to a partition of the matrices in $A$, and thus, up to extracting a subsequence, we  assume that this partition is the same for all $n\in\N$. In other words, $B_n=(B_{n,1},\dots,B_{n,N'})$ where $N'\in\cco 1,N\ccf$ is independent of $n$, and for all $j\in \cco 1,N'\ccf$,  $B_{n,j}=\sum_{i\in I_j} \pi_{n,i} A_i$, where the  sets $I_1,\dots,I_{N'}$ are disjoint, nonempty, and independent of $n$, and 
 $(\pi_{n,i})_{i\in I_j}$ is a probability vector. 
In particular, for all $n\in\N$, $\nu_n$ and $P_n$ are respectively of dimension $N'$ and $N'\times N'$. By compactness, up to extracting a subsequence, we assume that the coefficients $\pi_{n,i}$ converge as $n\to \infty$, in other words $B_n\to B=(B_1,\dots,B_{N'})$.

For $n\in\N$, let  $y_n$ be the solution of $\dot y_n(t) = B_{\sigma_n(t)}y_n(t)$ for $t\geqslant 0$. In particular, $(y_n,\sigma_n)$ is a Markov process for $B$ with parameters $(\nu_n,\mu_n,P_n,B)$. Applying Theorem~\ref{Thm:CV-mu-pas-borne}, we get that there exists a  convexified Markov process  $(x,\sigma)$ for $B$ (which is in particular a convexified process for $A$) such that, for all $T,\delta>0$,
\begin{eqnarray*}
\mathbb P \po \sup_{t\in [0,T]} |x(t)-y_n(t)| > \delta \pf & \underset{n\rightarrow+\infty}\longrightarrow & 0\,.
\end{eqnarray*}
On the other hand, since $y_n$ and $x_n$ share the same Markov chain $\sigma_n$, we immediately get that, almost surely, for all $t\geqslant 0$,
\begin{align*}
|x_n(t)-y_n(t)| & \leqslant \int_{0}^t \po \|B_{\sigma_n(s)}\| |x_n(s)-y_n(s)|+ \|B_{n,\sigma_n(s)}-B_{\sigma_n(s)}\| |x_n(s)|\pf \dd s  \\
&\leqslant K\int_{0}^t|x_n(s)-y_n(s)|\dd s + t e^{Kt}|x_0| \max_{j\in\cco 1,N'\ccf} \|B_{j}-B_{n,j}\| 
\end{align*}
where $K=\max_{i\in\cco 1,N\ccf}\|A_i\|$ and thus, almost surely,
\[\sup_{t\in [0,T]} |x_n(t)-y_n(t)| \underset{n\to\infty}\longrightarrow 0\,,\]
yielding the conclusion.
\end{proof}

Up to now, in this section, the initial condition $x_0 \in \mathbb R^d$ has been fixed and is common to all processes $x$ and $x_n$. Since we are interested in the study of Lyapunov exponents, it is useful to reformulate Corollary~\ref{Coro:CV-mu-pas-borne} in terms of the flow $\Phi_\sigma$. This is done by applying it to the Markov process $(\Phi_{\sigma_n}, \sigma_n)$, which is a PDMP corresponding to the flow in the space $\mathcal M_d(\R)$ of $\dot \Phi_{\sigma_n}(t) = A_{\sigma_n(t)} \Phi_{\sigma_n}(t)$
with initial condition $\Phi_{\sigma_n}(0)=\id$.
As a straightforward consequence, using the fact that the flow is bounded on compact time intervals uniformly with respect to $\sigma$, we obtain the following. 

\begin{coro}
\label{coro:continuity}
Let $(\nu_n, \mu_n, P_n, B_n)_{n \in \mathbb N}$ be a sequence of convexified Markov processes for $A$. Then there exists a convexified Markov process $(\nu_\ast, \mu_\ast, P_\ast, B)$ for $A$ such that, up to extracting a subsequence, we have, for every $T > 0$,
\begin{equation}
\label{eq:cv-esperance-log}
\mathbb E\left[\log \norm{\Phi_{\sigma_n}(T)}\right] \underset{n\rightarrow+\infty}\longrightarrow \mathbb E\left[\log \norm{\Phi_{\sigma}(T)}\right],
\end{equation}
where $\sigma_n$ and $\sigma$ are obtained from the convexified Markov processes $(\nu_n, \mu_n, P_n, B_n)$ and $(\nu_\ast, \mu_\ast, P_\ast, B)$, respectively.
\end{coro}
 
\section{On the equality between \texorpdfstring{$\lambda_{\mathrm d}(A)$}{lambda d(A)} and \texorpdfstring{$\lambda_{\mathrm p}^{\sup}(A)$}{lambda p sup(A)}}
\label{SecCharactSupMarkov}

One of the difficulties in analysis of $\lambda_{\mathrm p}^{\sup}(A)$ is that it may fail to be reached by some $\lambda_{\mathrm p}(\nu, \mu, P, A)$, as detailed next.

\begin{expl}
Consider the case $N = d = 2$ with
\[A_1=\begin{pmatrix}0&-1\\1&-1\end{pmatrix},\quad A_2=\begin{pmatrix}0&1\\-1&-1\end{pmatrix}, \quad M = \frac12A_1+\frac12A_2=\begin{pmatrix}0&0\\0&-1\end{pmatrix}.\] 
Then the spectral abscissa of $M$ is equal to 0. From Proposition~\ref{prop:converge} stated below (since  it 
is easy to check that $M$ satisfies condition (C) in Definition~\ref{defi:cond-C} below),
for $\nu = (\frac{1}{2}, \frac{1}{2})$ and $P = \begin{pmatrix}\frac{1}{2} & \frac{1}{2} \\ \frac{1}{2} & \frac{1}{2}\end{pmatrix}$, we have $\lim_{\mu\to+\infty}\lambda_{\mathrm{p}}(\nu,\mu,P,A)=0$. Hence, on the one hand, $\lambda_{\mathrm{p}}^{\sup}(A) \geqslant 0$. 

On the other hand, it is easily seen that the standard Euclidean norm is strictly decreasing along both flows $\dot x= A_ix$, $i=1,2$, so that $\log\norm{\Phi_\sigma(t)} <0$ for any piecewise constant $\sigma$. From Proposition \ref{prop:condition-initiale} and \eqref{eq:LyapProba}, $\lambda_{\mathrm{p}}(\nu,\mu,P,A)<0$ for every choice of the Markov process $(\nu,\mu,P)$.

Thus, $\lambda_{\mathrm{p}}^{\sup}(A)$ is not reached. 
\end{expl}

The equality cases between $\lambda_{\mathrm d}(A)$ and $\lambda_{\mathrm p}^{\sup}(A)$ will be addressed by studying the equality cases in \eqref{eq:triplet}. 
A helpful property is that, contrarily to what happens for $\lambda_{\mathrm p}^{\sup}(A)$, the supremum in the definition of $\lambda_{\mathrm p}^{\mathrm{conv}}(A)$ is always reached, as shown next.

\begin{prop}\label{prop:maximizing}
There exists a convexified Markov process  $(\nu, \mu, P, B)$ for $A$ such that  $\lambda_{\mathrm p}^{\mathrm{conv}}(A)=\lambda_{\mathrm p}(\nu, \mu, P, B)$. 
\end{prop}

\begin{proof}
Let $(\nu_n, \mu_n, P_n, B_n)_{n \in \mathbb N}$ be a maximizing sequence for $\lambda_{\mathrm p}^{\mathrm{conv}}(A)$, i.e., $\lambda_{\mathrm p}(\nu_n, \mu_n,\allowbreak P_n,\allowbreak B_n) \to \lambda_{\mathrm p}^{\mathrm{conv}}(A)$ as $n \to \infty$. By Proposition~\ref{prop:condition-initiale}, with no loss of generality, we may assume that $\nu_n$ is invariant under $P_n$ for every $n \in \mathbb N$.

By Corollary~\ref{coro:continuity}, there exists a convexified Markov process $(\nu, \mu, P, B)$ for $A$ such that, for every $T > 0$, \eqref{eq:cv-esperance-log} holds. Then, using \eqref{eq:defLambdaP} and \eqref{eq:LyapProba},
\begin{align*}
\lambda_{\mathrm p}(\nu, \mu, P, B) & = \limsup_{t \to +\infty} \frac{1}{t} \mathbb E[\log\norm{\Phi_{\sigma}(t)}] = \limsup_{t \to +\infty} \lim_{n \to +\infty} \frac{1}{t} \mathbb E[\log\norm{\Phi_{\sigma_n}(t)}] \\
& \geqslant \lim_{n \to +\infty} \lambda_{\mathrm p}(\nu_n, \mu_n, P_n, A) = \lambda_{\mathrm p}^{\mathrm{conv}}(A),
\end{align*}
hence the conclusion.
\end{proof}

We are now ready to prove Theorem~\ref{thm:lambda-p-conv}.

\begin{proof}[Proof of Theorem~\ref{thm:lambda-p-conv}]
Let us prove \ref{thm:main-iff-conv}. Assume first that $\lambda_{\mathrm d}(A) = \lambda_{\mathrm p}^{\mathrm{conv}}(A)$ and consider a convexified Markov process $(\nu, \mu, P, B)$ at which $\lambda_{\mathrm p}^{\mathrm{conv}}(A)$ is attained, whose existence is established in Proposition~\ref{prop:maximizing}. Hence
\[
\lambda_{\mathrm d}(B) \leqslant \lambda_{\mathrm d}(A) = \lambda_{\mathrm p}^{\mathrm{conv}}(A) = \lambda_{\mathrm p}(\nu, \mu, P, B) \leqslant \lambda_{\mathrm d}(B),
\]
implying that $\lambda_{\mathrm p}(\nu, \mu, P, B) = \lambda_{\mathrm d}(B)$. Let $i$ be an index belonging to a recurrent class of $P$ and accessible from $\nu$. Then, by Theorem~\ref{TheoPNotSC-ContinuousTime}, we deduce that $\spr(e^{B_i t}) = e^{\lambda_{\mathrm d}(B) t}$ for every $t \geqslant 0$, which yields $\lambda(B_i) = \lambda_{\mathrm d}(B)$, and the conclusion follows since $B_i \in \co(A)$ and $\lambda_{\mathrm d}(B) = \lambda_{\mathrm d}(A)$.

To prove the converse implication, let $M \in \co(A)$ be such that $\lambda(M) = \lambda_{\mathrm d}(A)$. Then the process $\dot x = M x$ is a convexified Markov process for $A$, with parameters $(\nu, \mu, P, B)$ given by $B = (M)$ and $(\nu, \mu, P)$ determining the (constant) Markov chain on a single state. Hence
\[\lambda_{\mathrm d}(A) = \lambda(M) = \lambda_{\mathrm p}(\nu, \mu, P, B) \leqslant \lambda_{\mathrm p}^{\mathrm{conv}}(A) \leqslant \lambda_{\mathrm d}(A),\]
yielding the conclusion.

The proof of \ref{thm:main-a} follows immediately from \ref{thm:main-iff-conv} and \eqref{eq:triplet}. 

To prove \ref{thm:main-b}, notice that, by the continuity of the spectral abscissa function, for every $\eta>0$, there exists $\varepsilon>0$ such that 
$\lambda(M^\varepsilon)>\lambda(M)-\eta= {\lambda_{\mathrm d}(A)}-\eta$. 
Let $(\nu_n, \mu_n, P_n)_{n\in\N}$ be a sequence of Markov processes for $A$ such that \eqref{eq:condC} holds true. 
Then, for $n$ large enough
\[\lambda_{\mathrm p}(\nu_n, \mu_n, P_n,A)>{\lambda_{\mathrm d}(A)}-\eta.\]
Hence $\lambda_{\mathrm d}(A) < \lambda_{\mathrm p}^{\sup}(A) + \eta$ and, since $\eta > 0$ is arbitrary, we deduce that $\lambda_{\mathrm d}(A) \leqslant \lambda_{\mathrm p}^{\sup}(A)$, yielding the conclusion thanks to \eqref{LambdaPLeqLambdaD-Continuous}.
\end{proof}

Theorem~\ref{thm:lambda-p-conv}~\ref{thm:main-b} raises the question of verifying the condition stated in \eqref{eq:condC}. If $M = A_i$ for some $i \in \llbracket 1, N\rrbracket$, \eqref{eq:condC} trivially holds by taking $M^\varepsilon = A_i$ and the Markov chain constantly equal to $i$. Otherwise, one can expect $M$ to be attained through arbitrarily fast switching. This leads to the question of the convergence of the Lyapunov exponent in this regime, which is the subject of the sequel of this section.

Let us introduce an explicit algebraic condition under which we are able to prove the convergence of Lyapunov exponents stated in \eqref{eq:condC}.

\begin{defi}[Condition (C)]
\label{defi:cond-C}
Let $M$ be in $\mathcal M_d(\R)$ and denote by $\xi_1>\xi_2>\dots>\xi_k$ the distinct real parts of the eigenvalues of $M$ (with $k\leqslant d$). 
For $j\in\llbracket 1,k\rrbracket$, denote by $E_j$ the space spanned by the generalized eigenvectors of $M$ corresponding to eigenvalues with real part $\xi_j$ and by $n_j$ the dimension of $E_j$. Let $F_M$ be the vector field on the $(d-1)$-dimensional unit sphere $S^{d-1}$ obtained by projecting $x\mapsto Mx$, i.e., $F_M(x)=Mx-(x\cdot Mx)x$. 

Given $A = (A_1, \dotsc, A_N) \in \mathcal M_d(\mathbb R)^N$, we say that $M$ satisfies \emph{condition (C) for $A$} if,  
for every $j\in\llbracket   2,k\rrbracket$,  there exists $i \in   \llbracket 1,N\rrbracket$  such that, for every $\theta \in  E_j\cap S^{d-1}$, one has  
$F_{A_i}(\theta)\not \in \bigoplus_{r\geqslant j} E_r$.
\end{defi}

Condition (C) can be used to obtain the following result on the convergence of Lyapunov exponents.

\begin{prop}\label{prop:converge}
Let $A = (A_1, \dotsc, A_N) \in \mathcal M_d(\mathbb R)^N$, $\pi = (\pi_1, \dotsc, \pi_N) \in (0, 1]^N$ be a probability vector, and define $M = \sum_{j=1}^N \pi_j A_j$. Let $P \in \mathcal M_N(\mathbb R)$ be the matrix with all rows equal to $\pi$ and $(\mu_n)_{n \in \mathbb N}$ be a sequence tending to $+\infty$. If $M$ satisfies condition (C) for $A$, then the sequence $(\pi, \mu_n, P)_{n\in\N}$ of Markov processes for $A$ satisfies
\[
\lim_{n \to \infty} \lambda_{\mathrm p}(\pi, \mu_n, P, A) = \lambda(M).
\]
\end{prop}

From the results in \cite[Section~2.5]{Benaim2014Stability}, the probabilistic Lyapunov exponent can be expressed in terms of the invariant measures of the Markov process $(\theta(t), \sigma(t))_{t \geqslant 0} = (x(t) / \norm{x(t)}, \sigma(t))_{t \geqslant 0}$ on $S^{d-1}\times\cco 1,N\ccf$. As a consequence, Proposition~\ref{prop:converge} follows from a simple combination of the results of \cite{Benaim2014Stability} and of the following proposition, whose proof is given in Appendix~\ref{app:B}.

\begin{prop}\label{prop:PierreP1}
Let $A = (A_1, \dotsc, A_N) \in \mathcal M_d(\mathbb R)^N$, $\pi = (\pi_1, \dotsc, \pi_N) \in (0, 1]^N$ be a probability vector, and define $M = \sum_{j=1}^N \pi_j A_j$. Let $P \in \mathcal M_N(\mathbb R)$ be the matrix with all rows equal to $\pi$, $\mu > 0$, and consider the Markov process $(\theta(t),\sigma(t))_{t\geqslant 0}$ on $S^{d-1} \times \llbracket 1, N\rrbracket$ obtained from $(\pi, \mu, P)$.

If $M$ satisfies condition (C) for $A$, then, for every $\varepsilon \in(0,1]$ and every neighborhood $K$ of $E_1$ in $S^{d-1}$, there exists $\mu_0>0$ such that, for every $\mu\geqslant \mu_0$ and every invariant measure $\rho$ of the process $(\theta(t),\sigma(t))_{t\geqslant 0}$, one has $\rho(K\times\cco 1,N\ccf) \geqslant 1 -\varepsilon$.
\end{prop} 

Notice that a similar result has recently been established in \cite{DU2021313}. More precisely, our result is similar in spirit to \cite[Theorem~1.1]{DU2021313} in the case where the diffusion coefficient $\delta$ is equal to $0$ and the limit cycle considered in \cite{DU2021313} is reduced to a single point, and our condition (C) corresponds to the second point of \cite[Assumption~1.2]{DU2021313}. The proof of \cite{DU2021313} is based on estimates on the exit times of the trajectory from neighborhoods of unstable equilibria, while our argument is based on the construction of a suitable Lyapunov function (similarly to \cite{Benaim2014Stability}, where the result is established in a particular case in dimension~$2$).

By combining Proposition~\ref{prop:converge} with Theorem~\ref{thm:lambda-p-conv}, one deduces the following statement.

\begin{coro}
\label{Coro:conditionC}
Let $A \in \mathcal M_d(\mathbb R)^N$. 
Assume that condition (C) holds for a dense subset of $\co(A)$.
 Then $\lambda_{\mathrm d}(A) = \lambda_{\mathrm p}^{\sup}(A)$ if and only if there exists 
 $M \in \co(A)$ such that $\lambda(M) = {\lambda_{\mathrm d}(A)}$.
\end{coro}

From what precedes one would prove a full converse of Theorem~\ref{thm:lambda-p-conv}~\ref{thm:main-a} if for every $A  \in \mathcal M_d(\mathbb R)^N$
condition (C) held for a dense subset of $\co(A)$. Our next result, whose proof is provided in Appendix~\ref{app:B}, states that this is the case for dimensions up to $3$.

\begin{lemm}\label{lemma:dim23}
Let $d\leqslant 3$ and $A \in \mathcal M_d(\mathbb R)^N$ be irreducible. Then condition (C) holds for a dense subset of $\co(A)$.
\end{lemm}

Thanks to Lemma~\ref{lemma:dim23}, we are now in position to prove Proposition~\ref{prop:dim23}.

\begin{proof}[Proof of Proposition~\ref{prop:dim23}]
By Theorem~\ref{thm:lambda-p-conv} \ref{thm:main-a}, it remains to prove that $\lambda_{\mathrm d}(A) = \lambda_{\mathrm p}^{\sup}(A)$ if there exists $M \in \co(A)$ such that $\lambda(M) = \lambda_{\mathrm d}(A)$. Hence, we assume that such a $M$ exists. We next use the block decomposition of matrices of $A$ used in the proof of Proposition~\ref{CoroFixedP-ContinuousTime} (and, in particular, the notations of \eqref{DecomposeAj}, extended to to the matrices in $\co(A)$) to deduce that there exists 
$j\in\llbracket 1,S\rrbracket$ such that $\lambda(M^{(j)})=\lambda(M)$. Note that $\lambda_{\mathrm d}(A^{(j)}) \leqslant \lambda_{\mathrm d}(A) = \lambda(M) = \lambda(M^{(j)}) \leqslant \lambda_{\mathrm d}(A^{(j)})$, implying that all these inequalities are in fact equalities. Thus, by applying Lemma~\ref{lemma:dim23} and Corollary~\ref{Coro:conditionC} to $M^{(j)}$ and $A^{(j)}$, we deduce that $\lambda_{\mathrm d}(A^{(j)}) = \lambda_{\mathrm p}^{\sup}(A^{(j)})$. Hence,
\[
\lambda_{\mathrm d}(A) = \lambda_{\mathrm d}(A^{(j)}) = \lambda_{\mathrm p}^{\sup}(A^{(j)}) \leqslant \lambda_{\mathrm p}^{\sup}(A) \leqslant \lambda_{\mathrm d}(A),
\]
yielding $\lambda_{\mathrm d}(A) = \lambda_{\mathrm p}^{\sup}(A)$.
\end{proof}

\section*{Acknowledgements}

The authors thank Edouard Strickler for fruitful discussions. G.~Mazanti was partially supported by ANR PIA funding number ANR-20-IDEES-0002. P.~Monmarch{\'e} acknowledges financial support by the French ANR grant SWIDIMS (ANR-20-CE40-0022).

\appendix

\section{Decomposition of a Markov chain}\label{Sec:MarkovDecomposition}

This section is devoted to the proof of Theorem~\ref{Thm:Landim}. It is based on the works of Landim and collaborators \cite{Landim2016,Landim2010,LandimSoftFD,LandimSoft}, to which we refer for more details and discussions. Even though the contents of  Theorem~\ref{Thm:Landim} are similar to those of Theorems~2.1, 2.7, and 2.12 of \cite{Landim2016}, we decided to state them in a modified way, more convenient in our context. Let us briefly highlight the main differences and explain how they are handled in the sequel.

\begin{itemize}
\item  In \cite{Landim2016}, the convergence after rescaling in time is given in term of the so-called soft topology, introduced in \cite{LandimSoft}, while in Theorem~\ref{Thm:Landim}~\ref{lan2_CVMarkov} and \ref{lan3_Pastransition} we state a convergence for the time marginals. The way to get this time marginals convergence instead of the soft one is established in \cite[Proposition~2.1]{LandimSoftFD}.

\item The case of a timescale which is strictly between $\theta^j$ and $\theta^{j+1}$ is not considered in \cite{Landim2016}. In particular, in \cite[Condition H1]{Landim2016}, the total rate of the limit Markov chain is supposed to be nonzero. Nevertheless, this assumption is not used to prove the convergence of the rescaled chain, it is just used, after the convergence is established, to say that the limit chain is not constant, so that in particular the next partition will be stricly coarser than the previous one. Replacing $\theta^{j+1}$ by some $\alpha$ with $\theta^j \ll \alpha \ll \theta^{j+1}$, the conditions \cite[H1--H3]{Landim2016} are still satisfied, except that in H1 the total rate is zero.

\item The fact that \eqref{eq:LandimCVX} and \eqref{eq:LandimCVX_integral} still hold with $\tilde \theta_n^j\simeq \theta_n^j$ is immediately obtained from the fact $\tilde \theta^j$ satisfies the same conditions \cite[H1--H3]{Landim2016} as $\theta^j$.
\end{itemize}

Throughout this section, we consider for each $n\in \N$ a  strongly connected  Markov chain $(\sigma_n(t))_{t\geqslant 0}$ on $\cco 1,N\ccf$ for some fixed $N\in\N$ with jump rates $(\lambda_n(i,j))_{i,j\in\cco 1,N\ccf}$. We suppose  that the conditions \eqref{cond:lambda_n} and \eqref{Eq:CondOrdered} are fulfilled. We denote by $\pi_n$ the unique invariant probability measure of $\sigma_n$.

For a nonempty $J\subset \cco 1,N\ccf$,  the \emph{trace process on $J$ associated with $\sigma_n$} is the process $(\sigma_n^J(t))_{t\geqslant 0}$ obtained from $\sigma_n$  by removing all the time spent outside of $J$. More precisely, denoting
\[T_n^J(t) \ = \ \int_0^t \1_{\sigma_n(s)\in J} \dd s\qquad\text{and}\qquad S_n^J(t) \ = \ \sup\{s\geqslant 0\mid T_n^J(s)\leqslant t\} \,,\]
we define $\sigma_n^J(t)  = \sigma_n(S_n^J(t))$. Remark that, the chain being strongly connected, $T_n^J(t) \rightarrow +\infty$ almost surely as $t \to +\infty$ and $S_n^J(t)$ is finite for all $t\geqslant 0$.

For $J\subset \cco 1,N\ccf$,  denote by $H_J$ and $H_J^+$ the hitting time of $J$ and the time of the first return to $J$, i.e.,
\[H_J \ = \ \inf\{t\geqslant 0\mid \sigma_n(t)\in J\},\qquad H_J^+ \ = \ \inf\{t\geqslant \tau_1\mid \sigma_n(t)\in J\}\,,\]
where $\tau_1 =\inf\{t\geqslant 0\mid \sigma_n(t)\neq \sigma_n(0)\}$ is the first jump time of the chain. For $i,j\in \cco 1,N\ccf$, denote $\lambda_n(i) = \sum_{k\neq i} \lambda_n(i,k)$ the holding rate of the chain at $i$, and $p_n(i,j) = \lambda_n(i,j)/\lambda_n(i)$ the transition probabilities from $i$. For two disjoint subsets $J_1$ and $J_2$ of $\cco 1,N\ccf$, the capacity between $J_1$ and $J_2$ is defined by
\[\mathrm{cap}_n\po J_1, J_2\pf \ = \ \sum_{i\in J_1} \pi_n(i) \lambda_n(i) \mathbb P_i\po H_{J_2} < H_{J_1}^+\pf\,.\]
Let $\mathcal E_1,\dotsc,\mathcal E_{\mathfrak{n}},\Delta$ be a partition of $\cco 1,N\ccf$ for some $\mathfrak{n}\geqslant 1$ with $\mathcal E_x\neq\emptyset$ for all $x\in\cco 1,\mathfrak{n}\ccf$ and let $\mathcal E = \bigcup_{x=1}^{\mathfrak{n}} \mathcal E_x$. It is proven in \cite[Proposition~6.1]{Landim2010} that the trace process $\sigma_n^{\mathcal E}$ is a Markov chain on $\mathcal E$ with rates
\[R_n(i,j) \ = \ \lambda_n(i) \mathbb P_i \po H_{\{j\}} = H^+_{\mathcal E}\pf\]
for $i\neq j$. For $x,y\in \cco 1,\mathfrak{n}\ccf$, $x\neq y$, denote by
\[r_n^{\mathcal E}(x,y) \ = \ \frac1{\pi_n(\mathcal E_x)}\sum_{i\in\mathcal E_x}\pi_n(i) \sum_{j\in\mathcal E_y} R_n(i,j)\] 
the mean (at equilibrium) rate at which the trace process jumps from $\mathcal E_x$ to $\mathcal E_y$. Consider the coarse-grained variable
\[\Psi^{\mathcal E}(i) \ = \ \sum_{x=1}^{\mathfrak{n}} x \1_{i\in\mathcal E_x}\,.\]
Finally, let $(\alpha_n)_{n\in\N}$, $(\beta_n)_{n\in\N}$ be two positive sequences. We consider the following conditions:
\begin{description}
\item[\textbf{H1.}]  For all $x,y\in\cco 1,\mathfrak n\ccf$ with $x\neq y$, there exists ${r^{\mathcal E}}(x,y)\geqslant 0$ so that
\[\beta_n r_n^{\mathcal E}(x,y) \ \underset{n\rightarrow +\infty}\longrightarrow \ r^{\mathcal E}(x,y)\,.\]
\item[\textbf{H2.}] For all $x\in \cco 1,\mathfrak{n}\ccf$ such that $\mathcal E_x$ is not a singleton and all $i,j\in \mathcal E_x$ with $i\neq j$,
\[\liminf_{n\rightarrow +\infty}\alpha_n \frac{\mathrm{cap}\po\{i\},\{j\}\pf}{\pi_n(\mathcal E_x)} \ > \ 0\,.\]
\end{description} 
Intuitively, \textbf{H2} means that mixing within a class $\mathcal E_x$ occurs in a time smaller than $\alpha_n$, and \textbf{H1} that transitions between different classes $\mathcal E_x$ and $\mathcal E_y$ occur at a time of order at least $\beta_n$.

\begin{prop}\label{prop:LandimTraceProcess}
Consider a partition $\mathcal E_1,\dots,\mathcal E_{\mathfrak{n}},\Delta$ of $\cco 1,N\ccf$ and two positive sequences $(\alpha_n)_{n\in\N}$, $(\beta_n)_{n\in\N}$ such that $\alpha_n \ll \beta_n$ and $\mathbf{H1}$--$\mathbf{H2}$ hold. Then, for $x\in \cco 1,\mathfrak{n}\ccf$ and for an initial condition $i\in\mathcal E_x$, $\po \Psi^{\mathcal E}(\sigma_n^{\mathcal E}(t\beta_n))\pf_{t\geqslant 0}$ converges in law in the Skorokhod topology toward the Markov chain $(X(t))_{t\geqslant 0}$ on $\cco 1,\mathfrak n\ccf$ with initial condition $x$ and jump rates $\po r^{\mathcal E}(y,z)\pf_{y,z\in\cco 1,\mathfrak{n}\ccf}$ as $n\rightarrow +\infty$.
\end{prop}
\begin{proof}
This is \cite[Proposition~6.1]{Landim2016}. Remark that in \cite{Landim2016} it is also assumed that the limit chain is not constant, i.e., $\sum_{x,y\in\cco 1,\mathfrak{n}\ccf} r^{\mathcal E}(x,y) \neq 0$. Nevertheless this is not used in the proof of \cite[Proposition~6.1]{Landim2016}, which relies on results from \cite{Landim2010} in which this additional assumption is not made.
\end{proof}

Remark that in the particular case  $\mathfrak{n}=1$ the result is trivial as condition \textbf{H1} is empty and $ \Psi^{\mathcal E}(\sigma_n^{\mathcal E}(t))=1$ for all $n\in\N$ and $t\geqslant 0$.

Now, consider the following additional conditions on the partition $\mathcal E_1,\dots,\mathcal E_{\mathfrak{n}},\Delta$ and the timescale $\beta$:
\begin{description}
\item[\textbf{H3.}] For all $i\in\cco 1,N\ccf$ and $t>0$,
\[\mathbb E_i \po \int_0^t \1_{\sigma_n(s\beta_n) \in \Delta} \dd s \pf \ \underset{n\rightarrow +\infty}\longrightarrow \ 0\,.\]
\item[\textbf{H4.}] For all $x\in\cco 1,\mathfrak{n}\ccf$ and $i\in\mathcal E_x$, 
\[\lim_{\delta\rightarrow 0} \limsup_{n\rightarrow +\infty} \sup_{s\in [2\delta,3\delta]} \mathbb P_i \po \sigma_n(s\beta_n) \in \Delta\pf = 0\,.\]
\end{description}

\begin{prop}\label{prop:Landim2}
Consider a partition $\mathcal E_1,\dots,\mathcal E_{\mathfrak{n}},\Delta$ of $\cco 1,N\ccf$ and two positive sequences $(\alpha_n)_{n\in\N}$, $(\beta_n)_{n\in\N}$ such that $\alpha_n \ll \beta_n$ and $\mathbf{H1}$--$\mathbf{H3}$ hold. Fix $x\in \cco 1,\mathfrak{n}\ccf$ and an initial condition $i\in\mathcal E_x$. Let $X$ be as in  Proposition~\ref{prop:LandimTraceProcess}. Then
\begin{enumerate}
\item   For all $t\geqslant 0$ and $y\in \cco 1,\mathfrak{n}\ccf$, 
\[\int_0^t \1_{\Psi^{\mathcal E}(\sigma_n(s\beta_n))=y} \dd s \ 
\underset{n\rightarrow +\infty}\longrightarrow\ \int_0^t \1_{X(s)=y}\dd s\qquad\mbox{in law}.\]
\item Let $\tau_n = \inf\{t\geqslant 0\mid \sigma_n(t) \notin \mathcal E_x \cup \Delta \}$. Then $\tau_n/\beta_n$ converges in law as $n\rightarrow +\infty$ to an exponential variable with parameter $r^{\mathcal E}(x):=\sum_{y\in\cco 1,\mathfrak n\ccf\setminus\{x\}}r^{\mathcal E}(x,y)$ (to be understood as $\tau_n/\beta_n\rightarrow +\infty$ in probability in the case where $r^{\mathcal E}(x) = 0$).
\item If, moreover, $\mathbf{H4}$ holds, then for all $t_1,\dots,t_k\geqslant 0$,  $\po \Psi^{\mathcal E}(\sigma_n(t_j\beta_n) )\pf_{j\in \cco 1,k\ccf }$ converges in law toward $\po X(t_j)\pf_{j\in \cco 1,k\ccf }$  as $n\rightarrow +\infty$.
\end{enumerate}
\end{prop}

\begin{proof}
For the first point, we know by Proposition~\ref{prop:LandimTraceProcess} that, for all $t\geqslant 0$, $\int_0^t \1_{\Psi^{\mathcal E}(\sigma_n^{\mathcal E}(s\beta_n))=y} \allowbreak \dd s$ converges in law toward $\int_0^t \1_{X(s)=y}\dd s$ as $n\rightarrow +\infty$. Moreover,
\[\mathbb E \po  \int_0^t \left| \1_{\Psi^{\mathcal E}(\sigma_n(s\beta_n))=y}- \1_{\Psi^{\mathcal E}(\sigma_n^{\mathcal E}(s\beta_n))=y}\right| \dd s \pf \ \leqslant \ 2 \mathbb E \po  \int_0^t  \1_{ \sigma_n(s\beta_n)\in\Delta } \dd s \pf \,,\]
which vanishes as $n\rightarrow +\infty$ according to \textbf{H3}.

The proof of the second point is similar. Indeed, denoting $\tau_n^{\mathcal E} = \inf\{t\geqslant 0\mid \sigma_n^{\mathcal E}(t) \notin \mathcal E_x  \}$, the convergence in law of $\Psi^{\mathcal E}(\sigma_n^{\mathcal E}(\cdot \beta_n))$ toward $X$ in the Skorokhod topology implies the convergence in law of $\tau_n^{\mathcal E} /\beta_n$ toward the first jump time of $X$. Moreover, necessarily, $\tau_n^{\mathcal E} \leqslant \tau_n$ (since some time is removed in the trace process), and for all $M,\varepsilon>0$, 
\begin{eqnarray*}
\mathbb P_x \po \tau_n^{\mathcal E} /\beta_n \geqslant M \pf & \leqslant & \mathbb P_x \po \tau_n /\beta_n \geqslant M \pf \\
& \leqslant  & \mathbb P_x \po \tau_n^{\mathcal E} /\beta_n \geqslant M  - \varepsilon \pf + \mathbb P_x\po \sup_{s \in [0,M\beta_n]}|s-T_n^{\mathcal E}(s)| \geqslant \varepsilon  \beta_n\pf\\
& \leqslant  & \mathbb P_x \po \tau_n^{\mathcal E} /\beta_n \geqslant M  - \varepsilon \pf + \mathbb P_x\po \int_0^{M\beta_n} \1_{\sigma_n(s) \in \Delta}\dd s  \geqslant \varepsilon \beta_n \pf\,.
\end{eqnarray*} 
The last term vanishes as $n\rightarrow +\infty$ thanks to \textbf{H3}. Noticing that $\varepsilon$ is arbitrary, we conclude by using the convergence of $\tau_n^{\mathcal E} /\beta_n $.

Finally, the convergence of the time marginals of $\Psi^{\mathcal E}(\sigma_n^{\mathcal E}(\cdot \theta_n) )$ is established in \cite[Proposition~2.1]{LandimSoftFD} under \textbf{H3}, \textbf{H4}, and the convergence in the Skorokhod topology for the trace process  proven in Proposition~\ref{prop:LandimTraceProcess}.
\end{proof}

\begin{proof}[Proof of Theorem~\ref{Thm:Landim}]
Let us recall the construction of \cite{Landim2016} (with our slightly different notations). The first timescale $\theta^1$ is defined by $\theta^1_n = \left(\sum_{i,j=1}^N \lambda_n(i,j)\right)^{-1}$, and the first partition is simply $\mathcal E_x^1 = \{x\}$ for $x\in \cco 1,N\ccf$, hence $\mathfrak{n}_1=N$, and $\Delta^1 = \emptyset$. As seen in \cite[Section~3]{Landim2016}, condition \eqref{Eq:CondOrdered} ensures that $\lambda_n(i,j)\theta_n^1$   admits a 
nonnegative limit $r_1(i,j)$ for all $i,j\in\cco 1,N\ccf$, with at least one pair $(i,j)$ such that $r_1(i,j)\neq 0$. Notice that there is no transition at a timescale smaller than $\theta^1$ since, for all $t\geqslant 0$, the probability that there has been a jump before time $t$ is less than $e^{-t/\theta^1}$. Let $X^1$ be the Markov chain in $\cco 1,N\ccf$ with transition rates $r_1(i,j)$. Since the jump rates of $\sigma_n(\cdot \theta_n^1)$ converge toward those of $X^1$, as in the proof of Proposition~\ref{Prop:CV-mu-borne} we can consider a synchronous coupling for which 
\begin{equation}\label{eq:VCX1unif}
\mathbb P \po X^1(t) = \sigma_n(t\theta_n^1)\ \forall t\in[0,T]\pf \underset{n\rightarrow +\infty}\longrightarrow \ 1
\end{equation}
for all $T>0$. This implies the convergence in law in the Skorokhod space, hence point \ref{lan2_CVMarkov} for $j=1$.

Denote by $\mathcal C_1^1,\dots,\mathcal C_{\mathfrak{n}_{2}}^1$ the recurrence classes of $X^1$ and, for all $x\in \cco 1,\mathfrak{n}_{2}\ccf$, set $\mathcal E_x^{2} =\mathcal C_x^1$. Denote by $\mathcal T_1$ the set of transient points of $X^1$ and set $\Delta^{2} =\mathcal T_1$. Since $X^1$ is not a constant chain, necessarily $\mathfrak{n}_2<\mathfrak{n}_1$.

Then, the timescales and the partitions are defined by induction. Suppose that $\theta^{j-1}$, $\mathfrak{n}_{j}$, $\mathcal E_1^{j},\dots,\mathcal E_{\mathfrak{n}_j}^j$  (with $\mathcal E_x^j \neq \emptyset$ for all $x\in\cco 1,\mathfrak{n}_j\ccf$), and $\Delta^j$ have been defined for some $j\geqslant 2$. If $\mathfrak n_j = 1$, we stop the construction (and set $\mathfrak{p}=j-1$). Otherwise, denoting $\mathcal E_{\neq x} = \bigcup_{y\neq x} \mathcal E_y$ for $x\in\cco 1,\mathfrak{n}_j\ccf $, we set
\[\theta^j_n \ = \ \po \sum_{x=1}^{\mathfrak n_j} \frac{\mathrm{cap}(\mathcal E_x,\mathcal E_{\neq x})}{\mu_n(\mathcal E_x)} \pf^{-1}\,,\]
which is well defined since the chain is strongly connected, so that the capacities are nonzero. By \cite[Theorems~2.7 and 2.12]{Landim2016}, for all $j\in\cco 2,\mathfrak{p}\ccf$, the partition $\mathcal E_1^j,\dots,\mathcal{ E}_{\mathfrak{n}_{j}}^j,\Delta^j$ and the timescales $\alpha_n = \theta_n^{j-1}$ and $\beta_n=\theta_n^j$ satisfy \textbf{H1}--\textbf{H3} (with, in \textbf{H1}, $r^{\mathcal E}(x,y)\neq 0$ for at least one pair $(x,y)$) and $\theta^j \ll \theta^{j+1}$. Let $X^j$ be the Markov chain on $\cco 1,\mathfrak{n}_j\ccf$ given by  Proposition~\ref{prop:LandimTraceProcess} for this partition and these timescales. Denote by $\mathcal C_1^j,\dots,\mathcal C_{\mathfrak{n}_{j+1}}^j$  its recurrence classes and $\mathcal T_j$ its transient points. For all $x\in \cco 1,\mathfrak{n}_{j+1}\ccf$, set $\mathcal E_x^{j+1} = \bigcup_{y\in \mathcal C_x^j} \mathcal E_y^j$ and $\Delta^{j+1} = \Delta^j \cup (\bigcup_{y\in\mathcal T_j} \mathcal E_y^j)$. The chain $X^j$ being nonconstant, $\mathfrak{n}_{j+1}<\mathfrak{n}_j$. This shows that this inductive definition of timescales and partitions ends in a finite number of steps.

Points \ref{lan1_timescale} and \ref{lan5coarser} of Theorem~\ref{Thm:Landim} are satisfied by construction. Point  \ref{lan6escape_expo} is  a consequence of Proposition~\ref{prop:Landim2}.

Let us prove that \textbf{H4} holds for this choice of partitions and timescales. More precisely, let us prove by induction on $j\in\cco 1,\mathfrak{p}\ccf$ that  
\begin{equation}\label{eq:induction_delta}
\forall \delta>0,\ \forall i\in\cco 1,N\ccf,\qquad  \limsup_{n\rightarrow +\infty} \sup_{s\geqslant \delta } \mathbb P_i \po \sigma_n(s\theta_n^j) \in \Delta^{j}\pf = 0\,,
\end{equation}
which is stronger than \textbf{H4}. For $j=1$, $\Delta^1 = \emptyset$, so there is nothing to prove. For $j=2$, let $\varepsilon>0$ and  $M>0$ be such that
\[\sup_{t\geqslant M} \max_{i\in\cco 1,N\ccf}\mathbb P_i\po X^1(t) \in \mathcal T_1\pf \ \leqslant \ \varepsilon.\]
Thanks to \eqref{eq:VCX1unif},  we can consider $n_0\in\N$  such that for all $n\geqslant n_0$
\[  \max_{i\in\cco 1,N\ccf}\mathbb P_i\po \sigma_n(M\theta^1_n) \in \Delta^2\pf \ \leqslant \ 2\varepsilon\,.\]
For $\delta>0$, for all $n\geqslant n_0$ large enough so that $\delta\theta^2_n > M\theta^1_n$, by the Markov property, for all $s\geqslant \delta$ and $i\in\cco 1,N\ccf$,
\[ \mathbb P_i \po \sigma_n(s\theta_n^j) \in \Delta^{2}\pf \ = \  \sum_{k=1}^N \mathbb P_i \po \sigma_n(s\theta_n^j-M\theta^1_n) =k\pf  \mathbb P_k \po \sigma_n(M\theta^1_n) \in \Delta^{2}\pf \ \leqslant \ 2\varepsilon\,.\]
Hence
\[\limsup_{n\rightarrow +\infty} \sup_{s\geqslant \delta } \mathbb P_i \po \sigma_n(s\theta_n^2) \in\Delta^2\pf \ \leqslant \ 2\varepsilon\]
with an arbitrary $\varepsilon>0$, which concludes the proof of \eqref{eq:induction_delta} for $j=2$. 

Now, suppose by induction that \eqref{eq:induction_delta} holds for all $k\in\cco 1,j\ccf$ for some $j\in \cco 1,\mathfrak{p}-1\ccf$. This implies $\textbf{H4}$ with $\beta_n=\theta^j_n$ and the partition $\mathcal E^j_1,\dots,\mathcal E^j_{\mathfrak{n}_j},\Delta^j$. By Proposition~\ref{prop:Landim2}, we get the convergence of the time marginals of $\Psi^j(\sigma_n(\cdot \theta_n^j))$ toward those of $X^j$. In particular, as in the previous case, for any $\varepsilon>0$ we can choose $M,n_0>0$ large enough so that for all $n\geqslant n_0$,
\[  \max_{i\in\cco 1,N\ccf}\mathbb P_i\po \sigma_n(M\theta^j_n) \in \Delta^{j+1}\setminus \Delta^j\pf \ \leqslant \ 2\varepsilon\,.\]
Since, by induction, we also know that $\mathbb P_i( \sigma_n(M\theta^j_n) \in   \Delta^j)$ vanishes as $n\rightarrow +\infty$, we get that
\[  \max_{i\in\cco 1,N\ccf}\mathbb P_i\po \sigma_n(M\theta^j_n) \in \Delta^{j+1}\pf \ \leqslant \ 3\varepsilon\]
for $n$ large enough. The conclusion follows again from the Markov property, as in the case  $j=2$.

We have thus established that \textbf{H4} holds with $\beta_n=\theta^j_n$ and the partition $\mathcal E_1^j,\dots,\mathcal E_{\mathfrak{n}_j}^j$ for all $j\in\cco 1,\mathfrak{p}\ccf$. All remaining points of Theorem~\ref{Thm:Landim} then follow from Proposition~\ref{prop:Landim2} and the fact that, when applying Proposition~\ref{prop:Landim2},  a sequence $\alpha$ with $\theta^{j-1}_n \ll \alpha_n \ll \theta_n^j$,  can replace   $\theta^{j-1} $ in  condition \textbf{H2} or, alternatively, can replace $\theta^j$ in the conditions \textbf{H1}, \textbf{H3}, and \textbf{H4} (in which case, in \textbf{H1}, the limit rates $r^{\mathcal E}$ are all zero, so that the limit chain is constant). Indeed, the proof that \textbf{H3} holds in this case is given in \cite[Lemma~7.2]{Landim2016}, which only requires that \textbf{H2} holds for $\theta^{j-1}$ and that $\alpha\gg \theta^{j-1}$. Similarly, our proof of \textbf{H4} is unchanged if, at some step of the inductive construction of the timescales and the partitions, we  replace  $\theta^j$ by some $\alpha$ with $\theta^{j-1}\ll \alpha \ll \theta^{j}$ (the only difference would be that the limit chain is constant, and thus we could not conclude that $\mathfrak{n}_{j}<\mathfrak{n}_{j-1}$ to ensure that the construction ends in a finite number of steps).
\end{proof}

\section{Technical results} \label{app:B}

\subsection{Proof of Proposition~\ref{prop:PierreP1}}

The proof of Proposition~\ref{prop:PierreP1} relies on the following linear algebra result.

\begin{lemm}\label{lem:foncth}
Let $M$ be in $\mathcal M_d(\R)$ and denote by $\xi_1>\xi_2>\dots>\xi_k$ the distinct real parts of the eigenvalues of $M$. For $j\in\llbracket 1,k\rrbracket$, denote by $E_j$ the space spanned by the generalized eigenvectors of $M$ corresponding to eigenvalues with real part $\xi_j$. Then, up to a linear change of coordinates, there exists a function $h\in C^\infty(S^{d-1},(0,+\infty))$ such that 
\begin{enumerate}
\item\label{510b} $\nabla h(\theta)=0$ for $\theta\in S^{d-1}\cap (\bigcup_{i=1}^k E_i)$;
\item\label{510c} $h$ attains its minimum  at $\theta$ if and only if $\theta\in E_1\cap S^{d-1}$ and its maximum  at $\theta$ if and only if $\theta\in E_k\cap S^{d-1}$;
\item\label{510d} $\nabla h(\theta) \cdot F_M(\theta) < 0$ for every  $\theta\in S^{d-1}\setminus \bigcup_{i=1}^k E_i$, where we recall that $F_M(\theta) = M \theta - (\theta \cdot M\theta)\theta$ for $\theta\in S^{d-1}$;
\item \label{510e}
$v \cdot \nabla^2 h(\theta)v\leqslant -|\mathrm{pr}_{i-1}(v)|^2$ for $i=1,\dots,k$, $\theta\in S^{d-1}\cap E_i$, and $v\in T_\theta S^{d-1}$, where $\nabla^2 h$ denotes the Hessian of $h$ and, for $i = 2, \dotsc, k$, $\mathrm{pr}_{i-1}$ denotes the projection onto $E_1\oplus \dotsb \oplus E_{i-1}$ along $E_i \oplus \dotsb \oplus E_k$ and $\mathrm{pr}_{0}=0$.
\end{enumerate}
\end{lemm}

\begin{proof}
Denote by $n_j$ the dimension of $E_j$ for $j\in \llbracket 1,k\rrbracket$.
Let us write a vector $x\in \mathbb{R}^d$ as $x=(x_1,\dots,x_k)$ with $x_j\in \mathbb{R}^{n_j}$ identified up to a linear system of coordinates in $E_j$ to be fixed later.

By construction, $E_j=\{x\mid x_i=0\mbox{ for all }i\ne j\}$ and 
\[ M=\mathrm{diag}(M_1,\dots,M_k),\]
where each $M_{j}$ is a $n_{j}\times n_{j}$ square matrix  whose eigenvalues have all real part equal to $\xi_j$. Notice that, up to applying a further change of coordinates in each space $E_j$, we can assume that $M_j^T+M_j$ is arbitrarily close to the $n_j\times n_j$ identity matrix multiplied by $2 \xi_j$. Hence, given $\epsilon>0$, for every $x_j\in \mathbb{R}^{n_j}$ we can assume that
\begin{equation}\label{eq:epsxj}
x_j^T M_j x_j=(\xi_j+\epsilon_{j,x_j})|x_j|^2,\qquad \epsilon_{j,x_j}\in (-\epsilon,\epsilon).
\end{equation}

Set $\phi(t)=t+t^2$ for $t\in [0,1]$ and let
\[h=h_1+\dots+h_{k-1}\]
with
\[h_j(\theta)=\frac12(|\theta_1|^2+\dots+|\theta_j|^2+\phi(|\theta_{j+1}|^2+\dots+|\theta_k|^2)).\] 

Notice that $[0,1]\ni t\mapsto 1-t+\phi(t)=1+t^2$ takes its minimal value at $t=0$ and its maximal value at $t=1$. Hence, $h_j$ attains its minimum (respectively, maximum) at $\theta\in S^{d-1}$ if and only if $|\theta_{j+1}|=\dots=|\theta_k|=0$ (respectively, $|\theta_1|=\dots=|\theta_{j}|=0$). In particular, \ref{510b} and \ref{510c} are satisfied.

Let us now turn to the proof of \ref{510d}. For $\theta\in S^{d-1}$ and $j\in \llbracket1,k-1\rrbracket$, denote $t_j=|\theta_{j+1}|^2+\dots+|\theta_k|^2$.
The derivative of $h_j$ at $\theta$ in the direction  $F_M(\theta)$
is equal to 
\begin{align*}
\nabla h_j(\theta) \cdot (M\theta- (\theta \cdot M\theta) \theta)={}&\sum_{\ell=1}^j (\theta_\ell^T M_\ell \theta_\ell- |\theta_\ell|^2 \sum_{i=1}^k  \theta_i^T M_i \theta_i)\\
&+\phi'(t_j)\sum_{\ell=j+1}^k (\theta_\ell^T M_\ell \theta_\ell- |\theta_\ell|^2 \sum_{i=1}^k  \theta_i^T M_i \theta_i).
\end{align*}
For  every $\ell\in \llbracket 1,N\rrbracket$, let us rewrite the term $\theta_\ell^T M_\ell \theta_\ell- |\theta_\ell|^2 \sum_{i=1}^k  \theta_i^T M_i \theta_i$ using \eqref{eq:epsxj} as follows:
\begin{align*}
\theta_\ell^T M_\ell \theta_\ell- |\theta_\ell|^2 \sum_{i=1}^k  \theta_i^T M_i \theta_i&=
(\xi_\ell + \epsilon_{\ell,\theta_\ell})|\theta_\ell|^2- |\theta_\ell|^2 \sum_{i=1}^k (\xi_i+\epsilon_{i,\theta_i})|\theta_i|^2\\
&=(\xi_\ell + \epsilon_{\ell,\theta_\ell})|\theta_\ell|^2- |\theta_\ell|^2\bigl(\xi_\ell + \epsilon_{\ell,\theta_\ell}\\
& \hphantom{ = } +\sum_{i\ne \ell}(\xi_i - \xi_\ell + \epsilon_{i,\theta_i}-\epsilon_{\ell,\theta_\ell})|\theta_i|^2\bigr)\\
&=- |\theta_\ell|^2\sum_{i\ne \ell}(\xi_i - \xi_\ell + \epsilon_{i,\theta_i}-\epsilon_{\ell,\theta_\ell})|\theta_i|^2,
\end{align*}
where the middle equality uses $\sum_{i=1}^k|\theta_i|^2=1$. Hence,
\begin{align*}
\nabla h_j(\theta) \cdot (M\theta - (\theta \cdot M\theta) \theta)={}&
-\sum_{\ell=1}^j \sum_{i\ne \ell}(\xi_i-\xi_\ell + \epsilon_{i,\theta_i}-\epsilon_{\ell,\theta_\ell})|\theta_i|^2|\theta_\ell|^2\\
&-\phi'(t_j)\sum_{\ell=j+1}^k\sum_{i\ne \ell}(\xi_i-\xi_\ell+\epsilon_{i,\theta_i}-\epsilon_{\ell,\theta_\ell})|\theta_i|^2|\theta_\ell|^2\\
={}&\sum_{\ell=1}^j\sum_{i=j+1}^k (1-\phi'(t_j)) (\xi_\ell-\xi_i + \epsilon_{\ell,\theta_\ell}-\epsilon_{i,\theta_i})|\theta_i|^2 |\theta_\ell|^2.
\end{align*}
We have $1-\phi'(t_j)=-2t_j\leqslant 0$ for every $t_j\in [0,1]$ 
with equality holding only for $t_j=0$.
Moreover,  $\xi_\ell - \xi_i + \epsilon_{\ell,\theta_\ell}-\epsilon_{i,\theta_i}>0$ for $\epsilon$ small enough, since $\ell<i$. 
Hence $\nabla h_j(\theta) \cdot (M\theta-(\theta \cdot M\theta)\theta)<0$ if $t_j>0$ and if there exist $\ell\in \llbracket 1,j\rrbracket$ and $i\in \llbracket j+1,k\rrbracket$ such that $|\theta_i| |\theta_\ell|\ne 0$, that is, if 
$\theta\not\in V_j\cup W_j$, where
\[V_j=E_1\oplus \dots\oplus  E_j, \qquad W_j=E_{j+1}\oplus \dots\oplus  E_k.\]
This proves \ref{510d}, since if $\theta\in S^{d-1}\setminus\bigcup_{i=1}^k E_i$, then there exists $j$ such that  $\theta\not\in V_j\cup W_j$.

We are left to prove \ref{510e}. If $\theta$ is in $S^{d-1}\cap E_i$, then it is in $V_j$ for $j\geqslant i$ and in $W_j$ for $j<i$. The proof works by computing $\nabla^2h_j$ both on $S^{d-1}\cap V_j$ and $S^{d-1}\cap W_j$ for $j\in\llbracket 1,k\rrbracket$. 

In order to compute $\nabla^2h_j$, we first extend $h_j$ to a function $H_j:x\mapsto \frac12(|x_1|^2+\dots+|x_j|^2+\phi(|x_{j+1}|^2+\dots+|x_k|^2))$ on $\R^d$. Then
\[\nabla h_j(\theta)=\nabla H_j(\theta)- (\theta \cdot \nabla H_j(\theta)) \theta\]
and
\[\nabla^2 h_j(\theta)v=
\nabla^2 H_j(\theta)v- (\theta \cdot \nabla^2 H_j(\theta)v) \theta - (v \cdot \nabla H_j(\theta)) \theta - (\theta \cdot \nabla H_j(\theta)) v,\]
for $v$ in  $T_\theta S^{d-1}$,
where the latter is identified with a linear subspace of $\R^d$.

A direct computation (using that $\phi'(0)=1$ and $\phi'(1)=3$) shows that if $\theta$ is in $S^{d-1}\cap V_j$ then $\nabla^2 h_j(\theta)=0$, while if $\theta$ is in $S^{d-1}\cap W_j$ then $\nabla^2 h_j(\theta)v=-2 \mathrm{pr}_j(v)$. Hence, for $\theta\in S^{d-1}\cap E_i$ and $v\in T_\theta S^{d-1}$,
\[
v \cdot \nabla^2 h(\theta)v=-2\sum_{j=1}^{i-1}
\abs*{\mathrm{pr}_j(v)}^2 \leqslant -2 \abs*{\mathrm{pr}_{i-1}(v)}^2. \qedhere
\]
\end{proof}

We can now proceed to the proof of Proposition~\ref{prop:PierreP1}.

\begin{proof}[Proof of Proposition~\ref{prop:PierreP1}]
Fix $\varepsilon, K$ as in the statement of the proposition and let $h$ be given by Lemma~\ref{lem:foncth}. To simplify the notations in this proof, for $i \in \llbracket 1, N\rrbracket$, we denote $F_{A_i}$ simply by $F_i$.

Notice that, for $j\in\cco 1,k\ccf$ and $\theta\in S^{d-1}\cap E_j$,  we have $\na^2 h(\theta) F_M(\theta)=0$. Indeed, if $\theta$ is a  generalized eigenvector of $M$ in $E_j$ then $F_M(\theta)\in E_j$ and the conclusion follows from Item \ref{510b} in Lemma~\ref{lem:foncth}. 
As a consequence, for $j\in\cco 1,k\ccf$, $\theta\in S^{d-1}\cap E_j$, and $i\in\cco1,N\ccf$,
\[F_i(\theta)\cdot\na \po (F_i-F_M)\cdot \na h \pf(\theta) = F_i(\theta)\cdot \na^2 h(\theta) F_i(\theta)\,, \]
where we again used Item~\ref{510b} in Lemma~\ref{lem:foncth}.

Let $\delta = \varepsilon \min_{i\in \llbracket 1,N\rrbracket} \pi_i/8$. Thanks to Lemma~\ref{lem:foncth} and Condition (C), up to multiplying $h$ by a positive constant, there exist
a map $s:\cco 2,k\ccf\to \cco 1,N\ccf$ and $K_1, \dotsc, K_k$ disjoint neighborhoods in $ S^{d-1}$  of, respectively,  $S^{d-1}\cap E_1,\dots,S^{d-1}\cap E_k$ such that $K_1 \subset K$ and 
\begin{eqnarray*}
\max_{j\in\cco2,k\ccf}\max_{\theta\in K_j} F_{s(j)}(\theta)\cdot\na \po (F_{s(j)}-F_M)\cdot \na h \pf(\theta)  &\leqslant & -1,\\
\max_{j\in\cco1,k\ccf}\max_{\theta\in K_j}\max_{i\in\cco1,N\ccf} F_{i}(\theta)\cdot\na \po (F_i-F_{M})\cdot \na h \pf(\theta)  &\leqslant & \delta.
\end{eqnarray*}

For all $j\in\cco2,k\ccf$ we consider $\psi_j\in\mathcal{C}^2({S}^{d-1})$ such that $\psi_j(\theta)=1$ for all $\theta\in S^{d-1}\cap K_j$ and $0$ for all $\theta\in S^{d-1}\cap K_{\ell}$ with $\ell\neq j$.

The generator of $(\theta(t),\sigma(t))_{t\geqslant0}$ being given by
\[L g(\theta,\sigma)\  =\ F_\sigma(\theta)\cdot \na_{\theta} g(\theta,\sigma) + \mu  \po \sum_{i=1}^N \pi_i g(\theta,i)  -g(\theta,\sigma)\pf \,,\]
we consider the Lyapunov function
\[f(\theta,\sigma) = h(\theta)+\frac1\mu \po  F_{\sigma}(\theta)-F_M(\theta)\pf \cdot\na h(\theta) - \frac{1}{4\mu^2}\sum_{j=2}^k \psi_j(\theta)\1_{s(j)}(\sigma)\]
to get
\begin{align*}
L f(\theta, \sigma) = {} & F_M(\theta) \cdot \na h(\theta) +\frac1\mu F_\sigma(\theta)\cdot\na \po (F_\sigma-F_M)\cdot \na h \pf(\theta) \\
& {} - \frac{1}{4\mu^2}\sum_{j=2}^k \left[F_\sigma(\theta)\cdot \na\psi_j(\theta) \1_{s(j)}(\sigma) + \mu \psi_j(\theta) \po \pi_{s(j)}- \1_{s(j)}(\sigma)\pf \right].
\end{align*}
We distinguish four cases. First, for all $\theta\in S^{d-1}\setminus(\bigcup_{j=1}^k K_j)$ and $\sigma\in\cco1,N\ccf$, we simply bound
\[L f(\theta,\sigma) \leqslant - \alpha + C\po\frac1\mu+\frac{1}{\mu^2}\pf\,,\]
where
\[\alpha = - \sup_{\theta\in S^{d-1}\setminus(\bigcup_{j=1}^k K_j)}F_M(\theta)\cdot \na h(\theta) >0\]
and $C>0$ is some constant independent of $\mu$. Second, for all $\theta\in K_1$ and $\sigma\in\cco1,N\ccf$, 
\[L f(\theta,\sigma) \leqslant \frac1\mu F_{\sigma}(\theta) \cdot\na \po (F_{\sigma}-F_M)\cdot \na h \pf(\theta)  \leqslant \frac{ \delta}{\mu}\,.\]
Third, for all $j\in \cco 2,k\ccf$ and $\theta\in K_j$, 
\[
L f(\theta,s(j)) \leqslant \frac1\mu F_{s(j)}(\theta) \cdot\na \po (F_{s(j)}-F_M)\cdot \na h \pf(\theta) - \frac{1}{4\mu } \po \pi_{i}-1\pf 
\leqslant  -\frac{1}{4\mu}\,.\]
Fourth, for all $j\in \cco 2,k\ccf$, $\theta\in K_j$, and $\sigma\neq s(j)$,
\[
L f(\theta,\sigma) \leqslant \frac1\mu F_{\sigma}(\theta) \cdot\na \po (F_{\sigma}-F_M)\cdot \na h \pf(\theta) - \frac{1}{4\mu }  \pi_{s(j)} 
\leqslant   \frac{4\delta-\pi_{s(j)}}{4\mu} \leqslant -\frac{\min_{i\in \llbracket 1,N\rrbracket}\pi_i}{8\mu}\,.\]
Gathering these four cases, for $\mu$ large enough we get
\[Lf(\theta,\sigma) \leqslant \frac{\min_{i\in \llbracket 1,N\rrbracket}\pi_i}{8\mu} \po \varepsilon \1_{\theta\in K_1}- \1_{\theta \notin K_1} \pf \,.\]
Any invariant measure $\rho$ for $L$ thus satisfies
\[0  = \rho Lf \leqslant \frac{\min_{i\in \llbracket 1,N\rrbracket}\pi_i}{8\mu} \po \varepsilon\rho( K_1\times\cco 1,N\ccf) - 1+\rho(K_1\times\cco 1,N\ccf)\pf \,,\]
hence $\rho(K_1\times\cco 1,N\ccf) \geqslant 1/(1+\varepsilon) \geqslant 1 - \varepsilon$. This concludes since $K_1\subset K$.
\end{proof}

\subsection{Proof of Lemma~\ref{lemma:dim23}}

Assume by contradiction that there exist an open nonempty subset $B$ of $\co(A)$ and two positive integers $j\leqslant k$ in 
$ \{2,3\}$ so that for every matrix $M\in B$, using the notation of Definition~\ref{defi:cond-C}, for every $i \in   \llbracket 1,N\rrbracket$ there exists $\theta \in  E_j\cap S^{d-1}$ with $F_{A_i}(\theta) \in \bigoplus_{r\geqslant j} E_r$. Denote by $n_l$ the dimension of the generalized eigenspace $E_l$ for $l\in\llbracket1,k\rrbracket$. By eventually shrinking $B$, the integers $n_l$  do not depend on $M$ in $B$. To further simplify the discussion, we will assume that $k=d$, i.e., the $n_l$ are all equal to one, since, otherwise, up to replacing the eigenvectors by projectors on the spaces $E_l$ along the direct sum of the other $E_{l^\prime}$, the subsequent computations carry over.

Let us parameterize $B$ by an open neighborhood $S$ of $0$ in $\mathbb R^m$, where $m$ denotes the dimension of $\co(A)$. Note that one can choose the assignment $s\mapsto M(s)$ affine in $s$. It is standard that the assignments given by $s\mapsto (\lambda_l(M(s)))_{1\leqslant l\leqslant d}$ and $s\mapsto (v_l(M(s)))_{1\leqslant l\leqslant d}$ for the spectrum of $M(s)$ and a basis of its unit length eigenvectors define smooth functions on $S$. In the sequel we simply write $\lambda_l(s)$ and $v_l(s)$ and, to highlight the fact that the spaces $E_l$ depend on $M(s)$, we write them as $E_l(s)$. (Note that if $n_l > 1$ for some $l\leqslant k$, then the map assigning to every $s\in S$ the projector on $E_l(s)$ is again smooth.)

We use $D(s)$ and $V(s)$ to denote, respectively, the diagonal matrix made of the eigenvalues of $M(s)$ and the matrix in $\mathcal M_d(\mathbb R)$ with columns $v_l(s)$ and set $W(s)=\big(V(s)^T\big)^{-1}$, whose columns we denote by $w_l(s)$, $l \in \llbracket 1, d\rrbracket$. We summarize these notations with the relations
\begin{equation}\label{eq:DV-M}
M(s)V(s)=V(s)D(s),\quad W^T(s)V(s)=\id,
\quad 
\hbox{ for all }s\in S.
\end{equation}
Notice, moreover, that
\begin{equation}
\label{eq:vT-vprime}
v_{l}^T(s)v'_{l}(s)=0,\qquad \hbox{ for all }s\in S\hbox{ and } l \in   \llbracket 1,d\rrbracket,
\end{equation}
where by $'$ we denote the differentiation with respect to $s\in S$. (One can either interpret such a differentiation in tensorial sense in the computations below, or simply consider a directional derivative in the space $\R^m$ along an arbitrary direction.) Note also that $M^T(s)W(s)=W(s)D(s)$ for $s\in S$, i.e., the vectors $w_l(s)$ are eigenvectors of $M^T$. 

We only treat the case where $j=2$ since the remaining case for $j=k=d=3$ is even simpler. Then one can choose $\theta=v_2(s)$ in $E_2(s)\cap S^{d-1}$ and the above assumption on $F_{A_i}$ reads
\[F_{A_i}(v_2(s))\in \bigoplus_{r\geqslant 2}E_r(s),
\quad \hbox{for all }s\in S \text{ and }i\in\llbracket1,N\rrbracket.
\]
In turn the above equation reduces to 
\begin{equation}\label{eq:w1v2}
w_1^T(s)A_iv_2(s)= 0,\quad \hbox{for all }s\in S\hbox{ and }i \in   \llbracket 1,N\rrbracket.
\end{equation}
In the sequel, for simplicity, we drop the variable $s$ from the notations. Set $R = V^{-1} V^\prime$ and note that
\begin{equation}\label{eq:Rr}
v_l^\prime = \sum_{q=1}^d r_{ql} v_q \qquad \text{ for } l \in \llbracket 1, d\rrbracket,
\end{equation}
where $(r_{lq})_{lq} = R$. Differentiating the first equation in \eqref{eq:DV-M} with respect to $s$, replacing $V'$ by $VR$, left multiplying by $W^T$, and using that $W^T MV=D$ yields 
\begin{equation}\label{eq:WMV}
W^TM'V=[R,D]+D'.
\end{equation}
Focusing on the coefficient $(1,2)$ in the above equation and taking into account \eqref{eq:w1v2} and the fact that the eigenvalues are distinct, one deduces that 
$r_{12}\equiv 0$.

When $d=2$, using \eqref{eq:vT-vprime} and \eqref{eq:Rr} for $l = 2$ and the above, we deduce that $v_2$ is a constant vector, and hence the line supported by $v_2$ is invariant by every matrix of $A$, contradicting the irreducibility assumption.

If $d=3$, we further differentiate \eqref{eq:WMV} to deduce that
\[ [W^TM' V,R]=[R',D]+[R,D']+D'',\]
where we have used the fact that $M''\equiv 0$ and the relation $W'=-WR^T$. Plugging again \eqref{eq:WMV}, we obtain that 
\[[R,[D,R]]=[R',D]+2[R,D']+D''.\]
Again, considering only the coefficient $(1,2)$ in the above equation, we deduce that $r_{13}r_{32}\equiv 0$. If $r_{13}$ is not identically equal to zero, then there exists an open subset of $S$ where $r_{32} \equiv 0$. Using \eqref{eq:Rr} for $l = 2$ and the above, we deduce that $v_2$ is a constant vector and we have a contradiction as previously. Assume now that $r_{13} \equiv 0$. Using again that $W'=-WR^T$, we have that $w_1^\prime = - r_{11} w_1$. We deduce that the line spanned by $w_1(0)$ is invariant by the matrices $A_1^T,\dotsc,A_N^T$, contradicting again the irreducibility of $A$. That concludes the proof of the lemma.\hfill$\Box$

\bibliographystyle{abbrv}
\bibliography{Bib}

\end{document}